\documentclass[reqno,a4paper,10pt]{amsart}

\makeatletter
\@namedef{subjclassname@2020}{2020 Mathematics Subject Classification}
\makeatother

\usepackage[latin1]{inputenc}
\usepackage[english]{babel}
\usepackage{eucal,amsfonts,amssymb,amsmath,amsthm,epsfig,mathrsfs}
\usepackage{cancel,soul}
\usepackage{color}
\usepackage{amscd,amsxtra}
\usepackage{enumerate}
\usepackage{latexsym}
\usepackage{bm}
\usepackage{mathscinet}

\allowdisplaybreaks

\newcounter{ipotesi}

\Alph{ipotesi}
 \makeatletter \@addtoreset{equation}{section}

\makeatother \makeatletter

\newtheorem{thm}{Theorem}[section]
\newtheorem{hyp}[thm]{Hypotheses}{\rm}
{\rm}
\newtheorem{lemma}[thm]{Lemma}
\newtheorem{cor}[thm]{Corollary}
\newtheorem{pro}[thm]{Proposition}
\newtheorem{defn}[thm]{Definition}
\newtheorem{rmk}[thm]{Remark}{\rm}

\newcounter{parentenv}

\newcommand{\R}{{\mathbb R}}

\newcommand{\N}{{\mathbb N}}

\newcommand{\fcon}{\operatorname{\mathscr{FC}}}

\newcommand{\eps}{\varepsilon}
\newcommand{\ra}{\rightarrow}

\newcommand{\OU}{{\operatorname{\mathscr{L}}}}
\newcommand{\ol}[1]{\overline{#1}}

\renewcommand{\hat}[1]{\widehat{#1}}

\newcommand{\set}[1]{{\left\{#1\right\}}}
\newcommand{\pa}[1]{{\left(#1\right)}}
\newcommand{\sq}[1]{{\left[#1\right]}}
\newcommand{\gen}[1]{{\left\langle #1\right\rangle}}
\newcommand{\abs}[1]{{\left|#1\right|}}
\newcommand{\norm}[1]{{\left\|#1\right\|}}

\newcommand{\tc}{\, \middle |\,}

\begin{document}

\frenchspacing

\title[The domain of non-symmetric and degenerate O-U operators in Banach spaces]{Regarding the domain of non-symmetric and, possibly, degenerate Ornstein--Uhlenbeck operators in separable Banach spaces}

\author[D. Addona]{{D. Addona}}

\author[G. Cappa]{{G. Cappa}}

\author[S. Ferrari]{{S. Ferrari}$^*$}\thanks{$^*$Corresponding Author}

\address[D. Addona and G. Cappa]{Dipartimento di Scienze Matematiche, Fisiche e Informatiche, Universit\`a degli Studi di Parma, Parco Area delle Scienze 53/A, 43124 Parma, Italy.}
\email{\textcolor[rgb]{0.00,0.00,0.84}{davide.addona@unipr.it,gianluca.cappa@nemo.unipr.it}}


\address[S. Ferrari]{Dipartimento di Matematica e Fisica ``Ennio De Giorgi'', Universit\`a del Salento. POB 193, 73100 Lecce,
Italy.}
\email{\textcolor[rgb]{0.00,0.00,0.84}{simone.ferrari@unisalento.it}}

\subjclass[2020]{28C20, 35J70}

\keywords{Degenerate, domain, divergence, non-symmetric, Ornstein--Uhlenbeck, Wiener spaces.}

\date{\today}

\begin{abstract}
Let $X$ be a separable Banach space and let $X^*$ be its topological dual. Let $Q:X^*\rightarrow X$ be a linear, bounded, non-negative and symmetric operator and let $A:D(A)\subseteq X\rightarrow X$ be the infinitesimal generator of a strongly continuous semigroup of contractions on $X$. We consider the abstract Wiener space $(X,\mu_\infty,H_\infty)$ where $\mu_\infty$ is a centred non-degenerate Gaussian measure on $X$ with covariance operator defined, at least formally, as
\begin{align*}
Q_\infty=\int_0^{+\infty} e^{sA}Qe^{sA^*}ds,
\end{align*} 
and $H_\infty$ is the Cameron--Martin space associated to $\mu_\infty$. 

Let $H$ be the reproducing kernel Hilbert space associated with $Q$ with inner product $[\cdot,\cdot]_H$. We assume that the operator $Q_\infty A^*:D(A^*)\subseteq X^*\rightarrow X$ extends to a bounded linear operator $B\in \mathcal L(H)$ which satisfies $B+B^*=-{\rm Id}_H$, where ${\rm Id}_H$ denotes the identity operator on $H$. Let $D$ and $D^2$ be the first and second order Fr\'echet derivative operators, we denote by $D_H$ and $D^2_H$ the closure in $L^2(X,\mu_\infty)$ of the operators $QD$ and $QD^2$, respectively, defined on smooth cylindrical functions, and by $W^{1,2}_H(X,\mu_\infty)$ and $W^{2,2}_H(X,\mu_\infty)$, respectively, their domains in $L^2(X,\mu_\infty)$. Furthermore, we denote by $D_{A_\infty}$ the closure of the operator $Q_\infty A^*D$ in $L^2(X,\mu_\infty)$ defined on smooth cylindrical functions, and by $W^{1,2}_{A_\infty}(X,\mu_\infty)$ the domain of $D_{A_\infty}$ in $L^2(X,\mu_\infty)$. We characterize the domain of the operator $L$, associated to the bilinear form
\begin{align*}
(u,v)\mapsto-\int_{X}[BD_Hu,D_Hv]_Hd\mu_\infty, \qquad u,v\in W^{1,2}_H(X,\mu_\infty),
\end{align*}
in $L^2(X,\mu_\infty)$. More precisely, we prove that $D(L)$ coincides, up to an equivalent remorming, with a subspace of $W^{2,2}_H(X,\mu_\infty)\cap W^{1,2}_{A_\infty}(X,\mu_\infty)$. We stress that we are able to treat the case when $L$ is degenerate and non-symmetric.
\end{abstract}

\maketitle

\section{Introduction}

The Ornstein--Uhlenbeck operator is, probably, the most famous example of second-order differential elliptic operator with unbounded coefficients. In finite dimension it admits the explicit formulation, for $x\in\R^d$,
\begin{align}\label{defn_L_intro}
L\varphi(x)
= & \sum_{i,j=1}^dq_{ij}D^2_{ij}\varphi(x)-\sum_{i,j=1}^da_{ij}x_jD_i\varphi(x)= {\rm Trace}[QD^2\varphi(x)]-\langle Ax,D\varphi(x)\rangle;
\end{align}
for functions $\varphi:\R^d\rightarrow \R$ smooth enough, where $Q$ and $A$ are constant real $d\times d$ matrices and $Q$ is symmetric and non-negative. The realization of \eqref{defn_L_intro} in various functional spaces has been widely studied. For example we refer to \cite{Ad13,DaLu95,DaLu07} for the study of the realization of $L$ in some spaces of continuous functions, while we refer to \cite{BrCuLaPr10,BrCuLaPr13,ChFaMePa05,GeLu08,Lu97a,Met01,MePaPr02,MePrRhSc02} for the study of its realization in $L^p$-spaces both with respect the Lebesgue measure and with respect to the invariant measure associated to its semigroup.
For an in-depth discussion of the non-degenerate Ornstein--Uhlenbeck operators in finite dimension we refer to the monograph \cite{Lor17} and to the paper \cite{LuMePa20}.

Beside their important role in analysis and differential equations, Ornstein--Uhlenbeck operators arise in many applications, such as mathematical physics, stochastic processes and finance (see, for example, \cite{Pas,RS}). It is quite remarkable that, differently from the Laplace operator, the Ornstein--Uhlenbeck operator can be naturally extended to infinite dimension, see e.g. the seminal papers \cite{Gr67,Pi75,Um64}, the monographs \cite{Bog98,DaZa02} and the references therein. 

In this paper we assume that $X$ is a separable Banach space, $X^*$ its topological dual, $Q:X^*\rightarrow X$ is a linear, bounded, non-negative and symmetric operator and $A:D(A)\subseteq X\rightarrow X$ is the infinitesimal generator of a strongly continuous semigroup of contractions $(e^{tA})_{t\geq0}$ on $X$ (see Hypotheses \ref{ipo_1}). We consider the abstract Wiener space $(X,\mu_\infty,H_\infty)$ where $\mu_\infty$ is a centred non-degenerate Gaussian measure on $X$ with covariance operator defined, at least formally (see Hypotheses \ref{portafoglio}), as
\begin{align}
\label{intro:Qinfty}
Q_\infty=\int_0^{+\infty} e^{sA}Qe^{sA^*}ds,
\end{align} 
and $H_\infty$ is the Cameron--Martin space associated to $\mu_\infty$. We aim to characterize the domain of a non-symmetric, possibly degenerate, Ornstein--Uhlenbeck operator in $L^2(X,\mu_\infty)$.  

To introduce the Ornstein--Uhlenbeck operator which we are going to study we need to recall some basic facts about abstract Wiener spaces. The operators $Q,A,A^*$ (the adjoint of $A$) and $Q_\infty$ are linked by the well-known Lyapunov equation 
\[Q_\infty A^*+AQ_\infty=-Q,\] 
where the identity holds on $D(A^*)$ (see e.g. \cite{DaZa02} for the Hilbert space case and \cite{GV03} for the general Banach space case). Let us denote by $H$ the reproducing kernel Hilbert space associated to $Q$ and by $[\cdot,\cdot]_H$ its inner product. We assume that the operator $Q_\infty A^*:D(A^*)\subseteq X^*\rightarrow X$ extends to a bounded linear operator $B\in \mathcal L(H)$ which satisfies $B+B^*=-{\rm Id}_H$, where ${\rm Id}_H$ denotes the identity operator on $H$. Letting $p\in(1,+\infty)$ and denoting by $D$ the Fr\'echet derivative operator we introduce the spaces $W^{1,p}_H(X,\mu_\infty)$ as the domain of the closure of the gradient operator $D_H:=QD$ in $L^p(X,\mu_\infty)$ defined on smooth cylindrical functions. We will show that given the closed coercive bilinear form
\begin{align}
\label{intro_bil_form}
\mathcal E(u,v):=-\int_{X}[BD_Hu,D_Hv]_Hd\mu_\infty, \qquad u,v\in W^{1,2}_H(X,\mu_\infty),
\end{align}
it is possible to associate to $\mathcal E$ an operator $L:D(L)\subseteq L^2(X,\mu_\infty)\rightarrow L^2(X,\mu_\infty)$ (see \cite{MR92}) which is called Ornstein--Uhlenbeck type operator, and on smooth cylindrical functions $f$ (i.e. functions of the form $f(x)=\varphi(\langle x,x_1^*\rangle, \ldots,\langle x, x^*_n\rangle)$, $x\in X$, with $\varphi\in C^2_b(\R^n)$ and $x_1^*,\ldots,x_n^*\in D(A^*)$) it acts as
\begin{align}\label{intro_L_infinito}
Lf(x)
= & {\rm Tr}_H[QD^2f(x)]+\langle x,A^*Df(x)\rangle, \qquad x\in X,
\end{align}
where ${\rm Tr}_H$ denotes the trace operator on $H$.

We recall that when $Q=Q_\infty$ and $A=-{\rm Id}_X$, the $H_\infty$-gradient $D_{H_\infty}:=Q_\infty D$, plays a central role in the characterization of the domain of the realization of \eqref{intro_L_infinito} in $L^2(X,\mu_\infty)$. Indeed, the domain of $L$ in $L^2(X,\mu_\infty)$ coincides with the Sobolev space $W^{2,2}(X,\mu_\infty)$ which is the domain of the closure of the operator $(D_{H_\infty},D^2_{H_\infty})$ defined on smooth cylindrical functions. The situation completely changes when neither $A$ nor $Q$ are the identity operator. In this setting to characterize the domain of $L$ we need to introduce other Sobolev spaces. We consider the gradient operator $D_{A_\infty}:=Q_\infty A^*D$, defined on cylindrical smooth functions: for any $p\in(1,+\infty)$ this operator is closable in $L^p(X,\mu_\infty)$, and the domain of its closure will be denoted by $W^{1,p}_{A_\infty}(X,\mu_\infty)$. Further, for any $p\in(1,+\infty)$ the Sobolev space $W^{2,p}_H(X,\mu_\infty)$ is the domain of the closure of the operator $(D_H,D^2_H)$ in $L^p(X,\mu_\infty)$ defined on cylindrical smooth functions. 

The main result of the paper is the following: if denote by $\mathcal U(X,\mu_\infty)$ the completion of the space $C_b^2(X) \cap W^{2,2}_H(X,\mu_\infty)\cap W^{1,2}_{A_\infty}(X,\mu_\infty)$ with respect to the norm 
\[\|f\|_{\mathcal U(X,\mu_\infty)}^2:=\|f\|_{W^{2,2}_H(X,\mu_\infty)}^2+\|D_{A_\infty}f\|_{L^2(X,\mu_\infty;H_\infty)}^2,\]
then $D(L)$, endowed with the graph norm, coincides with $\mathcal U(X,\mu_\infty)$ up to an equivalent renormings. We separately provide the two continuous embeddings $D(L)\subseteq \mathcal U(X,\mu_\infty)$ and $\mathcal U(X,\mu_\infty)\subseteq D(L)$ (Theorems \ref{emb1} and \ref{main_thm}). 

The embedding $D(L)\subseteq \mathcal U(X,\mu_\infty)$ can be restated as two separate conditions, indeed it implies that for any $u\in D(L)$ we have $u\in W^{2,2}_H(X,\mu_\infty)$, which is a ``natural'' condition arising from  maximal regularity theorems for stationary equations, both in finite and infinite dimension, and $u\in W^{1,2}_{A_\infty}(X,\mu_\infty)$, which is a typical result of the infinite dimensional setting, see for instance, \cite{ChGo01,DaZa02,Sh92}. To prove this result we make use of finite dimensional approximations along the directions of $H_\infty$, which are a well-known and useful tool also for gradient type perturbations of Ornstein--Uhlenbeck operators, both when $X$ is a Banach space, $Q=Q_\infty$ and $A=-{\rm Id}_X$ (see \cite{AdCaFe20, AFP19, CF16, CF18}) and when $X$ is a Hilbert space, $Q={\rm Id}_X$ and $Q_\infty=-\frac12 A^{-1}$ (see \cite{DaZa02}) or $Q$ is a trace class operator and $A=-Q^{-\alpha}$ for some appropriate $\alpha$ (see \cite{BF20}). We refer to the monograph \cite{Bog98} for an in-depth analysis of finite dimensional approximations in abstract Wiener spaces. The construction of the finite dimensional approximations along $H_\infty$ implies that they ``behave well'' when one consider the $H_\infty$-gradient and the related Sobolev spaces $W^{k,p}(X,\mu_\infty)$ (see \cite[Chapter 5]{Bog98}), but to the best of our knowledge the link between finite dimensional approximations along $H_\infty$ and the Sobolev spaces $W^{k,p}_H(X,\mu_\infty)$ has never been investigated. However, this connection is crucial in order to use this kind of approximations in our case. The main results in this direction are Lemma \ref{base_Hinfty A}, which shows the existence of a ``good'' orthonormal basis $\Theta:=\{e_i\,|\,i\in\N\}$ of $H_\infty$, and Proposition \ref{prop:appr_DH_CbTheta}, which states a density result for smooth and ``cylindrical along $\Theta$'' functions in the spaces $W^{k,p}_H(X,\mu_\infty)$.

At first we prove the embedding $D(L)\subseteq \mathcal U(X,\mu_\infty)$ for non-degenerate operator $L$, i.e., when $Qx^*=0$ implies $x^*=0$. To deal with the degenerate case we approximate the operator $Q$ by means of the family of injective operators $\{Q+\varepsilon Q_\infty\,|\,\varepsilon>0\}$. For these approximations we use the results already proved for non-degenerate operators and, after letting $\varepsilon$ tend to zero, we obtain that the continuous embedding $D(L)\subseteq \mathcal U(X,\mu_\infty)$ holds true also for a possible degenerate operator $L$. This technique has been already used for more general second-order degenerate stationary operators in \cite{FaLo09,Lo05-1,Lo05-2}. It is not surprising the we obtain a maximal regularity result also for degenerate operators $L$. Indeed, we recall that in finite dimension the derivatives of a degenerate Ornstein--Uhlenbeck semigroup (and therefore of the resolvent associated) mimic the regularity of the derivatives of the non-degenerate Ornstein--Uhlenbeck semigroup if we consider those directions which do not belong to the kernel of the diffusion matrix $Q$ (see \cite{FaLu06,Lu97b}). Since $D_H$ is a gradient operator along the directions of $Q$, morally we are taking into account only those directions along which we expect a ``good behaviour'' of the derivatives. 

To prove the converse embedding $\mathcal U(X,\mu_\infty)\subseteq D(L)$ we introduce the $H$-divergence operator ${\rm div}_H$, i.e., the adjoint operator of $D_H$ in $L^2(X,\mu_\infty)$. The $H_\infty$-divergence operator, that is, the adjoint of $D_{H_\infty}$ in $L^2(X,\mu_\infty)$, has been widely studied (see for instance \cite[Section 5.8]{Bog98}), but to the best of our knowledge the investigation of the operator ${\rm div}_H$ is far from complete. Here, inspired by the results in \cite[Section 4]{AdCaFe20}, we analyze the main features of ${\rm div}_H$ to show that the continuous embedding $\mathcal U(X,\mu_\infty)\subseteq D(L)$ holds true.

We remark that analogous results, but with completely different techniques, have been obtained in \cite{MV11}, where the authors provide a characterization of the domain of a non-symmetric possibly degenerate Ornstein--Uhlenbeck operators in $L^p(X,\mu_\infty)$, $p\in(1,+\infty)$ by means of $H^\infty$-functional calculus (see also \cite{MV08}). However, in this paper the authors obtain the characterization of $D(L)$ under stronger assumptions on $A$, while we only require a condition related to the non-symmetry of $L$. Further, as already said, similar techinques to the ones which we develop here have been used in \cite{AdCaFe20,CF16,CF18,DPL14} to deal with gradient type perturbations of Ornstein--Uhlenbeck operators. Hence, we hope that our results can be extended to the case of gradient type perturbations of non-symmetric and, possibly, degenerate Ornstein--Uhlenbeck operators.

The paper is organized as follows. In Section \ref{sec:preliminaries} we provide the notations, the basic assumptions and the technical results which we will employ in the rest of the paper. In particular, we state the standard hypotheses on the operators $Q$ and $A$ which ensure that the operator $Q_\infty:X^*\rightarrow X$, introduced in \eqref{intro:Qinfty}, is well-defined and it is the covariance operator of a centred non-degenerate Gaussian measure. Further, we recall the definition of the Cameron--Martin space $H_\infty$ and we list the main properties of the abstract Wiener space $(X,\mu_\infty,H_\infty)$. We also define the reproducing kernel Hilbert space $H$ associated to the operator $Q$ and we select an orthonormal basis $\{e_n\,|\,n\in\N\}\subseteq Q_\infty X^*$ of $H_\infty$ which enjoys some useful properties (Lemma \ref{base_Hinfty A}). Finally, for any $p\in(1,+\infty)$, we introduce the Sobolev spaces $W^{1,p}(X,\mu_\infty)$, $W^{1,p}_H(X,\mu_\infty)$, $W^{2,p}_H(X,\mu_\infty)$ and $W^{1,p}_{A_\infty}(X,\mu_\infty)$, which are respectively defined as the domain of the closure of the operators $Q_\infty D, Q D$, $QD^2$ and $Q_\infty A^* D$ in $L^p(X,\mu_\infty)$.

In Section \ref{sec:OU_operator} we define the non-symmetric degenerate Ornstein--Uhlenbeck operator $L$ by means of the bilinear form \eqref{intro_bil_form}, and we collect some properties of the semigroup and of the resolvent associated to $L$.

In Section \ref{sec:inclusione_1} we provide the continuous embedding $D(L)\subseteq \mathcal U(X,\mu_\infty)$ when $L$ is non-degenerate, i.e., when $Qx^*=0$ implies $x^*=0$. To this aim we make use of the finite dimensional approximations defined by means of the basis $\{e_n\,|\,n\in\N\}$ of $H_\infty$ introduced in Section \ref{sec:preliminaries}. As said before, this method is quite natural to study the properties of the functions in $W^{k,p}(X,\mu_\infty)$, but dealing with the Sobolev spaces $W^{k,p}_H(X,\mu_\infty)$ heavily complicates the computations. One of the main difference is that in the case when either $Q={\rm Id}_X$ or $A={\rm Id}_X$, the operator $L$ applied to smooth cylindrical functions gives rise to a finite dimensional Ornstein--Uhlenbeck operator. The situation completely changes when both $Q$ and $A$ differs from the identity operator, and so in the approximation procedure we need to use a family of finite dimensional Ornstein--Uhlenbeck operators $(\mathcal L_n)_{n\in\N}$. Thanks to this family of operators, for any $\lambda>0$ and any $f\in L^2(X,\mu_\infty)$ we are able to approximate the function $u=R(\lambda,L)f$ in the norm $\norm{\cdot}_{\mathcal U(X,\mu_\infty)}$ by means of a sequence of smooth cylindrical functions $(u_n)_{n\in\N}$, and we get the continuous embedding desired.

In Section \ref{sec:degenere} we extend the results of Section \ref{sec:inclusione_1} to the case when $Q$ is a degenerate operator. To obtain this extension, we consider the family of non-degenerate operators $\{Q_\varepsilon:=Q+\varepsilon Q_\infty\,|\,\varepsilon>0\}$ and we apply the results of Section \ref{sec:inclusione_1} to the family of Ornstein--Uhlenbeck operators $\{L_\varepsilon\,|\,\varepsilon>0\}$ with ``diffusion'' operators equals to $Q_\varepsilon$. By letting $\varepsilon$ tend to zero and taking into account that the estimates for $R(L_\varepsilon,f)$ does not depends on $\varepsilon$ we get the continuous embedding $D(L)\subseteq \mathcal U(X,\mu_\infty)$ also when $Q$ is degenerate.

In Section \ref{sec:Hdiv} we introduce the $H$-divergence operator ${\rm div}_H$ as the adjoint of $D_H$ in $L^2(X,\mu_\infty)$, and we show that $\mathcal U(X,\mu_\infty)\subseteq D(L)$ with continuous embedding. We conclude this section proving Theorem \ref{main_thm}, which is the main result of this paper.

In Section \ref{sect_example} we provide an example satisfying the various hypotheses we will assume throughout the paper.

We conclude with Appendix \ref{app_A} where we state and prove some results that we believe to be widely known, but for which we have not been able to find an appropriate reference in the literature.

\vspace{0.2cm}

{\small\noindent {\bf Acknowledgements and fundings.} The authors are members of GNAMPA (Gruppo Nazionale per l'Analisi Matematica, la Probabilit\`a e le loro Applicazioni) of the Italian Istituto Nazionale di Alta Matematica
(INdAM). 

S.F. has been partially supported by the OK-INSAID project Cod. ARS01-00917. 

G.C. and S.F. have been partially supported the INdAM-GNAMPA Project 2019 ``Metodi analitici per lo studio di PDE e problemi collegati in dimensione infinita''. 

The authors have been partially supported by the INdAM-GNAMPA Project 2017 ``Equazioni e si\-ste\-mi di equazioni di Kolmogorov in dimensione finita e non'', by the INdAM-GNAMPA Project 2018 ``Equazioni e sistemi di equazioni ellittiche e paraboliche a coefficienti illimitati'' and the research project PRIN 2015233N5A ``Deterministic and stochastic evolution equations'' of the Italian Ministry of Education, MIUR.

The authors are grateful to Alessandra Lunardi for many helpful conversations.}
 
\section{Notation, preliminary results and definitions}
\label{sec:preliminaries}

Let $X$ be a separable Banach space and let $X^*$ be its topological dual. We denote by $X'$ the algebric dual of $X$ and by $\norm{\cdot}_X$ and $\norm{\cdot}_{X^*}$ the norm of $X$ and its dual norm on $X^*$, respectively. With ${}_{X}\gen{\cdot,\cdot}_{X^*}$ we denote the duality between $X$ and $X^*$, meaning that for every $x\in X$ and $x^*\in X^*$
\[{}_X\gen{x,x^*}_{X^*}:=x^*(x).\]
When there is no possibility of confusion we will simply write $\gen{\cdot,\cdot}$. If $Y$ is a Banach space we denote by $\mathcal L(X,Y)$ the space of bounded and linear operators from $X$ to $Y$ endowed with the operator norm $\norm{\cdot}_{\mathcal{L}(X,Y)}$. If $Y=X$ then we write $\mathcal L(X)$.

For any $k\in\N\cup\{\infty\}$ and $n\in\N$ we denote by $C_b^k(\R^n)$ the space of real-valued, continuous and bounded functions on $\R^n$ whose derivatives up to the order $k$ are continuous and bounded. We denote by $C_b^k(X)$ the set of real-valued and Fr\'echet differentiable functions on $X$ up to order $k$ with bounded and continuous Fr\'echet derivatives. With $\fcon^k_b(X)$ we denote the space of real-valued, continuous, cylindrical and bounded functions on $X$ whose derivatives up to order $k$ are continuous and bounded. More precisely
\begin{align*}
\fcon^k_b(X):=\set{f:X\ra\R\tc\begin{array}{c}
\text{there exist }n\in\N,\ \varphi\in C_b^k(\R^n)\text{ and }x_1^*,\ldots,x_n^*\in X^*\\
\text{such that }f(x)=\varphi(\gen{x,x_1^*},\ldots,\gen{x,x_1^*})\text{ for any }x\in X
\end{array}}.
\end{align*}

Let $Y$ be a separable Hilbert space with inner product $[\cdot,\cdot]_{Y}$ and norm $\abs{\cdot}_Y$, and let $\gamma$ be a Borel measure on $X$. For any $p\in[1,+\infty)$ we set 
\begin{align*}
\|f\|^p_{L^p(X,\gamma;Y)}:=\int_X|f(x)|_Y^p \gamma(dx),
\end{align*}
for any function $f:X\rightarrow Y$ measurable with respect to $\gamma$. We denote by $L^p(X,\gamma;Y)$ the space of the equivalence classes of Bochner integrable functions $f$ such that $\|f\|_{L^p(X,\gamma;Y)}<+\infty$. If $Y=\R$ we simply write $L^p(X,\gamma)$.

For any $y,z\in Y$ we denote by $y\otimes z:Y\times Y\rightarrow \R$ the map defined by 
\begin{align*}
(y\otimes z)(x,w)=[y,x]_{Y}[z,w]_{Y}, \quad x,w\in Y.
\end{align*}

We recall the definition of trace class operator on a real Hilbert space $E$, with inner product $[\cdot,\cdot]_E$. We say that an operator $\Phi\in\mathcal{L}(E)$ is non-negative if $[\Phi x,x]_E\geq 0$, for any $x\in E$. Given a non-negative operator $\Phi\in \mathcal L(E)$, we say that $\Phi$ is a trace class operator whenever
\begin{align}\label{engi}
\sum_{n=1}^{+\infty}[\Phi e_n,e_n]_E<+\infty,
\end{align}
where $\{e_n\,|\,n\in\N\}$ is any orthonormal basis of $E$, recalling that the series in \eqref{engi} does not depend on the choice of the orthonormal basis $\{e_n\,|\,n\in\N\}$. We define the trace of a trace class operator $\Phi$ as
\begin{align*}
{\rm Tr}_E[\Phi]:=\sum_{n=1}^{+\infty}[\Phi e_n,e_n]_E.
\end{align*}

Further, let $E$ and $F$ be two real separable Hilbert spaces with norms $\abs{\cdot}_E$ and $\abs{\cdot}_F$ and associated inner products $[\cdot,\cdot]_E$ and $[\cdot,\cdot]_F$, respectively. We say that $T:E\rightarrow F$ is a Hilbert--Schmidt operator if
\begin{align}\label{gorilla}
\|T\|_{\mathcal H_2(E;F)}^2:=\sum_{k=1}^{+\infty}|Te_k|_F^2<+\infty,
\end{align}
for any orthonormal basis $\{e_k\,|\, k\in\N\}$ of $E$ (as for \eqref{engi}, the series on the right hand side of \eqref{gorilla} does not depend on the choice of the orthonormal basis). We denote the space of Hilbert--Schmidt operators from $E$ into $F$ by $\mathcal H_2(E;F)$, and if $E=F$ we simply write $\mathcal H_2(E)$. Finally, $\mathcal H_2(E;F)$ is a Hilbert space with inner product
\begin{align}\label{drone}
[T,S]_{\mathcal H_2(E;F)}:=\sum_{k=1}^{+\infty}[Te_k,Se_k]_F,
\end{align}
where $\{e_k\,|\,k\in\N\}$ is any orthonormal basis of $E$, again, we remark that the series in the right hand side of \eqref{drone} does not depend on the choice of the orthonormal basis. Clearly, if $T$ is a Hilbert--Schmidt operator from $E$ into $F$, then $T^*T$ is a trace class operator on $E$ and
\begin{align*}
\|T\|_{\mathcal H_2(E;F)}={\rm Tr}_E[T^*T].
\end{align*}
For more information see \cite[Sections XI.6 and XI.9]{DUN-SCH2}.

\subsection{Hypotheses on $Q$ and $A$}

In this subsection we will introduce and comment the hypotheses on the operators $Q$ and $A$ we will use throughout the paper. 

\begin{hyp}
\label{ipo_1}
Let $X$ be a separable Banach space and let $X^*$ be its topological dual. We assume that the operators $Q$ and $A$ satisfy the following conditions.
\begin{enumerate}[\rm (i)]
\item $Q:X^*\rightarrow X$ is a bounded and linear operator which is symmetric and non-negative. Namely $\langle Q x^*, y^*\rangle=\langle Q y^*, x^*\rangle$, for every $x^*,y^*\in X^*$, and $\langle Q x^*, x^*\rangle\geq0$, for every $x^*\in X^*$.
 
\item\label{akko} $A:D(A)\subseteq X\rightarrow X$ is the infinitesimal generator of a strongly continuous contraction semigroup $\left(e^{tA}\right)_{t\geq0}$ on $X$.
\end{enumerate}
\end{hyp}

We need to recall the definition of reproducing kernel Hilbert space, RKHS from here on. We refer to \cite{Bog98} for a more in-depth discussion.

\begin{defn}
\label{RKHS}
Let $F:X^*\rightarrow X$ be a bounded, linear, non-negative and symmetric operator. For any $x^*,y^*\in X^*$ we let 
\[[Fx^*,Fy^*]_K:=\langle Fx^*,y^*\rangle.\] 
We remark that $[\cdot,\cdot]_K$ is an inner product on $FX^*$ and we denote by $\abs{\cdot}_K$ the associated norm. Let $K$ be the completion of $FX^*$ with respect to $\abs{\cdot}_K$. We call $K$ the reproducing kernel Hilbert space (RKHS, from here on) associated with $F$ in $X$. 
\end{defn}

\begin{rmk}\label{vaccino}
It is well known that the injection of $FX^*$ in $X$ can be extended to an injection from $K$ into $X$. We will denote by $i_F:K\longrightarrow X$ this injection. Further, if we denote by $i^*_F:X^*\rightarrow K$ (here we have identified $K^*$ with $K$ by means of the Riesz representation theorem) the adjoint operator of $i_F$, it follows that $F=i_Fi^*_F$.
\end{rmk}

By \cite[Proposition 1.2]{VN98} the map $s\mapsto e^{sA}Qe^{sA^*}$ from $(0,+\infty)$ to $\mathcal{L}(X^*,X)$ is strongly measurable, whenever Hypotheses \ref{ipo_1} hold true. We may define, for any $t>0$, the non-negative symmetric operator $Q_t\in \mathcal L(X^*,X)$ by
\begin{align}\label{operator_Qt}
Q_t:=\int_0^t e^{sA}Qe^{sA^*}ds.
\end{align}
Further, we denote by $H_t$ the RKHS associated to $Q_t$. 

\begin{hyp}\label{portafoglio}
Assume Hypotheses \ref{ipo_1} hold true. The family of operators $\{Q_t\,|\, t\geq0\}$, defined in \eqref{operator_Qt}, satisfies
\begin{enumerate}[\rm (i)]
\item\label{portafoglio_1} $Q_t$ is the covariance operator of a centred Gaussian measure $\mu_t$ on $X$ for any $t>0$;
\item\label{portafoglio_2} for any $x^*\in X^*$, the weak limit $\lim_{t\rightarrow+\infty}Q_t x^*$ exists, and we will denote it by $Q_\infty x^*$. We assume that $Q_\infty$ is the covariance operator of a centred, non-degenerate Gaussian measure $\mu_\infty$ on $X$.
\end{enumerate}
\end{hyp}

Hypotheses \ref{portafoglio} is not new, indeed it was already assumed in various papers (see, for example, \cite[Sections 2 and 6]{GV03}, \cite{MV08} and \cite{VN98}). Hypothesis \ref{portafoglio}\eqref{portafoglio_2} implies that for every $x^*\in X^*$
\begin{align*}
\varphi_{\mu_\infty}(x^*)=e^{-\frac{1}{2} \langle Q_\infty x^*,x^*\rangle},
\end{align*}
where $\varphi_{\mu_\infty}$ denotes the characteristic function of the measure $\mu_\infty$.

%
%
%
%
%
%

\subsection{The space $H_\infty$ and the operator $i_\infty$}\label{subsect_Hinfinity}
We follow \cite[Chapter 2]{Bog98} to construct the Cameron--Martin space $H_\infty$ associated to $\mu_\infty$. This construction will give us the abstract Wiener space $(X,\mu_\infty,H_\infty)$ which will be the primary setting of our studies.
By \cite[Theorem 2.8.5]{Bog98}, it follows that $X^*\subseteq L^2(X,\mu_\infty)$, and we denote by $j:X^*\rightarrow L^2(X,\mu_\infty)$ the injection of $X^*$ in $L^2(X,\mu_\infty)$, namely for every $x\in X$ and $x^*\in X^*$
\[[j(x^*)](x):=\gen{x,x^*}.\]
We remark that by \cite[Theorem 2.2.4]{Bog98} and Hypothesis \ref{portafoglio}\eqref{portafoglio_2} we have that for every $x^*_1,x^*_2\in X^*$
\begin{align}
\label{caratt_cova_inf}
\langle Q_\infty x^*_1,x^*_2\rangle=\int_X j(x^*_1)j(x^*_2)d\mu_\infty.
\end{align}
We denote by $X_{\mu_\infty}^*$ the closure of $j(X^*)$ in $L^2(X,\mu_\infty)$ and we define $R:X^*_{\mu_\infty}\rightarrow (X^*)'$ by
\begin{align}
\label{op_R}
[R(f)](x^*):=\int_Xfj(x^*)d\mu_\infty, \quad f\in X_{\mu_\infty}^*, \ x^*\in X^*.
\end{align}

\begin{lemma}
If Hypotheses \ref{portafoglio} hold true, then $R(X^*_{\mu_\infty})\subseteq X$. Meaning that for every $f\in X^*_{\mu_\infty}$ there exists $\mathcal{R}(f)\in X$ such that for every $x^*\in X^*$
\[\gen{\mathcal{R}(f),x^*}=\int_X fj(x^*)d\mu_\infty.\]
In particular the operator $\mathcal{R}:L^2(X,\mu_\infty)\ra X$ is the adjoint of $j$.
\end{lemma}

\begin{proof}
By \cite[Corollary 3.93 and Corollary 3.94]{FAB1} it is enough to prove that, for any $f\in X^*_{\mu_\infty}$, the set $(R(f))^{-1}(0)\cap B_{X^*}$ is weak-star sequentally closed. But this fact follows by the dominated convergence theorem.
\end{proof}
 
We are now able to define the Cameron--Martin space $H_\infty$ (see \cite[Section 2.2]{Bog98}).

\begin{defn}
Assume Hypotheses \ref{portafoglio} hold true. The Cameron--Martin space associated to $\mu_\infty$ is $H_\infty :=\left\{h\in X\,\middle|\,|h|_{H_\infty}<+\infty\right\}$,
where 
\[|h|_{H_\infty} :=\sup\left\{\langle h,x^*\rangle\,\middle|\,x^*\in X^*\text{ such that } \|j(x^*)\|_{L^2(X,\mu_\infty)}\leq 1\right\}.\]
\end{defn} 

By \cite[Lemma 2.4.1]{Bog98} it follows that $h\in H_\infty$ if, and only if, there exists $\hat h\in X_{\mu_\infty}^*$ such that $\mathcal{R}(\hat h)=h$. Furthermore $H_\infty$ is a Hilbert space if endowed with the inner product
\begin{align}\label{prod_scal_Hinf}
[h,k]_{H_\infty}=\int_X\hat h\hat kd\mu_\infty, \quad h,k\in H_\infty.
\end{align}
We stress that, for any $x^*\in X^*$, by \eqref{caratt_cova_inf} and \eqref{op_R} it holds that $Q_\infty x^*\in H_\infty$ and  $\mathcal{R}(j(x^*))=Q_\infty x^*$, i.e., $\widehat {Q_\infty x^*}=j(x^*)$. 
Further, by \eqref{prod_scal_Hinf}, we deduce that
\begin{align*}
\langle Q_\infty x^*_1,x^*_2\rangle=[Q_\infty x^*_1,Q_\infty x^*_2]_{H_\infty}, \quad x^*_1,x^*_2\in X^*.
\end{align*}
The following lemma will be important in the paper.

\begin{lemma}\label{caratt_cameron_martin_inf}
Assume Hypotheses \ref{portafoglio} hold true. $H_\infty$ is the RKHS associated to $Q_\infty$ in $X$. We will denote by $i_\infty$ the injection $i_{Q_\infty}: H_\infty\ra X$ introduced in Remark \ref{vaccino}. 
\end{lemma}

\begin{proof}
The proof is quite simple but we provide it for the convenience of the reader. For every $h\in H_\infty$, there exists $\hat h \in X_{\mu_\infty}^*$ such that $\mathcal{R}(\hat h)=h$. In particular, there exists $(x^*_n)_{n\in\N}\subseteq X^*$ such that the sequence $(j(x^*_n))_{n\in\N}$ converges to $\hat h$ in $L^2(X,\mu_\infty)$. We claim that the sequence $(Q_\infty x^*_n)_{n\in\N}$ converges to $h$ in $H_\infty$. Indeed, by \eqref{prod_scal_Hinf} and recalling that $\widehat {Q_\infty x^*_n}=j(x^*_n)$ for any $n\in\N$, it follows that
\begin{align*}
\lim_{n\ra+\infty}|Q_\infty x^*_n-h|_{H_\infty}^2= \lim_{n\ra+\infty}\int_X|j(x^*_n)-\hat h|^2d\mu_\infty=0.
\end{align*}
This means that $\displaystyle H_\infty\subseteq\overline{Q_\infty X^*}^{|\cdot|_{H_\infty}}$. The converse inclusion follows by analogous arguments.
\end{proof}

\begin{rmk}\label{rmk:prop_i^*_infty}
If we denote by $i_\infty^*:X^*\rightarrow H_\infty$ the adjoint operator of $i_\infty$ (here we have identified $H_\infty^*$ with $H_\infty$ by means of the Riesz representation theorem), then $Q_\infty=i_\infty  i_\infty^*$. 
For any $x^*_1,x^*_2\in X^*$, by \eqref{caratt_cova_inf} and \eqref{prod_scal_Hinf} we have
\begin{align}
\label{prop_covariance_op}
\langle i_\infty i_\infty^* x^*_1,x^*_2\rangle
= [i_\infty^*x^*_1,i_\infty^*x^*_2]_{H_\infty}
=\int_X j(x^*_1)j(x^*_2)d\mu_\infty=\langle Q_\infty x^*_1,x^*_2\rangle.
\end{align}
Furthermore, by the non-degeneracy of $Q_\infty$, it follows that  $i^*_\infty:X^*\rightarrow H_\infty$ is injective. Indeed, if $x^*\in X^*$ satisfies $i^*_\infty x^*=0$, then by \eqref{prop_covariance_op} we have
\[0=[i_\infty^* x^*,i_\infty^* x^*]_{H_\infty}=\langle i_\infty i_\infty^* x^*,x^*\rangle=\langle Q_\infty x^*,x^*\rangle.\]
So, by Hypothesis \ref{portafoglio}$(ii)$, we get that $x^*=0$.
\end{rmk}

The following lemma will be useful in various parts of this paper. We show that there exists an orthonormal basis of $H_\infty$ whose elements are image of elements of $D(A^*)$ by means of the operator $i^*_\infty$.

\begin{lemma}\label{base_Hinfty A}
Assume Hypotheses \ref{portafoglio} hold true. There exist a sequences $(f^*_n)_{n\in\N}\subseteq D(A^*)$, an orthonormal basis $\{e_n\,|\,n\in\N\}$ of $H_\infty$ and a weak-star dense subset $\{g_n^*\,|\,n\in\N\}$ of $X^*$ such that
\begin{enumerate}[\rm (i)]
\item for every $n\in\N$, it holds $e_n=i^*_\infty f_n^*$;

\item\label{linind} for every $n\in\N$, the elements $f_1^*,\ldots,f_n^*$ of $D(A^*)$ are linearly independent;


\item for every $n\in\N$, there exists $\ol{n}\in\N$ such that $g_n^*$ belongs to the linear subspace of $X^*$ generated by $\{f_1^*,\ldots,f_{\ol{n}}^*\}$.\label{Himari}
\end{enumerate}
\end{lemma}

\begin{proof}
We recall that $X^*$ is separable with respect to the weak-star topology (see \cite[Corollary 3.104]{FAB1}). We fix a weak-star dense set $\{z_n^*\,|\,n\in\N\}$ of $X^*$. It is well known (see e.g. \cite[Theorem 2.2]{K65}) that $D(A^*)$ is weak-star dense in $X^*$. By Lemma \ref{lemma_top}, for any $n\in\N$, there exists a bounded sequence $(y^*_{n,m})_{m\in\N}\subseteq D(A^*)$ in $X^*$ such that $y^*_{n,m}$ converges weakly-star to $z^*_n$ as $m$ goes to infinity.

We claim that the set $\{y^*_{n,m}\,|\,n,m\in\N\}$ is weak-star dense in $X^*$. We recall that the family $\{C_{x^*}(x_1,\ldots,x_k;\varepsilon)\,|\, k\in\N, x_1,\ldots,x_k\in X, \eps>0\}$ is a basis of neighborhoods of $x^*\in X^*$ in the weak-star topology (see \cite[Section 5]{Meg98}), where
\begin{align*}
C_{x^*}(x_1,\ldots,x_k;\varepsilon)
:=\left\{y^*\in X^*\,\middle|\,\max_{i=1,\ldots,k}|\langle x_i,y^*-x^*\rangle|< \varepsilon\right\}.
\end{align*}
Hence, to prove the claim it is enough to show that for any $x^*\in X^*$, $k\in\N$, $x_1,\ldots,x_k\in X$ and $\eps>0$  there exist $\ol{n},\ol{m}\in\N$ such that $y^*_{\ol{n},\ol{m}}$ belongs to $C_{x^*}(x_1,\ldots,x_k;\varepsilon)$. By the weak-star denseness of $\{z_n^*\,|\,n\in\N\}$ in $X^*$ there exists $\bar n\in\N$ such that 
\begin{align}\label{dispossessed}
z^*_{\bar n}\in C_{x^*}(x_1,\ldots,x_k;\varepsilon/2).
\end{align}
For any $j=1,\ldots,k$ there exists $m_j\in \N$ such that for every $m\geq m_j$
\begin{align}\label{dispossessed2}
|\langle x_j,y^*_{\bar n,m}-z^*_{\bar n}\rangle|<\frac{\varepsilon}{2}.
\end{align}
Let $\bar m:=\max\{m_1,\ldots,m_k\}$. Hence, by \eqref{dispossessed} and \eqref{dispossessed2}, we get
\begin{align*}
\max_{i=1,\ldots,k}|\langle x_i,y^*_{\bar n,\bar m}-x^*\rangle|
\leq & \max_{i=1,\ldots,k}|\langle x_i,y^*_{\bar n,\bar m}-z^*_{\bar n}\rangle|+\max_{i=1,\ldots,k}|\langle x_i,z^*_{\bar n}-x^*\rangle|<\varepsilon,
\end{align*}
which implies that $y^*_{\bar n,\bar m}\in C_{x^*}(x_1,\ldots,x_k;\varepsilon)$.

Let $(g^*_n)_{n\in\N}\subseteq D(A^*)$ be a weak-star dense sequence in $X^*$ and $x^*\in X^*$. By Lemma \ref{lemma_top} there exists a subsequence $(g^*_{n_k})_{k\in\N}$ such that $g^*_{n_k}$ converges weakly-star to $x^*$ as $k$ goes to infinity. Therefore, by the dominated convergence theorem, the Banach--Steinhaus theorem (see \cite[Theorem 3.85]{FAB1}) and \eqref{prop_covariance_op} we have
\begin{align*}
\lim_{k\ra+\infty}|i^*_\infty g^*_{n_k}-i^*_\infty x^*|_{H_\infty}^2
= & \lim_{k\ra+\infty}\int_X|j(g^*_{n_k}-x^*)|^2d\mu_\infty\\
= & \lim_{k\ra+\infty}\int_X|\langle x,g^*_{n_k}-x^*\rangle|^2\mu_\infty(dx)=0.
\end{align*}
This implies that $\{i_\infty^* g^*_n\,|\,n\in\N\}$ is dense in $i^*_\infty (X^*)$ with respect to the strong topology of $H_\infty$. The density of $i^*_\infty(X^*)$ in $H_\infty$ implies that $\{i^*_\infty g^*_n\,|\,n\in\N\}$ is dense in $H_\infty$. By Remark \ref{rmk:prop_i^*_infty} we know that $i^*_\infty g^*_n=0$ if, and only if, $g^*_n=0$, hence, up to a refinement and a renumbering, we can assume that $i^*_\infty g^*_n\neq0$ for any $n\geq 2$. 

By the Gram--Schmidt process we can construct an orthonormal basis $\{e_n\,|\,n\in\N\}$ of $H_\infty$ as follows: let $h(1):=1$ and let
\begin{align*}
e_{1} := & |i^*_{\infty}g^*_2|_{H_\infty}^{-1}i^*_\infty g^*_2, 
\end{align*}
and for any $n\geq 2$, $n\in\N$, let
\begin{align*}
h(n):=\inf\set{m\geq h(n-1)+1\tc 
\begin{array}{c}
m\in\N \text{ and } i^*_\infty g^*_m \text{ is not a}\\
\text{ linear combination of }e_{1},\ldots,e_{n-1} 
\end{array}
},
\end{align*}
and
\begin{align*}
e_{n}:=\left|i^*_\infty g^*_{h(n)}-\sum_{j=1}^{n-1}\langle i_\infty e_{j},g^*_{h(n)}\rangle e_{j}\right|_{H_\infty}^{-1}\left(i^*_\infty g^*_{h(n)}-\sum_{j=1}^{n-1}\langle i_\infty e_{j},g^*_{h(n)}\rangle e_{j}\right).
\end{align*}
The fact that the function $h:\N\ra\N$ and the set $\{e_n\,|\,n\in\N\}$ exist, follows by the density of $\{i^*_\infty g_n^*\,|\,n\in\N\}$ in $H_\infty$. We recall that for every $n\in\N$ and $j=1,\ldots,n-1$
\[\langle i_\infty e_{j},g^*_{h(n)}\rangle e_j=[i_\infty^* g^*_{h(n)}, e_{j}]_{H_\infty}e_j,\] 
is the projection of $i^*_\infty g^*_{h(n)}$ on the linear subspace of $H_\infty$ generated by $b_{j}$. We remark that, by construction, $\{e_n\,|\,n\in\N\}$ is an orthonormal basis of $H_\infty$. 

We claim that for any $n\in\N$ there exists $f_n^*\in X^*$ such that $i_\infty^* f_n^*=e_n$. Indeed, by the definition of $\{e_n\,|\,n\in\N\}$ we infer that
\begin{align}
f^*_1&:= |i^*_{\infty}g^*_2|_{H_\infty}^{-1} g^*_2, \notag \\
f^*_2&:= |i^*_\infty g^*_{h(2)}-\langle i_\infty e_1,g^*_{h(2)}\rangle e_1|_{H_\infty}^{-1}(g^*_{h(2)}-\langle i_\infty e_1,g^*_{h(2)}\rangle f^*_1), \notag \\
&\ \ \, \vdots \notag \\
f^*_n&:= \left|i^*_\infty g^*_{h(n)}-\sum_{i=1}^{n-1}\langle i_\infty e_i,g^*_{h(n)}\rangle e_i\right|_{H_\infty}^{-1}\left(g^*_{h(n)}-\sum_{j=1}^{n-1}\langle i_\infty e_j,g^*_{h(n)}\rangle f^*_j\right), \notag \\
&\ \ \, \vdots \label{def_f^*_n} 
\end{align}
and so the claim holds true.

To get \eqref{linind} let $\xi_1,\ldots,\xi_n\in \R$ and let $x^*=\sum_{s=1}^n \xi_s f_s^*$. If $x^*=0$, then, by Hypotheses \ref{ipo_1}, $i^*_\infty x^*=0$ meaning
\begin{align*}
0=|i_\infty^* x^*|_{H_\infty}^2=\sum_{i,j=1}^n\xi_i \xi_j[i^*_\infty f_i^*,i_\infty^* f_j^*]_{H_\infty}=\sum_{i=1}^n \xi_i^2.
\end{align*}
So $\xi_s=0$ for every $s=1,\ldots, n$ and the claim is proved.

To colclude the proof we just need to show that for any $m\in\N$ the element $g^*_m$ belongs to the linear span of $\{f^*_{1},\ldots,f^*_{\overline n}\}$ for some $\ol{n}\in\N$. We claim that $\overline{n}=\min\{j\in\N\,|\, h(j)\geq m\}$. We consider two cases: if $m=h(k)$ for some $k\in\N$ we have $\overline n=k$. By \eqref{def_f^*_n} and the injectivity of $i^*_\infty$ (Remark \ref{rmk:prop_i^*_infty}) we have
\begin{align*}
g_m^*
= & g_{h(k)}^*
= \left|i^*_\infty g^*_{h(k)}-\sum_{i=1}^{k-1}\langle i_\infty e_i,g^*_{h(k)}\rangle e_i\right|_{H_\infty} f^*_k+\sum_{j=1}^{k-1}\langle i_\infty e_j, g^*_{h(k)}\rangle f_j^*.
\end{align*}
So $g_m^*$ belongs to the linear space generated by $\{f_1^*,\ldots,f_k^*\}$. If $m\notin \{h(n)\,|\,n\in\N\}$, then by our definition of $\overline n$ we have that $h(\overline n)>m$ and $i^*_\infty g^*_m$ is a linear combination of $\{e_{1},\ldots,e_{\overline n}\}=\{i^*_\infty f^*_{1},\ldots,i^*_\infty f^*_{\overline n}\}$. Again the injectivity of $i^*_\infty$ allows us to conclude the prof.
\end{proof}

The basis constructed in Lemma \ref{base_Hinfty A} will be important for the finite dimensional approximation procedure we will develop in Section \ref{sec:inclusione_1}. To this aim we need to introduce the following projection operators.

\begin{defn}
Assume Hypotheses \ref{portafoglio} hold true. We fix the orthonormal basis 
\[\Theta:=\{e_n\,|\,n\in\N\}\] 
constructed in Lemma \ref{base_Hinfty A}. For any $n\in\N$ we define the projection $P_n:X\rightarrow H_\infty$ as
\begin{align}
\label{appr_di_x}
P_nx:=\sum_{i=1}^n\langle x,f^*_i\rangle e_i, \qquad x\in X.
\end{align}
Furthermore for any $k\in\N$ we denote by $\fcon_{b,\Theta}^k(X)$ the subset of $\fcon_b^k(X)$ consisting of elements of the form $f(x)=\varphi(\langle x,f_1^*\rangle, \ldots,\langle x,f^*_m\rangle)$ for any $x\in X$, for some $m\in\N$ and $\varphi\in C_b^k(\R^m)$, where the elements $f_1^*,\ldots,f_m^*$ are those introduced in Lemma \ref{base_Hinfty A}. Finally for $k\in\N\cup\set{\infty}$ and $n\in\N$ we set
\begin{align*}
\fcon_{b,\Theta}^{k,n}(X):=\set{f\in \fcon_{b,\Theta}^{k}(X)\tc \begin{array}{c}
f(x)=\varphi(\langle x,f_1^*\rangle, \ldots,\langle x,f^*_n\rangle)\\ 
x\in X \text{ and } \varphi\in C_b^k(\R^n)
\end{array}}.
\end{align*}
\end{defn}

By \cite[Proposition 2.4]{GV03} we know that $H_\infty$ is invariant for $e^{tA}$, for any $t\geq0$, and that $(e^{tA}i_\infty)_{t\geq0}$ is a strongly continuous semigroup on $H_\infty$. Let us denote by $A_\infty$ its infinitesimal generator in $H_\infty$.

\begin{lemma}\label{compleanno}
If Hypotheses \ref{portafoglio} hold true, then for any $x^*\in D(A^*)$ we have $i_\infty^*x^*\in D(A^*_\infty)$ and $A_\infty^*i_\infty^*x^*=i_\infty^*A^* x^*$.
\end{lemma}

\begin{proof}
We start by claiming that for every $h\in D(A_\infty)$ it holds that $i_\infty h\in D(A)$ and $A i_\infty h=i_\infty A_\infty h$. Indeed by the invariance of $H_\infty$ with respect to $e^{tA}$, then $e^{tA}(i_\infty h)$ belongs to $H_\infty$ so
\begin{align*}
0\leq \lim_{t\ra 0^+}\left\|\frac{e^{tA}(i_\infty h)-i_\infty h}{t}- i_\infty A_\infty h\right\|_X &= \lim_{t\ra 0^+}\left\|\frac{i_\infty\pa{(e^{tA}i_\infty) h}-i_\infty h}{t}- i_\infty A_\infty h\right\|_X\\
&=\lim_{t\ra 0^+}\left\|i_\infty \pa{\frac{(e^{tA}i_\infty) h- h}{t}- A_\infty h}\right\|_X\\
&\leq \|i_\infty\|_{\mathcal{L}(H_\infty, X)}\lim_{t\ra 0^+}\left|\frac{(e^{tA}i_\infty) h- h}{t}- A_\infty h\right|_{H_\infty}=0.
\end{align*}
Now let $x^*\in D(A^*)$. For any $h\in D(A_\infty)$ we have
\begin{align*}
[A_\infty h,i^*_\infty x^*]_{H_\infty}
= & \langle i_\infty A_\infty h,x^*\rangle
= \langle A i_\infty h,x^*\rangle
=  \langle i_\infty h,A^*x^*\rangle
= [h,i_\infty^*A^*x^*]_{H_\infty},
\end{align*}
which gives $i_\infty^*x^*\in D(A_\infty^*)$ and $A_\infty^*i_\infty^*x^*=i_\infty^*A^*x^*$.
\end{proof}

We end this subsection by introducing the following spaces, which have already been considered in \cite{MV07,MV08}.
\begin{defn}
For any $k\in\N\cup\{\infty\}, n,m\in\N$ we set
\begin{align*}
\fcon_{b}^{k,m,n}(X)
:=\set{f\in\fcon_b^k(X)\tc 
\begin{array}{c}
\text{there exist } \varphi\in C_b^k(\R^m)\text{ and}\\ 
x_1^*,\ldots, x_m^*\in D((A^*)^n)\text{ such that}\\
f(x)=\varphi\big(\langle x,x_1^*\rangle,\ldots,\langle x,x_m^*\rangle\big)\text{ for }x\in X
\end{array} },
\end{align*}
with the convention that $D((A^*)^0)=X^*$. Moreover for $k\in\N\cup\set{\infty}$ and $n\in\N$ we set
\begin{align*}
\fcon_{b}^{k,n}(X):=\bigcup_{m\in\N}\fcon_{b}^{k,m,n}(X).
\end{align*}
\end{defn}

In what follows the spaces $\fcon_b^{k,1}(X)$ will play a crucial role. We stress that the spaces $\fcon_{b}^{k,1}(X)$, $k\in\N$, are different from those considered in \cite{CF16,DPL14,GGV03}, where the authors deal with the spaces $\fcon_{b}^{k}(X)=\fcon_b^{k,0}(X)$, $k\in\N\cup\{\infty\}$. Even if the spaces $\fcon_{b,\Theta}^k(X)$ and $\fcon_{b}^{k,1}(X)$ are smaller than $\fcon_{b}^{k}(X)$, we notice that both are dense in $L^p(X,\mu_\infty)$ for any $k\in\N$, for any $p\in[1,+\infty)$. This fact follows by Lemma \ref{base_Hinfty A}, the definition of $\fcon_{b,\Theta}^k(X)$, the weak-star density of $D(A^*)$ in $X^*$ (see \cite[Theorem 2.2]{K65}) and the density of $\fcon_{b}^{k}(X)$ in $L^p(X,\mu_\infty)$ for any $p\in[1,+\infty)$ and any $k\in\N\cup\set{\infty}$ (see \cite[Corollary 3.5.2]{Bog98}).

\subsection{The Sobolev spaces $W^{k,p}(X,\mu_\infty)$}

In this subsection we will introduce the classical Sobolev spaces of the Malliavin calculus (see \cite[Chapter 5]{Bog98}). These spaces will not be the main focus of our paper, but they will be instrumental to develop the theory we are aiming for.

We start by introducing a notion of derivative weaker than the classical Fr\'echet one. Let $\{e_n\,|\,n\in\N\}$ be the basis of $H_\infty$ constructed in Lemma \ref{base_Hinfty A}, we say that
$f:X\rightarrow\R$ is $H_\infty$-differentiable at $x_0\in X$ if there exists $\ell\in H_\infty$ such that
\[
f(x_0+h)=f(x_0)+[\ell,h]_{H_\infty}+o(|h|_{H_\infty}),\qquad\text{as $|h|_{H_\infty}\rightarrow0$.}
\]
In such a case we set $D_{H_\infty} f(x_0):=\ell$ and $D_i f(x_0):= [ D_{H_\infty} f(x_0), e_i]_{H_\infty}$ for any $i\in \N$. In a similar way we say that $f$ is twice $H_\infty$-differentiable at $x_0$ if $f$  is $H_\infty$-differentiable in a neighbourhood of $x_0$ and there exists $\mathcal B\in \mathcal{H}_2(H_\infty)$ such that
\[
f(x_0+h)=f(x_0)+[D_{H_\infty} f(x_0),h]_{H_\infty}+\frac{1}{2}[\mathcal B h,h]_{H_\infty}+o(|h|^2_{H_\infty}),\qquad\text{as $|h|_{H_\infty}\rightarrow0$.}
\]
In such a case we set $D^2_{H_\infty} f(x_0):=\mathcal B$ and $D_{ij} f(x_0):=[D^2_{H_\infty} f(x_0)e_j, e_i]_{H_\infty}$ for any $i,j\in \N$. We recall that if $f$ is twice $H_\infty$-differentiable at $x_0$, then $D_{ij}f(x_0)=D_{ji}f(x_0)$ for every $i,j\in\N$.

\begin{rmk}
If a function $f:X\rightarrow\R$ is (resp. twice) Fr\'echet differentiable at $x_0$ then it is (resp. twice) ${H_\infty}$-differentiable at $x_0$ and it holds $D_{H_\infty}f(x_0)=i^*_{\infty}Df(x_0)$, (resp. $D^2_{H_\infty}f(x_0)=(i^*_\infty \otimes i^*_\infty)D^2 f(x_0)$).
\end{rmk}

The Sobolev spaces in the sense of Malliavin $W^{1,p}(X,\mu_\infty)$ and $W^{2,p}(X,\mu_\infty)$ with $p\in[1,+\infty)$, are defined as the completions of the \emph{smooth cylindrical functions} $\fcon_b^\infty(X)$ in the norms
\begin{align*}
\norm{f}_{W^{1,p}(X,\mu_\infty)}^p&:=\norm{f}^p_{L^p(X,\mu_\infty)}+\int_X\abs{D_{H_\infty} f}_{H_\infty}^pd\mu_\infty;\\
\norm{f}_{W^{2,p}(X,\mu_\infty)}^p&:=\norm{f}^p_{W^{1,p}(X,\mu_\infty)}+\int_X\|D_{H_\infty}^2 f\|^p_{\mathcal{H}_2(H_\infty)}d\mu_\infty.
\end{align*}
This is equivalent to consider the domain of the closure of the gradient operator, defined on smooth cylindrical functions, in $L^p(X,\mu_\infty)$. For more informations see \cite[Section 5.2]{Bog98}.

\subsection{The space $H$ and the operator $i$}

In the same way as in Section \ref{subsect_Hinfinity} we can define the RKHS $H$ associated to $Q$. $H$ is a Hilbert space if endowed with the scalar product $[\cdot,\cdot]_H$. As in Lemma \ref{caratt_cameron_martin_inf} we denote by $i:H\rightarrow X$ the operator $i_H$ introduced in Definition \ref{RKHS}, and we consider the adjoint operator $i^*:X^*\rightarrow H$, where again we have identified $H^*$ and $H$. We recall that that $Q=ii^*$.

The following result is an easy consequence of Lemma \ref{lemma_top} and Lemma \ref{base_Hinfty A}. We give the proof for the sake of completeness.

\begin{cor}
\label{coro:conv_deb_H}
For all $x^*\in X^*$ there exists a subsequence $(g^*_{n_k})_{k\in\N}$ of the sequence $(g_n^*)_{n\in\N}$ introduced in Lemma \ref{base_Hinfty A} such that both $(g^*_{n_k})_{k\in\N}$ converges weakly-star to $x^*$ in $X^*$, and $(i^*g_{n_k}^*)_{k\in\N}$ converges weakly to $i^*x^*$ in $H$.
\end{cor}

\begin{proof}
Let $x^*\in X^*$. By Lemma \ref{lemma_top} and Lemma \ref{base_Hinfty A} there exists a subsequence $(g^*_{n_k})_{k\in\N}$ of the sequence $(g_n^*)_{n\in\N}$ constructed in Lemma \ref{base_Hinfty A} such that both $(g^*_{n_k})_{k\in\N}$ converges weakly-star to $x^*$ in $X^*$. For any $h\in H$ it holds
\begin{align*}
\lim_{k\ra+\infty}[i^*g^*_{n_k},h]_H
= \lim_{k\ra+\infty} \langle ih, g^*_{n_k}\rangle=\langle ih,x^*\rangle 
=[i^*x^*,h]_H.
\end{align*}
This conclude the proof of the corollary.
\end{proof}

The following hypotheses will be crucial un various proofs of this paper. 

\begin{hyp} \label{ipo_RKH}
Assume Hypotheses \ref{portafoglio} hold true and that for any $x^*\in D(A^*)$ it hold that $i^*_\infty A^*x^*\in H$ and there exists $c>0$ such that for every $x^*\in D(A^*)$
\begin{align*}
|i^*_\infty A^* x^*|_H\leq c|i^*x^*|_H.
\end{align*}
\end{hyp}

\noindent We remark that, by \cite[Theorem 8.3]{GV03}, Hypotheses \ref{ipo_RKH} is equivalent to the analyticity in $L^p(X,\mu_\infty)$ of the Ornstein--Uhlenbeck semigroup $S(t)$ defined on $C_b(X)$ by
\begin{align*}
(S(t)f)(x):=\int_Xf(e^{tA}x+y)\mu_t(dy), \quad f\in C_b(X),
\end{align*}
and extended to $L^p(X,\mu_\infty)$ for any $p\in(1,+\infty)$. Here $\mu_t$ are the measures introduced in Hypothesis \ref{portafoglio}\eqref{portafoglio_1}. 

\begin{lemma}\label{propr_B}
If Hypotheses \ref{ipo_RKH} hold true, then there exists an operator $B\in\mathcal L(H)$ such that for any $x^*\in D(A^*)$
\[Bi^*x^*=i^*_\infty A^*x^*,\]  
where the identity holds in $H$. Furthermore $\|B\|_{\mathcal L(H)}\leq c$, where $c$ is the constant appearing in Hypotheses \ref{ipo_RKH}, $B+B^*=-{\rm Id}_H$, and $2[Bh,h]_H=-|h|^2_H$ for any $h\in H$.
\end{lemma}

\begin{proof}
This is a known fact, see \cite[Proposition 2.1 and Lemma 2.2]{MV07}.
\end{proof}

We now introduce an operator which is key for proving an integration by parts formula with respect to suitable directions in $H$ (see e.g. \cite[Section 3]{GGV03}).

\begin{defn}
We define the operator $V:D(V):=\{i^*_\infty x^*\,|\,x^*\in X^*\}\subseteq H_\infty\rightarrow H$ as follows:
\begin{align*}
V(i^*_\infty x^*):=i^*x^*, \qquad x^*\in X^*.
\end{align*}
\end{defn}
The existence of the adjoint operator $V^*:D(V^*)\subseteq H\rightarrow H_\infty$ follows by the fact that  $V$ is densely defined in $H_\infty$. 

\begin{rmk}\label{heart}
By \cite[Theorem 3.5]{GGV03}, the operator $(V,D(V))$ is closable in $H_\infty$ and so $D(V^*)$ is dense in $H$. Moreover, if $Q=Q_\infty$, i.e. the Malliavin setting, then $V$ is the identity operator.
\end{rmk}

\subsection{The Sobolev spaces $W^{k,p}_H(X,\mu_\infty)$}
\label{subs:sobolev_1}
We recall the definition of the gradient operator $D_H$.

\begin{defn}\label{defn_gradH}
Assume Hypotheses \ref{ipo_RKH} hold true. Let $D_H:\fcon_{b}^{1}(X)\rightarrow L^p(X,\mu_\infty;H)$ be defined as
\begin{align*}
D_Hf(x):=i^*Df(x)=\sum_{j=1}^n\frac{\partial \varphi}{\partial \xi_j}(\langle x, x^*_1\rangle,\ldots,\langle x,x^*_n\rangle) i^*x^*_j, \qquad x\in X,
\end{align*}
where $f\in \fcon_{b}^{1}(X)$ is such that $f(x)=\varphi(\langle x,x^*_1\rangle,\ldots,\langle x, x^*_n\rangle)$ for some $n\in\N$, $\varphi\in C^1_b(\R^n)$, $x_i^*\in X^*$ for $i=1,\ldots,n$, and any $x\in X$.
\end{defn}

By \cite[Theorem 8.3 and Proposition 8.7]{GV03} the operator $D_H:\fcon_b^1(X)\rightarrow L^{p}(X,\mu_\infty;H)$ is closable in $L^p(X,\mu_\infty)$ for any $p\in[1,+\infty)$ (see also \cite[Theorem 3.5]{GGV03}). We still denote by $D_H$ the closure of $D_H$ in $L^p(X,\mu_\infty)$ and by $W^{1,p}_H(X,\mu_\infty)$ the domain of the closure. The space $W^{1,p}_H(X,\mu_\infty)$ is a Banach space if endowed with the norm 
\begin{align*}
\|f\|^p_{W^{1,p}_H(X,\mu_\infty)}:=\|f\|^p_{L^p(X,\mu_\infty)}+\|D_Hf\|^p_{L^p(X,\mu_\infty;H)}, \qquad f\in W^{1,p}_H(X,\mu_\infty).
\end{align*}
Furthermore the space $W^{1,2}_H(X,\mu_\infty)$ is a Hilbert space with inner product
\begin{align*}
[ f,g]_{W^{1,2}_H(X,\mu_\infty)}
:=\int_Xfgd\mu_\infty+\int_X[D_Hf,D_Hg]_Hd\mu_\infty, \qquad f,g\in W^{1,2}_H(X,\mu_\infty).
\end{align*}

To our purpose, we introduce the second order derivatives along the directions of $H$.
\begin{defn}\label{defn_gradH2}
Assume Hypotheses \ref{ipo_RKH} hold true. For any $f\in \fcon_{b}^{2}(X)$ we define the second order derivative along $H$ as
\begin{align*}
D^2_Hf(x):=\sum_{j,k=1}^n\frac{\partial^2\varphi}{\partial \xi_j\partial \xi_k}(\langle x, x^*_1\rangle,\ldots,\langle x,x^*_n\rangle)(i^*x^*_j\otimes i^*x_k^*), \qquad x\in X,
\end{align*}
where $f(x)=\varphi(\langle x, x^*_1\rangle,\ldots,\langle x,x^*_n\rangle)$ for some $n\in\N$, $\varphi\in C_b^2(\R^n)$, $x_1^*,\ldots,x_n^*\in X^*$ and any $x\in X$.
\end{defn}

We stress that if $f\in \fcon_{b}^{2}(X)$ then, for every $x\in X$, the operator $D^2_Hf(x)$ is a trace class operator for any $x\in X$. Indeed we can, and do, assume that $\{i^* x^*_i\,|\, i=1,\ldots, n\}$ is an orthogonal family in $H$ and $f(x)=\varphi(\langle x,x_1^*\rangle, \ldots, \langle x,x^*_n\rangle)$, so 
\begin{align*}
{\rm Tr}_H[D^2_Hf(x)]
=\sum_{j,k=1}^n\langle Q x^*_j,x^*_k\rangle\frac{\partial^2 \varphi}{\partial\xi_j\partial \xi_k}(\langle x, x^*_1\rangle,\ldots,\langle x,x^*_n\rangle), \quad x\in X.
\end{align*}

We need to recall the following result, which is an integration by parts formula along the directions of $D(V^*)$.

\begin{lemma}[\cite{GGV03}, Lemma 3.3]
\label{lemm:int_parti}
Assume Hypotheses \ref{ipo_RKH} hold true. For any $f\in \fcon_{b}^{1}(X)$ and $h\in D(V^*)$ it holds
\begin{align}\label{int_parti_peso}
\int_X[D_Hf,h]_H d\mu_\infty =\int_X f\widehat{V^* h}d\mu_\infty.
\end{align}
\end{lemma}
We are now ready to define the Sobolev spaces $W^{2,p}_H(X,\mu_\infty)$ arguing as in \cite{Fer15}. 

\begin{pro}
\label{clos_D_H2_peso}
Assume Hypotheses \ref{ipo_RKH} hold true and let $p\in(1,+\infty)$. The operator 
\[(D_H,D_H^2):\fcon_b^2(X)\ra L^p(X,\mu_\infty;H)\times L^p(X,\mu_\infty;\mathcal{H}_2(H)),\] 
introduced in Definitions \ref{defn_gradH} and \ref{defn_gradH2}, is closable in $L^p(X,\mu_\infty)$. We still denote by $(D_H,D^2_H)$ its closure and by $W^{2,p}_H(X,\mu_\infty)$ the domain of its closure. Finally $W^{2,p}_H(X,\mu_\infty)$ is a Banach space, if endowed with the norm
\begin{align*}
\|f\|^p_{W^{2,p}_H(X,\mu_\infty)}:=\|f\|^p_{W^{1,p}_H(X,\mu_\infty)}+\|D_H^2f\|^p_{L^p(X,\mu_\infty;\mathcal H_2(H))}, \qquad f\in W^{2,p}_H(X,\mu_\infty).
\end{align*}
If $p=2$ the space $W^{2,2}_H(X,\mu_\infty)$ is a Hilbert space with inner product
\begin{align*}
[f,g]_{W^{2,2}_H(X,\mu_\infty)}:=[f,g]_{W^{1,2}_H(X,\mu_\infty)}+\int_X[D^2_Hf,D^2_Hg]_{\mathcal H_2(H)}d\mu_\infty, \qquad f,g\in W^{2,2}_H(X,\mu_\infty).
\end{align*}
\end{pro}
\begin{proof} 
Since $D(V^*)$ is dense in $H$ (Remark \ref{heart}), there exists an orthonormal basis $\{v_n\,|\,n\in\N\}\subseteq D(V^*)$ of $H$. To show that $(D_H,D^2_H)$ is closable, let us consider a sequence $(f_n)_{n\in\N}\subseteq \fcon_{b}^{1}(X)$ such that $(f_n)_{n\in\N}$ and $(D_Hf_n,D^2_Hf_n)_{n\in\N}$ converge to zero and $(F,G)$ in $L^p(X,\mu_\infty)$ and in $L^p(X,\mu_\infty;H)\times L^p(X,\mu_\infty;\mathcal H_2(H))$, respectively. If we show that $\mu$-a.e. $F=0$ and $G=0$, we infer the closability of $(D_H,D^2_H)$. Let $g\in \fcon_{b}^{1}(X)$, applying \eqref{int_parti_peso} to the function $f_n\cdot g\in \fcon_{b}^{1}(X)$ we have
\begin{align}
\int_X[D_H f_n,v_j]_Hg d\mu_\infty 
=& \int_X[D_H(f_n g),v_j]_H d\mu_\infty-\int_X[D_Hg,v_j]_H f_nd\mu_\infty \notag \\
= & \int_Xf_ng\widehat{V^*v_j}d\mu_\infty-\int_X[D_Hg,v_j]_H f_nd\mu_\infty, \label{int_parti_completa}
\end{align}
for any $j\in\N$. Letting $n$ tend to infinity in \eqref{int_parti_completa} we infer that
\begin{align*}
\int_X[F,v_j]_Hg d\mu_\infty 
= \lim_{n\rightarrow+\infty}\int_X[D_Hf_n,v_j]_Hg d\mu_\infty =0,
\end{align*}
for any $j\in\N$ and any $g\in\fcon_{b}^{1}(X)$. The density of $\fcon_{b}^1(X)$ in $L^p(X,\mu_\infty)$ implies that $[F(x),v_j]_H=0$ for $\mu_\infty$-a.e. $x\in X$ and any $j\in\N$. This gives that $F(x)=0$ for $\mu_\infty$-a.e. $x\in X$.

Let us consider the second order derivatives. For any $g\in\fcon_b^1(X)$ and any $i,j\in\N$ we have
\begin{align*}
\int_X[D^2_H f_nv_i,v_j]_Hg d\mu_\infty 
= & \int_X[D_H[D_Hf_n,v_i]_H,v_j]_Hg d\mu_\infty \\
= & \int_X[D_Hf_n,v_i]_Hg\widehat{V^*v_j}d\mu_\infty-\int_X[D_Hg,v_j]_H [D_Hf_n,v_j]_Hd\mu_\infty.
\end{align*}
Letting $n$ tend to infinity and recalling that 
\begin{align*}
D_Hf_n\rightarrow 0  \text{ in }L^p(X,\mu_\infty;H), \qquad D^2_Hf_n\rightarrow G \text{ in }L^p(X,\mu_\infty;\mathcal H_2(H)),
\end{align*}
it follows that
\begin{align*}
\int_X[Gv_i,v_j]_Hgd\mu_\infty=0, \qquad g\in\fcon_b^1(X),\ i,j\in\N.
\end{align*}
Arguing as for $F$ we can conclude that $G(x)=0$ for $\mu_\infty$-a.e. $x\in X$. The second part of the statement follows by standard arguments.
\end{proof}
\noindent
We want to stress that if $Q=Q_\infty$, i.e., the Malliavin setting, then $D_H$ is the Malliavin derivative and for any $p\in[1,+\infty)$ the spaces $W^{1,p}_H(X,\mu_\infty)$ and $W^{2,p}_H(X,\mu_\infty)$ are the Sobolev spaces considered in \cite[Chapter 5]{Bog98}.


%


We conclude this subsection with a density result.

\begin{pro}
\label{prop:appr_DH_CbTheta}
Assume Hypothese \ref{ipo_RKH} hold true and let $\Theta=\{e_i\,|\,i\in\N\}$ be the orthonormal basis of $H$ introduced in Lemma \ref{base_Hinfty A}. For any $k\in\N$ and any $p\in(1,+\infty)$, the space $\fcon_{b,\Theta}^k(X)$ is dense in $W^{k,p}_H(X,\mu_\infty)$.
\end{pro}

\begin{proof}
The arguments we will use are similar to the ones in \cite[Lemma 2.12]{AD20}, so we just give a sketch of the proof. We limit ourselves to prove the statement for $k=1$, being the other cases analogous. Consider $\Phi\in \fcon_b^1(X)$. Without loss of generality, we can assume that $\Phi(x)=\varphi(\langle x,x^*\rangle)$ for some $\varphi\in C_b^1(\R)$, $x^*\in X^*$ and every $x\in X$. By Corollary \ref{coro:conv_deb_H} there exists a subsequence $(g^*_{n_k})_{k\in\N}$ of the sequence $(g^*_n)_{n\in\N}\subseteq D(A^*)$, constructed in Lemma \ref{base_Hinfty A}, which is convergent weakly-star to $x^*$ in $X^*$, as $k$ goes to infinity.

For every $n\in\N$ and $x\in X$ consider
\begin{align*}
\Psi_k(x):=\varphi(\langle  x,g_{n_k}^*\rangle).
\end{align*}
Since $g_{n_k}^*$ is a linear combination of a finite number of  elements of $\{f^*_i\,|\,i\in\N\}$ (Lemma \ref{base_Hinfty A}\eqref{Himari}), it follows that $\Psi_k\in \fcon_{b,\Theta}^1(X)$ for any $n\in\N$. Further, by the dominated convergence theorem we infer that $\Psi_k$ converges to $\Phi$ in $L^p(X,\mu_\infty)$ as $n$ goes to infinity. Observe that, for every $n\in\N$ and $x\in X$, it holds
\begin{align*}
D_H\Psi_k(x)=\varphi'(\langle x,g^*_{n_k}\rangle)i^*g^*_{n_k}.
\end{align*}
Repeating the computations in the proof of \cite[Lemma 2.12]{AD20} we can find a subsequence $(\Psi_{k_m})_{m\in\N}$ of $(\Psi_k)_{k\in\N}$ such that, 
\begin{align*}
L^p(X,\mu_\infty)&\text{-}\lim_{m\ra+\infty}\Phi_{m} =\Phi,\\
L^p(X,\mu_\infty;H)&\text{-}\lim_{m\ra+\infty}D_H  \Phi_{m} = D_H\Phi;
\end{align*}
where $\Phi_m:=\frac{1}{m}\sum_{j=1}^{m}\Psi_{k_j}$, for $m\in\N$. Since, for any $m\in\N$, the function $\Phi_m$ belongs to $\fcon_{b,\Theta}^1(X)$, the proof is completed.
\end{proof}


\subsection{The Sobolev spaces $W^{1,p}_H(X,\mu_\infty;H)$}\label{sect_bfD}
Arguing as in Subsection \ref{subs:sobolev_1} we define the $H$-valued Sobolev spaces. For any $F\in \fcon_{b}^{\infty}(X;H)$ we set
\begin{align*}
\mathbf{D}_HF(x):=\sum_{j=1}^nD_HF_j(x)\otimes i^*_\infty x^*_j, \qquad x\in X,
\end{align*}
where $F(x):=\sum_{j=1}^nF_j(x)i^*_\infty x^*_j$, for some $F_j\in \fcon_b^\infty(X)$, $x^*_j\in X^*$, and $j=1,\ldots,n$.
Arguing as in Proposition \ref{clos_D_H2_peso} we get the following result.

\begin{pro}\label{apo}
Assume Hypotheses \ref{ipo_RKH} hold true and let $p\in(1,+\infty)$. The operator 
\[\mathbf{D}_H:\fcon_{b}^{\infty}(X;H)\ra L^p(X,\mu_\infty;\mathcal{H}_2(H))\] 
is closable in $L^p(X,\mu_\infty;H)$. We still denote by $\mathbf{D}_H$ the closure of $\mathbf{D}_H$ in $L^p(X,\mu_\infty;H)$ and we denote by $W^{1,p}_H(X,\mu_\infty;H)$ the domain of its closure. The space $W^{1,p}_H(X,\mu_\infty;H)$ endowed with the norm
\begin{align*}
\|f\|_{W^{1,p}_H(X,\mu_\infty;H)}^p:=\|f\|_{L^p(X,\mu_\infty;H)}^p+\|\mathbf{D}_Hf\|_{L^p(X,\mu_\infty;H;\mathcal H_2(H))}^p, \qquad f\in W^{1,p}_H(X,\mu_\infty;H),
\end{align*}
is a Banach space, and for $p=2$ it is a Hilbert space with inner product
\begin{align*}
[f,g]_{W^{1,2}_H(X,\mu_\infty;H)}
:=\int_X[f,g]_Hd\mu_\infty+\int_X[\mathbf{D}_Hf,\mathbf{D}_Hg]_{\mathcal H_2(H)}d\mu_\infty, \qquad f,g\in W^{1,2}_H(X,\mu_\infty;H).
\end{align*}
\end{pro}

The relationship between $D^2_H$ and $\mathbf{D}_H$ is explained in the next result.
\begin{lemma}
\label{lem:sobolev_vett_der-seconda}
Assume Hypotheses \ref{ipo_RKH} hold true. If $u$ belongs to $W^{2,2}_H(X,\mu_\infty)$, then $D_Hu \in W^{1,2}(X,\mu_\infty;H)$ and $\mathbf{D}_H(D_Hu)=D^2_Hu$, where the equality holds $\mu_\infty$-almost everywhere.
\end{lemma}

\begin{proof}
Let $(u_n)_{n\in\N}\subseteq \fcon_b^\infty(X)$ be a sequence convergent to $u$ in $W^{2,2}(X,\mu_\infty)$. Standard computations give that $(D_Hu_n)_{n\in\N}\subseteq \fcon_b^\infty(X;H)$,  and for every $n\in\N$ 
\[D^2_Hu_n=\mathbf{D}_HD_Hu_n.\] 
A standard argument involving \eqref{int_parti_peso} shows that $D_Hu_n$ and $D_H^2u_n$ converge to $D_Hu$ and $D_H^2u$ in $L^2(X,\mu_\infty;H)$ and $L^2(X,\mu_\infty;\mathcal H_2(H))$, respectively, as $n$ goes to infinity. Hence, we get the thesis.
\end{proof}

Observe that by Lemma \ref{lem:sobolev_vett_der-seconda}, the spaces $W^{2,p}_H(X,\mu_\infty)$, introduced in Proposition \ref{clos_D_H2_peso}, are contained in those defined in \cite{MV07}. The problem whether these two different definitions coincide is still open.

\subsection{The Sobolev spaces $W^{1,p}_{A_\infty}(X,\mu_\infty)$}\label{sect_Ainfty}

We conclude this section by defining the last type of Sobolev spaces we need. Let $D_{A_\infty}$ the gradient along the directions of $i_\infty^* A^* X^*$, namely for any $f\in \fcon_{b}^{1,1}(X)$ we set
\begin{align}
\label{def_DAinfty}
D_{A_\infty}f(x):=i_\infty^* A^*Df(x)=\sum_{j=1}^n\frac{\partial \varphi}{\partial\xi_j}(\langle x,x^*_1\rangle,\ldots,\langle x,x^*_n\rangle)i_\infty^* A^*x^*_j, \qquad x\in X;
\end{align}
where $f(x)=\varphi(\langle x,x^*_1\rangle,\ldots,\langle x,x^*_n\rangle)$ for some $n\in\N$, $\varphi\in C^1_b(\R^n)$ and $x^*_1,\ldots,x_n^*\in D(A^*)$. The following lemma is an integration by part formula similar to \eqref{int_parti_peso}.

\begin{lemma}
If Hypotheses \ref{ipo_RKH} hold true, then for any $f\in \fcon_{b}^{1,1}(X)$ and $h\in D(A_\infty)$ we have
\begin{align*}
\int_X[D_{A_\infty}f,h]_{H_\infty}d\mu_\infty
=  \int_Xf\widehat{A_\infty h}d\mu_\infty.
\end{align*}
\end{lemma}

\begin{proof}
Let $f\in \fcon_{b}^{1,1}(X)$ and let $h\in D(A_\infty)$. By Lemma \ref{compleanno} it holds
\begin{align*}
\int_X[D_{A_\infty}f,h]_{H_\infty}d\mu_\infty
= & \int_X[i_\infty^* A^*Df,h]_{H_\infty}d\mu_\infty
= \int_X[A_\infty^*i_\infty^*Df,h]_{H_\infty}d\mu_\infty\\
= &\int_X[i_\infty^*Df,A_\infty h]_{H_\infty}d\mu_\infty
=  \int_X[D_{H_\infty}f,A_\infty h]_{H_\infty}d\mu_\infty.
\end{align*}
The thesis follows applying Lemma \ref{lemm:int_parti}.
\end{proof}

Since $A_\infty$ is the infinitesimal generator of a strongly continuous and contraction semigroup on $H_\infty$ (see \cite[Proposition 2.4]{GV03}), it follows that there exists an orthonormal basis of $H_\infty$ which consists of elements of $D(A_\infty)$. After this consideration, the proof of the following proposition can be obtained by repeating verbatim the computations in Proposition \ref{clos_D_H2_peso}, hence we skip it.

\begin{pro}
Assume Hypotheses \ref{ipo_RKH} hold true and let $p\in(1,+\infty)$. The operator 
\[D_{A_\infty}:\fcon_b^{1,1}(X)\ra L^p(X,\mu_\infty;H_\infty),\] 
defined in \eqref{def_DAinfty}, is closable in $L^p(X,\mu_\infty)$. We still denote by $D_{A_\infty}$ the closure of $D_{A_\infty}$ in $L^p(X,\mu_\infty)$ and we denote by $W^{1,p}_{A_\infty}(X,\mu_\infty)$ the domain of its closure. The space $W^{1,p}_{A_\infty}(X,\mu_\infty)$ endowed with the norm
\begin{align*}
\|f\|^p_{W^{1,p}_{A_\infty}(X,\mu_\infty)}:=\|f\|^p_{L^p(X,\mu_\infty)}+\|D_{A_\infty}f\|^p_{L^p(X,\mu_\infty;H_\infty)}, \qquad f\in W^{1,p}_{A_\infty}(X,\mu_\infty),
\end{align*}
is a Banach space, and for $p=2$ it is a Hilbert space with inner product
\begin{align*}
[f,g]_{W^{1,2}_{A_\infty}(X,\mu_\infty)}
:=\int_Xfgd\mu_\infty+\int_X[D_{A_\infty}f,D_{A_\infty}g]_{H_\infty}d\mu_\infty, \qquad f,g\in W^{1,2}_{A_\infty}(X,\mu_\infty).
\end{align*}
\end{pro}

The operator $B$ allows us to make explicit the action of the operator $V^*$ on a subspace of $D(V^*)$.
\begin{lemma}
\label{domV^*}
If Hypotheses \ref{ipo_RKH} hold true, then for any $x^*\in D(A^*)$, we have $Bi^*x^*\in D(V^*)$ and 
\[V^*(Bi^*x^*)=i^*_\infty A^*x^*.\] 
As a byproduct, if $f\in \fcon_b^{m,n}(X)$ for some $m\geq 0$ and $n\geq1$, it follows that $BD_Hf(x)\in D(V^*)$ and $V^*(BD_Hf(x))=D_{A_\infty}f(x)$ for any $x\in X$.
\end{lemma} 

\begin{proof}
The first part is contained in the proof of \cite[Theorem 2.3]{MV07}, but for the sake of completeness we provide the simple proof here. Let $x^*\in D(A^*)$ and recall that by Hypotheses \ref{ipo_RKH} $i_\infty^*A^*x^*$ belongs to $H$, so $i^*(i_\infty^*A^*x^*)=i_\infty^*A^*x^*$. By Proposition \ref{propr_B}, for any $x^*\in D(A^*)$ and $y^*\in X^*$ it holds
\begin{align*}
[Bi^*x^*,V(i^*_\infty y^*)]_H
= & [Bi^*x^*,i^*y^*]_H
= [i^*_\infty A^*x^*,i^*y^*]_H\\
= & [i^*(i^*_\infty A^*x^*),i^*y^*]_H = \langle i^*_\infty A^*x^*,y^*\rangle
=[i^*_\infty A^*x^*,i^*_\infty y^*]_{H_\infty},
\end{align*}
which means that $Bi^*x^*\in D(V^*)$ and $V^*(Bi^*x^*)=i^*_\infty A^* x^*$. The second part of the statement of the lemma follows by combining the first part and the definition of $D_{A_\infty}$.
\end{proof}

\subsection{The operator $\mathbb{V}$}

In this subsection we study the operator $\mathbb V:D(\mathbb V):=\fcon_{b}(X;D(V))\subseteq L^2(X,\mu_\infty;H_\infty)\rightarrow L^2(X,\mu_\infty;H)$, where
\begin{align*}
\fcon_b(X;D(V)):=\left\{F:X\rightarrow H_\infty\, \middle| \, 
\begin{array}{c}
\text{there exists }n\in\N \text{ such that for any}\\  
j=1,\ldots,n\text{ there exist }F_j\in C_b(X)\\
\text{and } x^*_j\in X^*,\text{ such that } F=\sum_{j=1}^nF_ji^*_\infty x^*_j
\end{array}
\right\}.
\end{align*}
defined on functions $F\in \fcon_b(X;D(V))$ by
\begin{align}\label{defn_Vciccio}
(\mathbb VF)(x):=\sum_{j=1}^nF_j(x)V(i_\infty^* x^*_j)
=\sum_{j=1}^nF_j(x)i^* x^*_j=V(F(x)), \qquad x\in X.
\end{align} 

\begin{pro}
Assume Hypotheses \ref{ipo_RKH} hold true. The operator $\mathbb V$, defined in \eqref{defn_Vciccio}, is closable in $L^2(X,\mu_\infty;H_\infty)$. We still denote its closure by $(\mathbb V,D(\mathbb V))$.
\end{pro}

\begin{proof}
Let $(F_n)_{n\in\N}\subseteq D(\mathbb V)$ be such that $F_n$ converges to zero in $L^2(X,\mu_\infty;H_\infty)$, as $n$ goes to infinity, and 
\begin{align*}
\lim_{n\ra+\infty}\norm{\mathbb{V}F_n-G}_{L^2(X,\mu_\infty;H)}=0.
\end{align*}
To conclude the proof we just need to show that $G(x)=0$ for $\mu_\infty$-a.e. $x\in X$. Observe that there exists a subsequence $(F_{n_k})_{k\in\N}$ of $(F_n)_{n\in\N}$ and a Borel set $N$ of $X$ such that $\mu_\infty(N)=0$ and
\begin{align*}
\lim_{k\ra+\infty}\norm{F_{n_k}(x)}_{H_\infty}=0,\qquad \lim_{k\ra+\infty}\norm{\mathbb{V}F_{n_k}(x)-G(x)}_H=0,
\end{align*}
for every $x\in X\smallsetminus N$. By \eqref{defn_Vciccio} we know that $(\mathbb{V}F_{n_k})(x)=V(F_{n_k}(x))$ for all $x\in X\smallsetminus N$. By the closability of $V$ (see \cite[Theorem 3.5]{GGV03}) we obtain that for every $x\in X\smallsetminus N$ it holds
\begin{align*}
0=\lim_{k\ra+\infty}\norm{\mathbb{V}F_{n_k}(x)-G(x)}_H=\lim_{k\ra+\infty}\norm{V(F_{n_k}(x))-G(x)}_H=\norm{G(x)}_H.
\end{align*}
So $G(x)=0$ for $\mu_\infty$-a.e. $x\in X$.
\end{proof}

By $(\mathbb V^*, D(\mathbb V^*))$ we denote the adjoint operator of $\mathbb V$ in $L^2(X,\mu_\infty;H_\infty)$. We exploit some properties of $D(\mathbb V^*)\subseteq L^2(X,\mu_\infty;H)$ in the following proposition, and to this aim we introduce the following spaces of functions. Let $\Theta=\{e_i\,|\,i\in\N\}$ be the orthonormal basis constructed in Lemma \ref{base_Hinfty A}, and let $(f^*_j)_{j\in\N}$ be its associated sequence introduced in Lemma \ref{base_Hinfty A}. We let
\begin{align*}
\fcon_{b,\Theta}(X;D(V)):=\left\{F=\sum_{j=1}^nF_ji^*_\infty f^*_j\,\middle|\,
\begin{array}{c}
n\in\N, F_j\in C_b(X) \text{ and }\\
i^*_\infty f^*_j\in\Theta\text{ for any } j=1,\ldots,n,
\end{array}  
\right\},
\end{align*}
and
\begin{align*}
\fcon_b^{k}(X;D(V^*)):=\left\{F:X\rightarrow H\,\middle|\,
\begin{array}{c}
F=\sum_{i=1}^nf_ih_i \text{ for some }n\in\N, f_i\in \fcon_b^k(X)\\
\text{and }h_i\in D(V^*)\text{ for any }i=1,\ldots,n,
\end{array}
\right\}.
\end{align*}
If $k=0$ we set $\fcon_b^{0}(X;D(V^*)):=\fcon_b(X;D(V^*))$.

\begin{pro}
\label{pro:caratt_V_L2}
If Hypotheses \ref{ipo_RKH} are satisfied, then the following statements hold true.
\begin{enumerate}[(i)]
\item If $F$ belongs to $D(\mathbb V^*)$, then for $\mu_\infty$-a.e. $x\in X$ we have $F(x)\in D(V^*)$ and 
\begin{align}\label{Mari}
[\mathbb V^*F](x)=V^*(F(x)).
\end{align}
\item $F\in L^2(X,\mu_\infty;H)$ belongs to $D(\mathbb V^*)$ if, and only if, there exists $\Phi\in L^2(X,\mu_\infty;H_\infty)$ such that for all $G\in\fcon_{b,\Theta}(X;D(V))$
\begin{align}
\label{dom_VL2}
\int_{X}[F,\mathbb VG]_{H}d\mu_\infty
= \int_X[\Phi,G]_{H_\infty}d\mu_\infty.
\end{align}
In this case we have $\mathbb VF=\Phi$.
\item $\fcon_b(X;D(V^*))\subseteq D(\mathbb V^*)$ and \eqref{Mari} holds for all $x\in X$, whenever $F\in \fcon_b(X;D(V^*))$.
\end{enumerate}
\end{pro}

\begin{proof}
{\bf (i)} Let $F\in D(\mathbb V^*)$, $x^*\in X^*$, $f\in C_b(X)$ and let $G=fi^*_\infty g_{n}^*$, where the family $\{g^*_n\,|\,n\in\N\}$ has been introduced in Lemma \ref{base_Hinfty A}. For any $n\in\N$ we have
\begin{align*}
\int_X[\mathbb V^*F,i^*_\infty g_n^*]_{H_\infty}fd\mu_\infty
= & \int_X[\mathbb V^*F,G]_{H_\infty}d\mu_\infty
= \int_X[F,\mathbb VG]_{H}d\mu_\infty
= \int_X[F,i^*g_n^*]_Hfd\mu_\infty.
\end{align*}
The arbitrariness of $f$ gives 
\[[(\mathbb V^*F)(x),i^*_\infty g_n^*]_{H_\infty}=[F(x),i^*g_n^*]_H\qquad \text{for }\mu_\infty\text{-a.e. }x\in X.\] 
A standard argument gives that there exists a Borel set $N\subseteq X$ with $\mu_\infty(N)=0$ such that for any $x\in X\smallsetminus N$ we have $[(\mathbb V^*F)(x),i^*_\infty g_n^*]_{H_\infty}=[F(x),i^*g_n^*]_H$ for any $n\in\N$. By Lemma \ref{base_Hinfty A} and Corollary \ref{coro:conv_deb_H} it follows that for any $x\in X\smallsetminus N$ we have 
\[[(\mathbb V^*F)(x),i^*_\infty x^*]_{H_\infty}=[F(x),i^*x^*]_H=[F(x),V(i^*_\infty x^*)]_H\qquad \text{for any }x^*\in X^*.\] The density of $i_\infty^*X^*$ in $D(V)$ implies that for any $x\in X\smallsetminus N$ we get $F(x)\in D(V^*)$ and $(\mathbb V^*F)(x)=V^*(F(x))$.

{\bf (ii)} The ``only if'' part follows by the definition of adjoint operator, hence it remains to prove the other implication. Let $\Phi\in L^2(X,\mu_\infty;H_\infty)$ be such that \eqref{dom_VL2} is satisfied; to prove the statement it is enough to show that \eqref{dom_VL2} holds true for any $G\in \fcon_b(X;D(V))$. Let $G\in\fcon_b(X;D(V))$ have the form
\begin{align*}
G=\sum_{j=1}^n G_ji^*_\infty y^*_j, \quad n\in\N,\ G_j\in\fcon_b(X),\ y^*_j\in X^*, \ j=1,\ldots,n,
\end{align*}
and for any $m\in\N$ we define the function $\mathfrak G_m$ by
\begin{align*}
\mathfrak G_m:=\sum_{j=1}^nG_ji^*_\infty (g_{n_m}^*)^{(j)},
\end{align*}
where, for any $j=1,\ldots,n$, the sequence $((g_{n_m}^*)^{(j)})_{m\in\N}$ converges weakly-star to $y^*_j$ in $X^*$ as $m$ goes to infinity (Lemma \ref{base_Hinfty A}). We recall that, by the Banach--Steinhaus theorem, there exists $K>0$ such that 
\[\sup\{\|(g_{n_m}^*)^{(j)}\|_X^*\,|\,j=1,\ldots,n,\ m\in\N\}\leq K.\] 
By Lemma \ref{base_Hinfty A}\eqref{Himari} it follows that $\mathfrak G_m\in \fcon_{b,\Theta}(X;D(V))$ for any $m\in\N$. By Corollary \ref{coro:conv_deb_H} we have
\begin{align*}
\lim_{m\ra+\infty}[F(x),\mathbb V\mathfrak G_m(x)]_H
&=\lim_{m\ra+\infty}\sum_{j=1}^n[F(x),i^* (g_{n_m}^*)^{(j)}]_HG_j(x)\\
&= 
\sum_{j=1}^n[F(x),i^*y^*_j]G_j(x)
=  [F(x),\mathbb VG(x)],
\end{align*}
for $\mu_\infty$-a.e. $x\in X$. Further,
\begin{align*}
[F(x),\mathbb V\mathfrak G_m(x)]_H
\leq C|F(x)| , \quad \mu_\infty{\textup-a.e. } \ x\in X,
\end{align*}
where \(C=nK\|i^*\|_{\mathcal L(X^*;H)}\sup\{\|G_j\|_{\infty}\,|\,j=1,\ldots,n\}\).
The dominated convergence theorem implies that
\begin{align*}
\lim_{m\rightarrow+\infty}
\int_X[F,\mathbb V\mathfrak G_m]_Hd\mu_\infty=\int_X[F,\mathbb VG]_Hd\mu_\infty.
\end{align*}
Analogous computations give
\begin{align*}
\lim_{m\rightarrow+\infty}
\int_X[\Phi,\mathfrak G_m]_{H_\infty}d\mu_\infty=\int_X[\Phi,G]_{H_\infty}d\mu_\infty.
\end{align*}
Since for any $m\in\N$ it holds
\begin{align*}
\int_X[F,\mathbb V\mathfrak G_m]_Hd\mu_\infty
=  \int_X[\Phi,\mathfrak G_m]_{H_\infty}d\mu_\infty,
\end{align*}
we get the assertion.

{\bf (iii)} The statement is trivial, since for any $F\in\fcon_b(X;D(V^*))$ of the form
\begin{align*}
F=\sum_{j=1}^nf_jh_j, \quad f_j\in\fcon_b(X), \ h_j\in D(V^*), \ j=1,\ldots,n,
\end{align*}
for any $G\in \fcon_b(X;D(V))$ and $x\in X$ we get
\begin{align*}
[F(x),\mathbb VG(x)]_H
= & [F(x),V(G(x))]_H
= [V^*(F(x)),G(x)]_{H_\infty}.\qedhere
\end{align*}
\end{proof}

\section{Non-symmetric Ornstein--Uhlenbeck operators}
\label{sec:OU_operator}
In this section we introduce and recall some of the basic properties of non-symmetric Ornstein--Uhlenbeck operators. Consider the bilinear closed form defined as
\begin{align}\label{rough}
\mathcal E(u,v):=-\int_X [BD_Hu,  D_Hv]_H d\mu_\infty,  \qquad u,v\in W^{1,2}_H(X,\mu_\infty);
\end{align}
where $B$ is the operator introduced in Lemma \ref{propr_B}. By Lemma \ref{propr_B}, for every $u\in W^{1,2}_H(X,\mu_\infty)$ it holds
\begin{align}
\mathcal E(u,u)
= & -\int_X[ BD_Hu,D_Hu]_Hd\mu_\infty
=\frac 12 \int_X[D_Hu,D_Hu]_H d\mu_\infty =\frac12 \|D_Hu\|_{L^2(X,\mu_\infty;H)}^2,
\label{topolino}
\end{align}
which implies that $\mathcal E$ is non-negative. If we consider the symmetric part $\mathcal{E}_S(u,v):=\frac12(\mathcal E(u,v)+\mathcal E(v,u))$ of $\mathcal E$, with $u,v\in W^{1,2}_H(X,\mu_\infty)$, we have
\begin{align*}
\mathcal{E}_S(u,v)
=&- \frac12\int_X([BD_Hu,D_Hv]_H+[BD_Hv,D_Hu]_H)d\mu_\infty \\
= &- \frac12\int_X([BD_Hu,D_Hv]_H+[B^*D_Hu,D_Hv]_H)d\mu_\infty=\frac12\int_X[D_Hu,D_Hv]_Hd\mu_\infty,
\end{align*}
which gives that $(\mathcal{E}_S,W^{1,2}_H(X,\mu_\infty))$ is a symmetric closed form on $L^2(X,\mu_\infty)$. Finally, for any $u,v\in W^{1,2}_H(X,\mu_\infty)$, by Hypothesis \ref{ipo_RKH} we have
\begin{align*}
|\mathcal E(u,v)|
\leq & \int_X|[BD_Hu,D_Hv]_H|d\mu_\infty
\leq  \|B\|_{\mathcal L(H)}\int_X|D_Hu|_H |D_Hv|_Hd\mu_\infty \\
\leq & c\ \!\|D_Hu\|_{L^2(X,\mu_\infty;H)} \|D_Hv\|_{L^2(X,\mu_\infty;H)}
=4c\ \!\mathcal E(u,u)^{1/2}\mathcal E(v,v)^{1/2},
\end{align*}
where $c$ is the constant appearing in Hypotheses \ref{ipo_RKH}. This implies that $({\mathcal E},W^{1,2}_H(X,\mu_\infty))$ satisfies the {\it strong} (and hence the {\it weak}) {\it sector condition} (see \cite[Chapter 1, Section 2 and Exercise 2.1]{MR92}) and therefore $({\mathcal E},W^{1,2}_H(X,\mu_\infty))$ is a coercive closed form on $L^2(X,\mu_\infty)$.
According to \cite[Chapter 1]{MR92} we can define a densely defined operator $L$ as follows:
\begin{align}
\label{cerniera}
\left\{
\begin{array}{ll}
D(L) :=\Big\{u\in W^{1,2}_H(X,\mu_\infty)\Big|\ {\textrm{there exists $g\in L^2(X,\mu_\infty)$ such that }}\\
\qquad \qquad \quad \qquad \qquad \qquad \qquad \displaystyle \mathcal E(u,v)=-\int_X gvd\mu_\infty, \ \forall v\in\fcon_{b,\Theta}^{1}(X)\Big\}, \\
Lu :=g.
\end{array}
\right.
\end{align}
Our choice of $\fcon_{b,\Theta}^{1}(X)$ as the space of test functions follows by its density in $W^{1,2}_H(X,\mu_\infty)$ (Proposition \ref{prop:appr_DH_CbTheta}).

\begin{rmk}
By \cite[Sections 1 and 2 of chapter 1]{MR92}  $L$ generates a strongly continuous and contraction semigroup on $L^2(X,\mu_\infty)$ which we denote by $(T(t))_{t\geq0}$. The operator $L$ is called {\textit{Ornstein--Uhlenbeck operator in $L^2(X,\mu_\infty)$}} and the associated semigroup $(T(t))_{t\geq0}$ is called {\textit{Ornstein--Uhlenbeck semigroup in $L^2(X,\mu_\infty)$}}.
Therefore, for any $\lambda >0$ and any $f\in L^2(X,\mu_\infty)$ there exists a unique $u\in D(L)$ such that 
\begin{align*}
\lambda u-Lu=f,
\end{align*}
in $L^2(X,\mu_\infty)$. Further, multiplying both sides of $\lambda u-Lu=f$ by $u$, integrating on $X$ with respect to $\mu_\infty$ and taking advantage of \eqref{topolino} and \eqref{cerniera}, it follows that
\begin{align*}
\|u\|_{L^2(X,\mu_\infty)}\leq \frac 1\lambda\|f\|_{L^2(X,\mu_\infty)}, \quad
\|D_Hu\|_{L^2(X,\mu_\infty;H)}\leq \sqrt{\frac 2\lambda}\|f\|_{L^2(X,\mu_\infty)}.
\end{align*}
\end{rmk}

We conclude this section with a lemma analizing the action of the semigroup $T(t)$ and of the resolvent $R(\lambda,L)$ on the space $\fcon_b^{2,1}(X)$.

\begin{lemma}\label{lemma:stima_semigruppo_DH_DAinfty}
Assume Hypotheses \ref{ipo_RKH} hold true. For any $f\in \fcon^{2,1}_b(X)$, $t\geq 0$, $\lambda>0$ and $x\in X$
\begin{enumerate}[\rm (i)]
\item $T(t)f$ and $BD_HT(t)f$ belong to $\fcon^{2,1}_b(X)$ and $D(\mathbb{V}^*)$, respectively, and
\begin{align*}
\mathbb{V}^*(BD_HT(t)f)=D_{A_\infty}T(t)f;
\end{align*}

\item $R(\lambda,L)f$ belongs to $C^2_b(X)\cap W^{1,2}_{A_\infty}(X,\mu_\infty)\cap W^{2,2}_H(X,\mu_\infty)$, while $BD_HR(\lambda,L)f$ belongs to $D(\mathbb{V}^*)$ and
\begin{align}\label{iden_V*2}
\mathbb{V}^*(BD_HR(\lambda,L)f)=D_{A_\infty}R(\lambda,L)f;
\end{align}

\item there exists $C:=C(f)>0$ such that
\begin{align}
&\sup_{\substack{x\in X, t\geq 0}}\set{|BD_HT(t)f(x)|_{H}+|D_{A_\infty}T(t)f(x)|_{H_\infty}+\|D^2_HT(t)f(x)\|_{\mathcal H_2(H)}}\leq C,\label{stima_der_smgr}\\
&\sup_{x\in X}\set{|BD_HR(\lambda,L)f(x)|_H+|D_{A_\infty}R(\lambda,L)f(x)|_{H_\infty}+\|D^2_HR(\lambda,L)f(x)\|_{\mathcal H_2(H)}}\leq \frac{C}{\lambda}.\label{stima_der_risolvente}
\end{align}
\end{enumerate} 
\end{lemma}

\begin{proof}
It is enough to prove the statement for a function $f(\cdot)=\varphi(\langle \cdot,x^*\rangle)$, for some $x^*\in D(A^*)$ and $\varphi\in C^2_b(\R)$, the general case follows similarly. Arguing as in \cite[Lemma 3.1]{MV07} it follows that $T(t)f\in \fcon_b^{2,1}(X)$ for any $t\geq0$, which gives $T(t)f\in W^{1,2}_{A_\infty}(X,\mu_\infty)\cap W^{1,2}_H(X,\mu_\infty)$. Following again \cite[Lemma 3.1]{MV07} we obtain
\begin{align}
\label{formula_explcita_DHT(t)}
BD_HT(t)f(x)
= \sq{\int_X\varphi'(\langle e^{tA}x+y,x^*\rangle)\mu_t(dy)}Bi^*e^{tA^*}x^*, \qquad t\geq0, \ x\in X,
\end{align}
therefore by Lemma \ref{domV^*} it follows that $BD_HT(t)f(x)\in D(V^*)$ and 
\begin{align*}
V^*(BD_HT(t)f(x))
= & \sq{\int_X\varphi'(\langle e^{tA}x+y,x^*\rangle)\mu_t(dy)}i_\infty^*A^*e^{tA^*}x^*\\ 
= & i_\infty^*A^*DT(t)f(x) 
=  D_{A_\infty}T(t)f(x),
\end{align*}
for any $t\geq0$ and $x\in X$. Finally, by \eqref{formula_explcita_DHT(t)} and Hypothesis \ref{ipo_1}\eqref{akko}, for any $x\in X$ and $t\geq 0$ it holds
\begin{align}\label{sweat}
|BD_HT(t)f(x)|_H
\leq &  \|\varphi'\|_\infty \|B\|_{\mathcal L(H)}\|i^*\|_{\mathcal L(X^*,H)}\|x^*\|_{X^*}.
\end{align}
Further, for any $t\geq 0$, it holds $i_\infty^*A^*e^{tA^*}x^*=(e^{tA}i_\infty)^*A^*x^*$. By \cite[Proposition 2.4]{GV03} the semigroup $((e^{tA}i_\infty)^*))_{t\geq 0}$ is a contraction semigroup on $H_\infty$ (in \cite[Proposition 2.4]{GV03} the semigroup is denoted by $(\mathbf{S}_\infty^*(t))_{t\geq 0}$). It follows that for every $t\geq 0$ and $x\in X$
\begin{align}\label{soap}
|D_{A_\infty}T(t)f(x)|_{H_\infty}
\leq & \|\varphi'\|_\infty|i^*_\infty A^*x^*|_{H_\infty}.
\end{align}
Now observe that for every $t\geq 0$ and $x\in X$
\begin{align*}
D^2_HT(t)f(x)
= \sq{\int_X\varphi''(\langle e^{tA}x+y,x^*\rangle)\mu_t(dy)}(i^*e^{tA^*}x^*\otimes i^*e^{tA^*}x^*).
\end{align*}
So we get that for every $t\geq 0$ and $x\in X$
\begin{align}\label{Ase}
\|D^2_HT(t)f(x)\|_{\mathcal{H}_2(H)}&\leq \sq{\int_X|\varphi''(\langle e^{tA}x+y,x^*\rangle)|\mu_t(dy)}\|i^*e^{tA^*}x^*\|_H^2\notag\\
&\leq\|\varphi''\|_\infty\|i^*\|^2_{\mathcal{L}(X^*,H)} \|e^{tA^*}x^*\|^2_{X^*}\leq\|\varphi''\|_\infty\|i^*\|^2_{\mathcal{L}(X^*,H)} \|x^*\|^2_{X^*}.
\end{align}
So by \eqref{sweat}, \eqref{soap} and \eqref{Ase} we get \eqref{stima_der_smgr}. 

Recalling the representation formula for the resolvent operator
\begin{align*}
R(\lambda,L)f(x)=\int_0^{+\infty}e^{-\lambda t}T(t)f(x)dt,\qquad \lambda>0,x\in X;
\end{align*}
it is possible to get that $R(\lambda,L)f\in C_b^2(X)$. Further, \eqref{iden_V*2} and  \eqref{stima_der_risolvente} follow arguing as in the first part of the proof.
\end{proof}

\section{Finite dimensional approximations}
\label{sec:inclusione_1}

In this section we prove a maximal regularity result for the solution $u$ to the equation $\lambda u -Lu=f$, when $Q$ is non-degenerate. We assume that the following hypotheses hold true.
\begin{hyp}
\label{hyp:Qnon_deg}
$Q$ is non-degenerate, i.e., $Qx^*=0$ implies $x^*=0$.
\end{hyp}
Throughout this section $\Theta=\{e_i\,|\,i\in\N\}$ is the basis constructed in Lemma \ref{base_Hinfty A} and $(f^*_j)_{j\in\N}$ is its associated sequence, still introduced in Lemma \ref{base_Hinfty A}, while with $\nabla$ we denote the classical gradient operator in $\R^n$.

\subsection{Properties of $\fcon^{k,n}_{b,\Theta}(X)$}

In this subsection we collect some lemmata and propositions which will be useful in the rest of the paper. We start with a lemma showing how the quadratic form $\mathcal{E}$ defined in \eqref{rough} acts on functions belonging to $\fcon^{1,n}_{b,\Theta}(X)$ for $n\in\N$.

\begin{lemma}
\label{lemma:forma_bilineare_1}
Assume Hypotheses \ref{ipo_RKH} and \ref{hyp:Qnon_deg} hold true and let $n\in\N$. If $u,v\in \fcon^{1,n}_{b,\Theta}(X)$, then
\begin{align*}
\mathcal E(u,v)=- \int_X[i^*_\infty A^*Du,D_{H_\infty}v]_{H_\infty} d\mu_\infty.
\end{align*}
\end{lemma}

\begin{proof}
Let $u$ and $v$ be such that $u(x)=\varphi(\langle x,f_1^*\rangle, \ldots,\langle x,f_n^*\rangle)$ and $v(x)=\psi(\langle x,f_1^*\rangle, \ldots,\langle x,f_n^*\rangle)$ for some $\varphi,\psi\in C_b^1(\R^n)$. By Lemma \ref{domV^*} we get
\begin{align*}
\mathcal E(u,v)=
& -\sum_{j,k=1}^n[Bi^*f^*_j,i^*f^*_k]_H\int_X\frac{\partial \varphi}{\partial \xi_j}(\langle x,f_1^*\rangle, \ldots,\langle x,f_n^*\rangle)\frac{\partial \psi}{\partial \xi_k}(\langle x,f_1^*\rangle, \ldots,\langle x,f_n^*\rangle)\mu_\infty(dx) \\
= &- \sum_{j,k=1}^n[i^*_\infty A^*f^*_j,i_\infty^*f^*_k]_{H_\infty}\int_X\frac{\partial \varphi}{\partial \xi_j}(\langle x,f_1^*\rangle, \ldots,\langle x,f_n^*\rangle)\frac{\partial \psi}{\partial \xi_k}(\langle x,f_1^*\rangle, \ldots,\langle x,f_n^*\rangle)\mu_\infty(dx) \\
= &- \int_X[i^*_\infty A^*Du,D_{H_\infty}v]_{H_\infty} d\mu_\infty.\qedhere
\end{align*}
\end{proof}

The next lemma functions both as an introduction of two matrixes which will be connected to the second derivative operator $D_H^2$ and to show some of their basic properties. We recall that $P_n$ has been defined in \eqref{appr_di_x}.

\begin{lemma}\label{monstress}
Assume Hypotheses \ref{ipo_RKH} and \ref{hyp:Qnon_deg} hold true. For every $n\in\N$ let $\mathscr{Q}_n=(q_{kj})_{k,j=1}^n$ and $\mathscr{B}_n=(b_{kj})_{k,j=1}^n$ be two $n\times n$-matrices whose entries are
\begin{align}
q_{kj}:=\langle Qf_j^*,f_k^*\rangle,\quad b_{kj}:=\langle Q_\infty A^*f_j^*,f_k^*\rangle,\qquad k,j=1,\ldots,n.
\label{def_Qn_bn_Mn}
\end{align}
For any $n\in\N$, it holds that $2q_{kj}=-(b_{kj}+b_{jk})$, for any $k,j=1,\ldots,n$. Moreover, $\mathscr Q_n$ is positive definite and $\mathscr{B}_n$ is negative definite.
\end{lemma}

\begin{proof}
By Lemma \ref{propr_B} and Lemma \ref{domV^*} we have for every $k,j=1,\ldots,n$
\begin{align*}
2q_{kj}&=2\langle i^*f_j^*,f_k^*\rangle=2[i^*f_j^*,i^*f_k^*]_H=-[(B+B^*)i^*f_j^*,i^*f_k^*]_H\\
&=-[Bi^*f_j^*,i^*f_k^*]_H-[B^*i^*f_j^*,i^*f_k^*]_H=-[Bi^*f_j^*,i^*f_k^*]_H-[Bi^*f_k^*,i^*f_j^*]_H\\
&=-\langle Bi^*f_j^*,f_k^*\rangle-\langle Bi^*f_k^*,f_j^*\rangle=-\langle i^*_\infty A^*f_j^*,f_k^*\rangle-\langle i^*_\infty A^*f_k^*,f_j^*\rangle\\
&=-(b_{kj}+b_{jk}).
\end{align*}
To prove that $\mathscr Q_n$ is positive definite, we notice that, for any $\xi\in \R^n\smallsetminus \{0\}$ with $\xi=(\xi_1,\ldots,\xi_n)$, by Lemma \ref{base_Hinfty A}\eqref{linind} and Hypothesis \ref{hyp:Qnon_deg} we get
\begin{align*}
\langle \mathscr Q_n\xi,\xi\rangle_{\R^n}
= & \sum_{j,k=1}^n q_{kj}\xi_k\xi_j 
= \sum_{j,k=1}^n[i^*f_j^*,i^*f_k]_H\xi_j\xi_k
=[i^*\Xi,i^*\Xi]_H=\langle Q\Xi,\Xi \rangle>0,
\end{align*}
where $\Xi=\sum_{j=1}^n\xi_j f_j^*$. To show that $\mathscr{B}_n$ is negative definite we use analogous arguments. Let $\xi\in \R^n\smallsetminus\set{0}$ with $\xi=(\xi_1,\ldots,\xi_n)$, then
\begin{align*}
\langle \mathscr{B}_n\xi,\xi\rangle_{\R^n}  &=\sum_{j,k=1}^nb_{kj}\xi_k\xi_j =\frac12\sum_{j,k=1}^n(b_{kj}+b_{jk})\xi_j\xi_k=-\sum_{j,k=1}^nq_{kj}\xi_j\xi_k=-\sum_{j,k=1}^n[i^*f_j^*,i^*f_k]_H\xi_j\xi_k
\\
&=-[i^*\Xi,i^*\Xi]_H=-\langle Q\Xi,\Xi \rangle<0,
\end{align*}
where $\Xi=\sum_{j=1}^n\xi_j f_j^*$. 
\end{proof}

The following proposition regarding the derivatives of the functions in $\fcon^{k,n}_{b,\Theta}(X)$ has a straightforward proof which we leave to the reader.

\begin{pro}\label{prop_fcon1}
Assume Hypotheses \ref{ipo_RKH} and \ref{hyp:Qnon_deg} hold true and let $n\in\N$ and $i=1,2$. Let $g_i\in \fcon^{i,n}_{b,\Theta}(X)$ be such that $g_i(x)=\varphi_i(\langle x,f_1^*\rangle,\ldots, \langle x,f_n^*\rangle)$ for some $\varphi_i\in C^i_b(\R^n)$, for $i=1,2$. If $x=P_nx$ then
\begin{align*}
\langle x, A^*Dg_1(x)\rangle
= & \sum_{k,j=1}^n b_{kj}\langle x,f^*_k\rangle\frac{\partial \varphi_1}{\partial \xi_j}(\langle x,f_1^*\rangle, \ldots,\langle x,f_n^*\rangle)\\
= & \langle (\langle x,f_1^*\rangle, \ldots,\langle x,f_n^*\rangle),\mathscr{B}_n\nabla\varphi_1(\langle x,f_1^*\rangle, \ldots,\langle x,f_n^*\rangle)\rangle_{\R^n};\\
{\rm Tr}_H[D_H^2g_2(x)]= &{\rm Tr}_{\R^n}[\mathscr{Q}_n\nabla\varphi_2 (\langle x,f_1^*\rangle, \ldots,\langle x,f_n^*\rangle)]
=  \sum_{j,k=1}^nq_{kj}\frac{\partial^2\varphi_2}{\partial \xi_j\partial \xi_k}(\langle x,f_1^*\rangle, \ldots,\langle x,f_n^*\rangle).
\end{align*}
Furthermore, for every $j=1,\ldots,n$, it holds
\begin{align*}
[i^*_\infty f^*_j,i^*_\infty Dg_1(x) ]_{H_\infty}&=\langle Q_\infty f^*_j,Dg_1(x)\rangle=\frac{\partial\varphi}{\partial_1\xi_j}(\langle x,f_1^*\rangle, \ldots,\langle x,f_n^*\rangle);\\
[i^*_\infty f^*_j,i^*_\infty D^2g_2(x) i^*_\infty f^*_k]_{H_\infty}&= \langle Q_\infty f^*_j,D^2g_2(x)Q_\infty f^*_k\rangle=\frac{\partial^2 \varphi_2}{\partial \xi_j\partial \xi_k}(\langle x,f_1^*\rangle, \ldots,\langle x,f_n^*\rangle).
\end{align*}
\end{pro}

The next lemma show how the quadratic form $\mathcal{E}$, introduced in \eqref{rough}, behaves when applied to functions belonging to $\fcon_{b,\Theta}^{2,n}(X)$.
\begin{lemma}
Assume Hypotheses \ref{ipo_RKH} and \ref{hyp:Qnon_deg} hold true. For any $u,v\in\fcon_{b,\Theta}^{2,n}(X)$ it holds
\begin{align}
\label{rappr_forma_dir_n}
\mathcal E(u,v)
=-\int_X\left(\frac12{\rm Tr}_H[D_H^2u(x)]+\langle P_nx,A^*Du(x)\rangle \right)v(x)\mu_\infty(dx).
\end{align}
\end{lemma}

\begin{proof}
Observe that $D_{H_\infty}v=P_nD_{H_\infty}v$ since $v\in\fcon_{b,\Theta}^{2,n}(X)$. By Lemma \ref{lemma:forma_bilineare_1} it holds
\begin{align*}
\mathcal E(u,v)= & -\int_X[i^*_\infty A^*Du,D_{H_\infty}v]_{H_\infty} d\mu_\infty
= -\int_X[P_ni^*_\infty A^*Du,D_{H_\infty}v]_{H_\infty} d\mu_\infty \\
= & -\sum_{j,k=1}^n\langle Q_\infty A^*f^*_j,f^*_k\rangle \int_X[e_k,D_{H_\infty}v(x)]_{H_\infty}\frac{\partial \varphi}{\partial\xi_j}(\langle x,f^*_1\rangle,\ldots,\langle x,f_n^*\rangle)\mu_\infty(dx),
\end{align*}
where $u(x)=\varphi(\langle x,f^*_1\rangle,\ldots,\langle x,f_n^*\rangle)$ for any $x\in X$, and $\varphi\in C_b^2(\R^n)$. Integrating by parts (see \cite[Theorem 5.1.8]{Bog98}) we get
\begin{align}
\mathcal E(u,v)
= & -\sum_{j,k=1}^n \langle Q_\infty A^*f^*_j,f^*_k\rangle\int_X\Big[\langle x,f^*_k\rangle\frac{\partial \varphi}{\partial\xi_j}(\langle x,f^*_1\rangle,\ldots,\langle x,f_n^*\rangle) \notag\\
& -\frac{\partial^2 \varphi}{\partial \xi_k \partial\xi_j}(\langle x,f^*_1\rangle,\ldots,\langle x,f_n^*\rangle) \Big]v(x)\mu_\infty(dx).
\label{forma_espl_E_n}
\end{align}
By the definition of $P_n$ we infer that
\begin{align}
\label{forma_espl_drift_n}
\sum_{j,k=1}^n \langle Q_\infty A^*f^*_j,f^*_k\rangle \langle x,f^*_k\rangle\frac{\partial \varphi}{\partial\xi_j}(\langle x,f^*_1\rangle,\ldots,\langle x,f_n^*\rangle)=\langle P_nx,A^*Du(x)\rangle.
\end{align}
Furthermore, the symmetry of $\frac{\partial^2 \varphi}{\partial \xi_k \partial\xi_j}$ and the Lyapunov equation $Q_\infty A^*+AQ_\infty=-Q$ (see \cite[Formula 4.1]{GV03}) give
\begin{align}
 \sum_{j,k=1}^n \langle Q_\infty A^*f^*_j &,f^*_k\rangle \frac{\partial^2 \varphi}{\partial \xi_k \partial\xi_j}(\langle x,f^*_1\rangle,\ldots,\langle x,f_n^*\rangle) \notag  \\
= & \frac12 \sum_{j,k=1}^n (\langle Q_\infty A^*f^*_j,f^*_k\rangle+\langle AQ_\infty f^*_j,f^*_k\rangle) \frac{\partial^2 \varphi}{\partial \xi_k \partial\xi_j}(\langle x,f^*_1\rangle,\ldots,\langle x,f_n^*\rangle)  \notag \\
= & -\frac12 \sum_{j,k=1}^n \langle Qf^*_j,f^*_k\rangle  \frac{\partial^2 \varphi}{\partial \xi_k \partial\xi_j}(\langle x,f^*_1\rangle,\ldots,\langle x,f_n^*\rangle).
\label{forma_espl_n_diff}
\end{align}
By combining \eqref{forma_espl_E_n}, \eqref{forma_espl_drift_n} and \eqref{forma_espl_n_diff} and by applying Proposition \ref{prop_fcon1} we get the thesis.
\end{proof}

\subsection{Finite dimensional results and approximation of the solution}
\label{appr_sol_ell_problem}

In this section we introduce and study an elliptic differential equation in $\R^n$ which is linked to a stationary problem in infinite dimension.


For any $n\in\N$ consider the finite dimensional elliptic problem
\begin{align}
\label{finite_dim_ell_prob}
\lambda v(\xi)-\mathscr L_{n} v(\xi)=\varphi(\xi),
\end{align}
where $\lambda>0$, $\varphi:\R^n\ra\R$ is regular enough and $\mathscr L_{n}$ is the second order elliptic operator defined on smooth functions $\psi$ by
\begin{align}\label{invincible}
\mathscr L_{n} \psi(\xi)
= \frac12{\rm Tr}_{\R^n}[\mathscr Q_n \nabla^2\psi(\xi)]+\langle \xi,\mathscr B_n\nabla\psi(\xi)\rangle_{\R^n},
\end{align}
where $\mathscr{Q}_n$ and $\mathscr{B}_n$ are the matrices introduced in Lemma \ref{monstress}. We recall some classical results on the existence, uniqueness and other properties of solutions of elliptic equations in finite dimension.

For $k\in\N\cup\set{0}$ and $0<\gamma< 1$, we denote by $C_b^{k+\gamma}(\R^n)$ the space of the $k$-times differentiable functions with bounded and $\gamma$-H\"older derivatives up to the order $k$, endowed with its standard norm (see \cite[Section 2.7]{Tri78}), i.e. for $\theta\in C_b^{k+\gamma}(\R^n)$ we let $\norm{\theta}_{C_b^{k+\gamma}(\R^n)}:=\norm{\theta}_{C_b^k(\R^n)}+[\partial^k\theta]_\gamma$ where
\[[\partial^k\theta]_\gamma:=\sum_{\abs{\beta}=k}\sup\set{\frac{\abs{\partial^{\beta}\theta(\xi_1)-\partial^{\beta}\theta(\xi_2)}}{\abs{\xi_1-\xi_2}^\gamma}\tc \xi_1,\xi_2\in\R^n,\ \xi_1\neq\xi_2}.\]
Also $C^{k+\beta,m+\gamma}(\mathcal O\times\R^n)$ for $k,m\in\N\cup \set{0}$, $0<\beta,\gamma<1$ and $\mathcal O$ an open subset of $\R$ are the spaces of $k$-times differentiable functions with $\beta$-H\"older derivatives up to the order $k$ in the first variable and $m$-times differentiable functions with $\gamma$-H\"older derivatives up to the order $m$ in the second variable. As usual when we add the subscript \emph{loc} we mean that the H\"older condition holds locally.

\begin{pro}\label{Teo holder}
Assume Hypotheses \ref{ipo_RKH} and \ref{hyp:Qnon_deg} hold true and let $0<\gamma<1$. For every $\lambda>0$ and any $\varphi\in C_b^\gamma(\R^n)$ equation \eqref{finite_dim_ell_prob} has a unique solution $v\in C_b^{2+\gamma}(\R^n)$, and there exists a constant $C>0$, independent of $\varphi$, such that
\begin{gather}\label{stima holder}
\norm{v}_{ C_b^{2+\gamma}(\R^n)}\leq C\norm{\varphi}_{ C_b^\gamma(\R^n)}.
\end{gather}
Moreover if $\varphi\in C^\infty(\R^n)$, then $v\in C^\infty(\R^n)$.
\end{pro}
\noindent Inequality \eqref{stima holder} was proved in \cite[Theorem 1]{LV98}, and the local regularity result can be found in \cite[Theorem 3.1.1]{LSU68}.

Let $\lambda>0$ and let $f\in \fcon_{b,\Theta}^{\infty}(X)$ be such that $f(x)=\varphi(\langle x,f^*_1\rangle,\ldots,\langle x,f^*_{N_0}\rangle)$ for some $N_0\in\N$ and $\varphi\in C_b^\infty(\R^{N_0})$. 
Throughout the rest of this subsection we consider $n\geq  N_0$. Let
\begin{gather}\label{equazione finito dimensional}
\lambda v(\xi)-\OU_{n} v(\xi)=\varphi(\pi_{N_0}\xi),\qquad \xi\in\R^n,
\end{gather}
where $\pi_{N_0}:\R^n\ra\R^{N_0}$ is the projection on the first $N_0$ coordinates, and $\OU_{n}$ was introduced in \eqref{invincible}. We claim that the solution of \eqref{equazione finito dimensional} approximate, in some sense, the solution of $\lambda u-Lu=f$, where $L$ has been defined in \eqref{cerniera}.

By Proposition \ref{Teo holder}, equation (\ref{equazione finito dimensional}) admits a unique solution $v_n$ belonging to $ C^{\infty}(\R^n)\cap\bigcup_{\gamma\in(0,1)} C^{2+\gamma}_b(\R^n)$. Further, arguing as in \cite[Propositon 5.4]{CF16} it follows that $v_n$ is more regular. We give the proof of this fact for the convenience of the reader.

\begin{pro}\label{Soluzion e' c3b}
If Hypotheses \ref{ipo_RKH} and \ref{hyp:Qnon_deg} hold true, then $v_n$ belongs to $ C^3_b(\R^n)$.
\end{pro}

\begin{proof}
We just need to prove that the third order derivatives are bounded. Differentiating \eqref{equazione finito dimensional} it follows that
\begin{align}\label{wonder}
\lambda \nabla_jv_n(\xi)-\OU_{n} \nabla_jv_n(\xi)
=  \nabla_j\varphi(\pi_{N_0}\xi) +{\sum_{i=1}^n(\mathscr B_n)_{ji}\nabla_iv_n(\xi)},
\end{align}
for $1\leq j\leq N_0$, and
\begin{align}\label{woman}
\lambda \nabla_jv_n(\xi)-\OU_{n} \nabla_jv_n(\xi)
=  {\sum_{i=1}^n(\mathscr B_n)_{ji}\nabla_iv_n(\xi)}, 
\end{align}
for $j=N_0+1,\ldots, n$.
In \eqref{wonder} and \eqref{woman}, the right-hand sides are Lipschitz continuous and bounded. By Proposition \ref{Teo holder} we get $\nabla_jv_n\in\bigcup_{\gamma\in(0,1)} C^{2+\gamma}_b(\R^n)$ for every $j=1,\ldots,n$. In particular $v_n$ belongs to $ C^3_b(\R^n)$.
\end{proof}

For any $n>N_0$ and $x\in X$ we set
\[f(x):=\varphi(P_{N_0}x), \qquad V_n(x):=v_n(\langle x, f^*_1\rangle,\ldots,\langle x,f^*_n\rangle).\]

\begin{lemma}
\label{lem:sol_inf_dim_appr}
Assume Hypotheses \ref{ipo_RKH} and \ref{hyp:Qnon_deg} hold true. For any $n\geq N_0$ it holds that $V_n$ belongs to $\fcon^{3,n}_{b,\Theta}(X)$ and any $x\in X$
\begin{align}
\label{geologia}
\lambda V_n(x)- \frac{1}{2}{\rm Tr}_H[D_H^2V_n(x)]
-\langle P_n x,A^*DV_n(x)\rangle= f(x).
\end{align}
Furthermore for any $v\in \fcon _{b,\Theta}^{2,n}(X)$ 
\begin{align}
\mathcal E(V_n,v)
= & \int_X\left(f-\lambda V_n\right)v d\mu_\infty.
\label{uragano}
\end{align}
\end{lemma}

\begin{proof}
Observe that the fact that $V_n$ belongs to $\fcon^{3,n}_{b,\Theta}(X)$ follows by Proposition \ref{Soluzion e' c3b}. Let $\{h_n\,|\,n\in\N\}$ be an orthonormal basis of $H$. By Proposition \ref{prop_fcon1} we get
\begin{align*}
{\rm Tr}_H[D_H^2V_n(x)]
= & \sum_{i,\ell=1}^nq_{\ell i}\frac{\partial^2 v_n}{\partial \xi_i\partial\xi_\ell}(\langle x, f^*_1\rangle,\ldots,\langle x,f^*_n\rangle) = {\rm Tr}_{\R^n}[\mathscr Q_n\nabla^2 v_n(\langle x, f^*_1\rangle,\ldots,\langle x,f^*_n\rangle)].
\end{align*}
By Proposition \ref{prop_fcon1} we have
\begin{align*}
\langle P_nx,A^*DV_n(x)\rangle
= & \sum_{j,k=1}^n  \langle x,f^*_k\rangle b _{jk}\frac{\partial v_n}{\partial \xi_j}(\langle x, f^*_1\rangle,\ldots,\langle x,f^*_n\rangle) \\
=&  \langle (\langle x, f^*_1\rangle,\ldots,\langle x,f^*_n\rangle),\mathscr B_n\nabla v_n(\langle x, f^*_1\rangle,\ldots,\langle x,f^*_n\rangle)\rangle_{\R^n},
\end{align*}
Hence,
\begin{align*}
 \frac{1}{2}{\rm Tr}_{H}[D_H^2V_n(x)]
+\langle P_n x,A^*DV_n(x)\rangle
= & \mathscr L_{n}v_n(\langle x, f^*_1\rangle,\ldots,\langle x,f^*_n\rangle) \\
=& \lambda v_n(\langle x, f^*_1\rangle,\ldots,\langle x,f^*_n\rangle)-\varphi(\langle x, f^*_1\rangle,\ldots,\langle x,f^*_{N_0}\rangle) \\
=&  \lambda V_n(x)-f(x),
\end{align*}
and the thesis follows. \eqref{uragano} follows by combining \eqref{rappr_forma_dir_n} and \eqref{geologia}.
\end{proof}

\subsection{Back to infinite dimension}\label{backtoinfinity}

Let $\lambda>0$, let $f\in \fcon_{b,\Theta}^\infty(X)$ be such that $f(x)=\varphi(\langle x,f^*_1\rangle,\ldots,\langle x,f^*_{N_0}\rangle)$ for some $N_0\in\N$ and $\varphi\in C_b^\infty(\R^{N_0})$, and let $u\in D(L)$ be the unique solution of 
\[\lambda u-Lu=f.\] 

\begin{pro}
\label{prop:boundedness}
Assume Hypotheses \ref{ipo_RKH} and \ref{hyp:Qnon_deg} hold true and let $\lambda,f$ and $u$ as above. If $(V_m)_{m\in\N}$ is the sequence of functions introduced in Section \ref{appr_sol_ell_problem}, then $V_m$ converges to $u$ in $W^{1,2}_H(X,\mu_\infty)$ as $m$ goes to infinity. Moreover
\begin{align}
\label{stime_u_funz_der}
\|u\|_{L^2(X,\mu_\infty)}^2\leq \frac1{ \lambda} \|f\|_{L^2(X,\mu_\infty)}, \quad
\|D_Hu\|_{L^2(X,\mu_\infty;H)}^2\leq \frac{2}{\sqrt\lambda}\|f\|_{L^2(X,\mu_\infty)}.
\end{align}
\end{pro}

\begin{proof}
Let $m\in\N$. By Proposition \ref{Soluzion e' c3b} there exists $N_m\in\N$ such that $V_m$ belongs to $\fcon_{b,\Theta}^{3,N_m}(X)$. Using an obvious adjustment we can assume that also $f$ belongs to $\fcon_{b,\Theta}^{3,N_m}(X)$ for every $m\in\N$. By \eqref{geologia} we infer that for every $x\in X$ and $m\in\N$
\begin{align}
\label{equazione_fin_dim_1}
&  \lambda V_m(x)-\frac12{\rm Tr}_H[D_H^2V_m(x)]
-\langle P_n x,A^*DV_m(x)\rangle
=  f(x).
\end{align}
Multiplying both sides of \eqref{equazione_fin_dim_1} by $V_m$, integrating on $X$ with respect to $\mu_\infty$ and using \eqref{uragano} we obtain
\begin{align}
\label{equazione_fin_dim_2}
\lambda\int_XV_m^2d\mu_\infty-\mathcal E(V_m,V_m)=\int_XfV_md\mu_\infty,
\end{align}
for any $m\in\N$. Recalling \eqref{topolino} we obtain
\begin{align*}
\lambda\|V_m\|_{L^2(X,\mu_\infty)}^2\leq \|f\|_{L^2(X,\mu_\infty)}\|V_m\|_{L^2(X,\mu_\infty)},
\end{align*}
for any $m\in\N$. Thus
\begin{align}
\label{equazione_fin_dim_3}
\limsup_{m\rightarrow+\infty}\|V_m\|_{L^2(X,\mu_\infty)}\leq \frac1\lambda\|f\|_{L^2(X,\mu_\infty)}.
\end{align}
By \eqref{equazione_fin_dim_2} and \eqref{equazione_fin_dim_3} we infer that
\begin{align*}
\limsup_{m\rightarrow+\infty}\left(\lambda\|V_m\|_{L^2(X,\mu_\infty)}^2+\frac12\|D_HV_m\|_{L^2(X,\mu_\infty;H)}^2\right)\leq \frac1{\sqrt{\lambda}}\|f\|_{L^2(X,\mu_\infty)}^2.
\end{align*}
Consider a subsequence $(V_{m_k})_{k\in\N}$ of $(V_m)_{m\in\N}$ that converges weakly to $V$ in $W^{1,2}_H(X,\mu_\infty)$ as $k$ goes to infinity. Up to extracting another subsequence we can, and do, assume that both
$(\|V_{m_k}\|_{L^2(X,\mu_\infty)})_{k\in\N}$ and  $(\|D_HV_{m_k}\|_{L^2(X,\mu_\infty;H)})_{k\in\N}$
converge somewhere as $k$ goes to infinity. 

We claim that $V\in D(L)$ and $V=u$. To prove the claim, we consider $g\in \fcon_{b,\Theta}^1(X)$ and we multiply both the sides of \eqref{equazione_fin_dim_1} by $g$. Let $r\in\N$ be such that $g\in \fcon_{b,\Theta}^{1,r}(X)$. Then, for any $k\in\N$ such that $m_k\geq r$, by \eqref{uragano} we have
\begin{align}
\label{prova_V_D(L)}
\mathcal E(V_{m_k},g)=\int_X\lambda V_{m_k} gd\mu_\infty-\int_X fgd\mu_\infty.
\end{align}
Since $V_{m_k}$ converges weakly to $V$ in $W^{1,2}_H(X,\mu_\infty)$ as $k$ goes to infinity, by \eqref{prova_V_D(L)} we deduce that
\begin{align}
\label{forma_bil_soluzione}
\mathcal E(V,g)=\int_X(\lambda V-f)gd\mu_\infty.
\end{align}
The arbitrariness of $g$ and the definition of $D(L)$ imply that $V\in D(L)$ and $LV=\lambda V-f$, i.e., $V$ is the solution of $\lambda v-Lv=f$. This means that $V=u$.

By letting $g=u$ in \eqref{forma_bil_soluzione}, we obtain
\begin{align}
\label{1ciabatta}
\int_Xfud\mu_\infty =\lambda\|u\|_{L^2(X,\mu_\infty)}^2+\frac12 \|D_Hu\|_{L^2(X,\mu_\infty;H)}^2.
\end{align}
Now we take  $g=V_{m_k}$ in \eqref{forma_bil_soluzione}. Letting $k$ go to infinity, it follows that
\begin{align}
\label{1infradito}
\int_Xfud\mu_\infty=\lambda \lim_{k\rightarrow+\infty} \|V_{m_k}\|_{L^2(X,\mu_\infty)}^2+\frac12\lim_{k\rightarrow+\infty}\|D_HV_{m_k}\|_{L^2(X,\mu_\infty;H)}^2.
\end{align}
Since $(V_{m_k})_{k\in\N}$ converges weakly to $u$ in $W^{1,2}_H(X,\mu_\infty)$ as $k$ goes to infinity we get
\begin{align*}
\|u\|_{L^2(X,\mu_\infty)}
\leq \lim_{k\rightarrow+\infty}\|V_{m_k}\|_{L^2(X,\mu_\infty)}, \quad
\|D_{H}u\|_{L^2(X,\mu_\infty;H)}
\leq \lim_{k\rightarrow+\infty}\|D_HV_{m_k}\|_{L^2(X,\mu_\infty;H)}.
\end{align*}
Hence, by \eqref{1ciabatta} and \eqref{1infradito},
\begin{align*}
\|u\|_{L^2(X,\mu_\infty)}
=\lim_{k\rightarrow+\infty}\|V_{m_k}\|_{L^2(X,\mu_\infty)}, \quad
\|D_{H}u\|_{L^2(X,\mu_\infty;H)}
= \lim_{k\rightarrow+\infty}\|D_HV_{m_k}\|_{L^2(X,\mu_\infty;H)},
\end{align*}
which implies that $(V_{m_k})_{k\in\N}$ converges strongly to $u$ in $W^{1,2}_{H}(X,\mu_\infty)$ as $k$ goes to infinity. By a classical argument we get that any subsequence of $(V_m)_{m\in\N}$ admits a subsequence which converges strongly to $u$ in $ W^{1,2}_{H}(X,\mu_\infty)$. This means that the whole sequence $(V_m)_{m\in\N}$ converges strongly to $u$ in $W^{1,2}_{H}(X,\mu_\infty)$ as $m$ goes to infinity.
\end{proof}

In the next proposition we analyse the behaviour of the second order derivatives of $V_m$. To this aim, we need of the following hypotheses.
\begin{hyp}\label{ipo_2}
Assume Hypotheses \ref{ipo_RKH} hold true. There exists $\nu\in[0,1)$ such that for any $m\in\N$ and any $m\times m$ symmetric matrix $C$ with real entries it holds
\begin{align}
\label{cond_traccia_m}
{\rm Tr}_{\R^m}[(\mathscr B_m-\mathscr B_m^*) C(\mathscr B_m-\mathscr B_m^*)C]\geq -\nu{\rm Tr}_{\R^m}[\mathscr Q_mC\mathscr Q_mC],
\end{align}
where $\mathscr B_m$ and $\mathscr Q_m$ have been defined in Lemma \ref{monstress}.
\end{hyp}

Observe that since, for any $m\in\N$, the matrix $\mathscr Q_m$ is positive definite (Lemma \ref{monstress}), then 
\[{\rm Tr}_{\R^m}[\mathscr Q_mC\mathscr Q_mC]\geq0,\]
for any $m\times m$ symmetric matrix $C$ with real entries. Moreover if the Ornstein--Uhlenbeck is $Q$-symmetric, i.e. if $QA^*x^*=AQx^*$ for any $x^*\in D(A^*)$, then $B\in\mathcal L(H)$ is self-adjoint (see \cite[Section 4]{GV03}) and so $\mathscr B_m$ is self-adjoint for any $m\in\N$. This means that $\mathscr B_m-\mathscr B_m^*=0$, and so Hypothesis \ref{ipo_2} is satisfied. In Section \ref{sect_example} we give an example satifying Hypotheses \ref{ipo_2}, but such that the matrices $\mathscr{B}_m$ are not symmetric and the operator $Q$ and $A$ do not commutes.

\begin{pro}
\label{prop:der_seconde}
Assume Hypotheses \ref{hyp:Qnon_deg} and \ref{ipo_2} hold true and let $\lambda,f$ and $u$ as at the beginning of Section \ref{backtoinfinity}. If $(V_m)_{m\in\N}$ is the sequence of functions introduced in Section \ref{appr_sol_ell_problem}, then the sequences $(D_H^2V_m)_{m\in\N}$ and $(P_mD_{A_{\infty}}V_m)_{m\in\N}$ are bounded in $L^2(X,\mu_\infty;\mathcal H_2(H))$ and $L^2(X,\mu_\infty;H_\infty)$, respectively.
Furthermore
\begin{align}
\label{stima_fz_appr_m_2}
\limsup_{m\rightarrow+\infty}\left(\frac14(1-\nu)\|D_H^2V_m\|_{L^2(X,\mu_\infty;\mathcal H_2(H))}^2
+\|P_mD_{A_{\infty}}V_m\|_{L^2(X,\mu_\infty;H_\infty)}^2\right)
\leq 2\|f\|_{L^2(X,\mu_\infty)}^2,
\end{align}
where $\nu$ is the constant appearing in Hypotheses \ref{ipo_2}.
\end{pro}

\begin{proof}
Let $\{h_n\,|\,n\in\N\}$ be an orthonormal basis of $H$ and let $\ell\in\N$. Throughout the proof we will use the operator $\partial_\ell^H:\fcon_{b,\Theta}^1(X)\ra \fcon_{b,\Theta}(X)$ defined as
\[\partial_\ell^Hf:=[D_Hf,h_\ell]_H.\]
For any $m\geq N_0$ and any $x\in X$, Lemma \ref{lem:sol_inf_dim_appr} implies that
\begin{align}
\label{1eq_m_caccioppoli_1}
\lambda V_m(x)-\frac12{\rm Tr}_H[D^2_HV_m(x)]
-\langle P_{m}x,A^*DV_m(x)\rangle
=f(x).
\end{align}
Applying $\partial ^H_\ell$ to both sides of \eqref{1eq_m_caccioppoli_1} we get
\begin{gather}
\lambda \partial^H_\ell V_m(x)
-\frac12{\rm Tr}_H[D^2_H\partial_\ell^HV_m(x)]
-\langle \partial_\ell^H(P_{m}x),A^*DV_m(x)\rangle
-\langle P_{m}x, A^*D \partial_\ell^HV_m(x)\rangle 
= \partial_\ell^Hf(x). \label{1eq_m_caccioppoli_2}
\end{gather}
Straightforward computations give
\begin{align*}
{\rm Tr}_H[D_H\partial_\ell^HV_m(x)]
= & \sum_{i,j,s=1}^{m}[i^*f^*_j,i^* f^*_i]_H[i^*f^*_s,h_\ell]_H\frac{\partial^3 v_{m}}{\partial\xi_i\partial\xi_j\partial\xi_s}(\langle x,f^*_1\rangle,\ldots,\langle x,f^*_{m}\rangle), \\
\langle \partial_\ell^H(P_{m}x),A^*DV_m(x)\rangle
= & \sum_{i,j=1}^{m}[i^*f^*_i,h_\ell]_H\langle e_i,A^*f^*_j\rangle \frac{\partial v_{m}}{\partial \xi_j}(\langle x,f^*_1\rangle,\ldots,\langle x,f^*_{m}\rangle), \\
\langle P_{m}x, A^*D \partial_\ell^HV_m(x)\rangle
=& \sum_{i,j,s=1}^{m}\langle x,f^*_i\rangle \langle e_i,A^*f^*_j\rangle[i^*f^*_s,h_\ell] \frac{\partial^2 v_{m}}{\partial \xi_j\partial\xi_s}(\langle x,f^*_1\rangle,\ldots,\langle x,f^*_{m}\rangle),
\end{align*}
where $v_{m}$ is the unique solution of $\lambda v-\mathscr L_{m}v=\varphi$, with $\varphi\in C^\infty_b(\R^{m})$ satisfying $f(x)=\varphi(\langle x,f^*_1\rangle,\ldots,\langle x,f^*_{m}\rangle)$ for $x\in X$. Multiplying both sides of \eqref{1eq_m_caccioppoli_2} by $[BD_H V_m,h_\ell]_{H}$ and summing up $\ell$ from $1$ to $+\infty$, we get, after long but pretty standard calculations,
\begin{align*}
-\lambda [B & D_HV_m(x),D_HV_m(x)]_H+\big[BD_HV_m(x),{\rm Tr}_H[D^2_H(D_HV_m)(x)]_H\big]_H\\
&\quad +\big[BD_HV_m(x),\langle P_nx,A^*D(D_HV_m)(x)\rangle\big]_H+\big[P_mD_{A_{\infty}}V_m(x),P_mD_{A_{\infty}}V_m(x)\big]_{H_\infty}\\
&=[D_Hf(x),BD_HV_m(x)]_H,
\end{align*}
for any $x\in X$ and any $m\geq N_0$. Integrating with respect to $\mu_\infty$ and applying \eqref{rappr_forma_dir_n} and \eqref{uragano} we get
\begin{align}
\frac12\lambda\|D_HV_m\|_{L^2(X,\mu_\infty;H)}^2
&+\int_X[BD_HD_HV_m,D_HBD_HV_m]_Hd\mu_\infty\notag \\
&+\|P_mD_{A_{\infty}}V_m\|_{L^2(X,\mu_\infty;H_\infty)}^2=\int_X[D_Hf,BD_HV_m]_Hd\mu_\infty. \label{1eq_m_caccioppoli_3}
\end{align}
We now focus on the second addend of \eqref{1eq_m_caccioppoli_3}. It holds
\begin{align*}
[BD_HD_HV_m,D_HBD_HV_m]_H
= & \sum_{i,j,k,\ell=1}^{m}[Bi^*f^*_j,i^*f^*_k]_H[Bi^*f^*_i,i^*f^*_\ell]_H\frac{\partial^2 v_m}{\partial \xi_j\partial\xi_\ell}\frac{\partial^2 v_m}{\partial \xi_i\partial\xi_k} \\
= & {\rm Tr}_{\R^m}[\mathscr B_mD^2v_m\mathscr B_mD^2v_m],
\end{align*}
where $\mathscr B_m$ has been defined in \eqref{def_Qn_bn_Mn}. Recalling that by Lemma \ref{monstress} it holds $\mathscr B_m+\mathscr B_m^*=-2\mathscr Q_m$ for any $m\in\N$, combining \eqref{prop_nonsymm_matrix} and \eqref{cond_traccia_m} we obtain
\begin{align}
[BD_HD_HV_m,D_HBD_HV_m)]_H\notag 
= &  \frac14{\rm Tr}_{\R^m}[ \mathscr Q_m D^2v_m  \mathscr Q_m D^2v_m]\notag  \\
& +\frac14 {\rm Tr}_{\R^m}[( \mathscr B_m-\mathscr B_m^*) D^2v_m ( \mathscr B_m-\mathscr B_m^*) D^2v_m] \notag \\
\geq&  \frac14(1-\nu){\rm Tr}_H[D^2_HV_mD^2_HV_m] \notag \\		
= & \frac14(1-\nu)\|D^2_HV_m\|_{\mathcal H_2(H)}^2.
\label{1eq_m_caccioppoli_4}
\end{align}
As far as the right-hand side of \eqref{1eq_m_caccioppoli_3} is concerned, by \eqref{uragano} we get
\begin{gather*}
\int_X[D_Hf,BD_HV_m]_Hd\mu_\infty
= \mathcal E(V_m,f)
=  \int_X\left(f-\lambda V_m\right)f d\mu_\infty.
\end{gather*}
Putting together \eqref{1eq_m_caccioppoli_3} and  \eqref{1eq_m_caccioppoli_4}, taking the supremum limit as $m$ goes to infinity and taking into account \eqref{stime_u_funz_der} we infer that
\begin{align*}
\frac12\lambda\|D_Hu\|_{L^2(X,\mu_\infty;H)}^2 +\limsup_{m\rightarrow+\infty}\frac14(1-\nu)\|D_H^2V_m\|_{L^2(X,\mu_\infty;\mathcal H_2(H))}^2& \\
+\limsup_{m\rightarrow+\infty}\|P_mD_{A_{\infty}}V_m\|_{L^2(X,\mu_\infty;H_\infty)}^2 &\leq 2\|f\|_{L^2(X,\mu_\infty)}^2.\qedhere
\end{align*}
\end{proof}

\begin{rmk}
\label{rmk:successione_um_wek_conv}
By \eqref{stima_fz_appr_m_2} the sequences $(D_H^2V_m)_{m\in\N}$ and $(P_mD_{A_\infty}V_m)_{m\in\N}$ are bounded in $L^2(X,\mu_\infty;\mathcal  H_2(H))$ and in $L^2(X,\mu_\infty;H_\infty)$, respectively. By a standard compactness arguments, there exist $\Psi\in L^2(X,\mu_\infty;\mathcal H_2(H))$ and $\Phi\in L^2(X,\mu_\infty;H_\infty)$ and two subsequences $(D_H^2V_{m_k})_{k\in\N}$ and $(P_{m_k}D_{A_\infty}V_{m_k})_{k\in\N}$, such that the first one converges weakly to $\Psi$ in $L^2(X,\mu_\infty;\mathcal  H_2(H))$, while the latter one converges weakly to $\Phi$ in $L^2(X,\mu_\infty;H_\infty)$, as $k$ goes to infinity.
\end{rmk}

We conclude this section with a maximal regularity result.

\begin{thm}
\label{thm:max_reg}
Assume Hypotheses \ref{hyp:Qnon_deg} and \ref{ipo_2} hold true. If $\lambda>0$, $f\in L^2(X,\mu_\infty)$ and $u\in D(L)$ is the unique solution of $\lambda u-Lu=f$, then $u\in W^{2,2}_H(X,\mu_\infty)\cap W^{1,2}_{A_\infty}(X,\mu_\infty)$ and there exists a positive constant $K$ such that
\begin{align}
\label{stima_totale_u}
\|u\|_{W^{2,2}_H(X,\mu_\infty)}+\|{D_{A_\infty}u}\|_{L^2(X,\mu_\infty;H_\infty)}\leq K\|f\|_{L^2(X,\mu_\infty)}.
\end{align}
Finally, $D(L)$ is continuously embedded in the $\mathcal{U}(X,\mu_\infty)$, the completion of the space $C_b^2(X)\cap W^{2,2}_H(X,\mu_\infty)\cap W^{1,2}_{A_\infty}(X,\mu_\infty)$, with respect to the norm 
\[\|f\|^2_{\mathcal{U}(X,\mu_\infty)}:=\|f\|^2_{W^{2,2}_H(X,\mu_\infty)}+\|D_{A_\infty}f\|^2_{L^2(X,\mu_\infty;H_\infty)},\] 
defined for $f\in C_b^2(X)\cap W^{2,2}_H(X,\mu_\infty)\cap W^{1,2}_{A_\infty}(X,\mu_\infty)$.
\end{thm}

\begin{proof}
We split the proof into two steps: in the first step we consider the case when $f$ belongs to $\fcon_{b,\Theta}^\infty(X)$, in the second step we generalize the result to any $f\in L^2(X,\mu_\infty)$.

{\bf{Step 1.}} Let $f\in\fcon_{b,\Theta}^\infty(X)$. Let $u\in D(L)$ be such that $\lambda u-Lu=f$ and let $(V_{m_k})_{k\in\N}\subseteq \fcon_{b,\Theta}^{3}(X)$ be the subsequence defined in Remark \ref{rmk:successione_um_wek_conv}. By Proposition \ref{prop:boundedness} we know that $V_{m_k}$ converges strongly to $u$ in $W^{1,2}_H(X,\mu_\infty)$, as $k$ goes to infinity. We recall that Lemma \ref{lemma:stima_semigruppo_DH_DAinfty} implies that 
\[u=R(\lambda,L)f\in C_b^2(X)\cap W^{2,2}_H(X,\mu_\infty)\cap W^{1,2}_{A_\infty}(X,\mu_\infty).\]
We denote by $\mathbf{D}_H^*:D(\mathbf{D}_H^*)\subseteq L^2(X,\mu_\infty;\mathcal H_2(H))\rightarrow L^2(X,\mu_\infty;H)$ the adjoint operator of $\mathbf{D}_H$ in $L^2(X,\mu_\infty;H)$ (see Section \ref{sect_bfD}). Hence, by Lemma \ref{lem:sobolev_vett_der-seconda}, for any $G\in D(\mathbf{D}_H^*)$
\begin{align*}
\int_X[\Psi,G]_{\mathcal H_2(H)}d\mu_\infty
= & \lim_{k\rightarrow+\infty}\int_X[D^2_HV_{m_k},G]_{\mathcal H_2(H)}d\mu_\infty
=\lim_{k\rightarrow+\infty} \int_X[D_HV_{m_k},\mathbf{D}_H^*G]_{H}d\mu_\infty \\
= &  \int_X[D_Hu,\mathbf{D}_H^*G]_{H}d\mu_\infty= \int_X[D^2_Hu,G]_{\mathcal H_2(H)}d\mu_\infty.
\end{align*}
The density of $D(\mathbf{D}_H^*)$ in $L^2(X,\mu_\infty;\mathcal H_2(X))$ implies that $D^2_Hu=\Psi$. 

If $G$ belongs to $\fcon_{b,\Theta}(X;D(V))$, then, by the very definition of $\fcon_{b,\Theta}(X;D(V))$, there exists $\overline n\in\N$ such that $P_{\overline n}G=G$. By the definition of $\mathbb V$, Lemma \ref{domV^*} and Proposition \ref{pro:caratt_V_L2} we get
\begin{align*}
\int_X[BD_Hu,\mathbb VG]_Hd\mu_\infty
= & \lim_{k\rightarrow+\infty}\int_X[BD_HV_{m_k},\mathbb VG]_Hd\mu_\infty
= \lim_{k\rightarrow+\infty}\int_X[\mathbb V^*(BD_HV_{m_k}),G]_Hd\mu_\infty \\
= & \lim_{k\rightarrow+\infty}\int_X[D_{A_\infty}V_{m_k},G]_{H_\infty}d\mu_\infty 
=  \lim_{k\rightarrow+\infty}\int_X[P_{m_k}D_{A_\infty}V_{m_k},P_{\overline n}G]_{H_\infty}d\mu_\infty \\
= & \int_X[\Phi,P_{\overline n}G]_{H_\infty} d\mu_\infty
= \int_X[\Phi,G]_{H_\infty} d\mu_\infty.
\end{align*}
By the density of $\fcon_{b,\Theta}(X;D(V))$ in $D(\mathbb V)$ we infer that $BD_Hu\in D(\mathbb V^*)$ and $\mathbb V^*(BD_Hu)=\Phi$. Now Proposition \ref{pro:caratt_V_L2} gives us that for $\mu_\infty$-a.e. $x\in X$ it hold $BD_Hu(x)\in D(V^*)$ and, for such $x$'s,
\[V^*(BD_Hu(x))=(\mathbb V^*(BD_Hu))(x).\]  
Combining these facts with Lemma \ref{lemma:stima_semigruppo_DH_DAinfty} we infer that $\Phi(x)=D_{A_\infty}u(x)$ for $\mu_\infty$-a.e. in $x\in X$. Now combinig the above arguments, \eqref{stime_u_funz_der} and \eqref{stima_fz_appr_m_2} we obtain \eqref{stima_totale_u}.

{\bf Step 2.} Let $f\in L^2(X,\mu_\infty)$ and let $(f_n)_{n\in\N}\subseteq \fcon_{b,\Theta}^\infty(X)$ be a sequence converging to $f$ in $L^2(X,\mu_\infty)$ as $n$ goes to infinity. For any $n\in\N$ let $u_n$ be the solution of $\lambda v-Lv=f_n$. 

By the first step of the proof and \eqref{stima_totale_u}, the sequence $(u_n)_{n\in\N}$ is a Cauchy sequence in $W^{2,2}_H(X,\mu_\infty)$ and the sequence $(D_{A_\infty}u_n)_{n\in\N}$ is a Cauchy sequence in $L^2(X,\mu_\infty;H_\infty)$. So there exist $u\in W^{2,2}(X,\mu_\infty)$ and $\Psi \in L^2(X,\mu_\infty;H_\infty)$ such that $u_n$ converges to $u$ in $W^{2,2}_H(X,\mu_\infty)$ and $D_{A_\infty}u_n$ converges $\Psi$ in $L^2(X,\mu_\infty;H_\infty)$, as $n$ goes to infinity. By the closability of $D_{A_\infty}$ (Section \ref{sect_Ainfty}), we get that $u\in W^{2,2}_H(X,\mu_\infty)\cap W^{1,2}_{A_\infty}(X,\mu_\infty)$ and there exists $K>0$ such that 
\begin{align*}
\|u\|_{W^{2,2}_H(X,\mu_\infty)}+\|D_{A_\infty}u\|_{L^2(X,\mu_\infty;H_\infty)}\leq K \|f\|_{L^2(X,\mu_\infty)}.
\end{align*}
It remains to prove that $u$ belongs to $D(L)$ and it satisfies $\lambda u-Lu=f$. For any $n\in\N$ and any $g\in \fcon_{b,\Theta}^1(X)$ we have
\begin{align}\label{pers}
\lambda \int_X u_ngd\mu_\infty-\int_X [BD_Hu_n,D_Hg]_Hd\mu_\infty=\int_Xf_ngd\mu_\infty.
\end{align}
Taking the limit for $n$ tending to infinity in \eqref{pers} we get
\begin{align*}
\lambda \int_X ugd\mu_\infty-\int_X [BD_Hu,D_Hg]_Hd\mu_\infty=\int_Xfgd\mu_\infty,
\end{align*}
which gives $\mathcal E(u,g)=\int_X(f-\lambda u)gd\mu_\infty$.
This implies that $u\in D(L)$ and $Lu=\lambda u-f$.

To prove the last assertion, consider $u\in D(L)$ and set $f:=\lambda u-Lu\in L^2(X,\mu_\infty)$. Let $(f_n)_{n\in\N}\subseteq \fcon_{b,\Theta}^\infty(X)$ be such that $f_n$ converges to $f$ in $L^2(X,\mu_\infty)$ as $n$ goes to infinity, and for any $n\in\N$ let $u_n\in D(L)$ be the unique solution to $\lambda v-Lv=f_n$. By Lemma \ref{lemma:stima_semigruppo_DH_DAinfty} we know that for any $n\in\N$
\[u_n=R(\lambda,L)f_n \in C_b^2(X)\cap W^{2,2}_H(X,\mu_\infty)\cap W^{1,2}_{A_\infty}(X,\mu_\infty).\] 
To conclude, we notice that \eqref{stima_totale_u}, with $u$ replaced by $u-u_n$ and $f$ replaced by $f-f_n$, gives that $u_n$ converges to $u$, as $n$ goes to infinity, with respect to $\norm{\cdot}_{\mathcal{U}(X,\mu_\infty)}$.
\end{proof}

\section{The case of a degenerate $Q$}
\label{sec:degenere}


In this section we show that Hypotheses \ref{hyp:Qnon_deg} can be omited from the results of Section \ref{sec:inclusione_1}  and we prove that Theorem \ref{thm:max_reg} still holds true. Throughout this section we let ${\rm Id}_n$ be the identity operator from $\R^n$ to itself and we denote with $\nabla$ the gradient operator in $\R^n$.  

For any $\varepsilon>0$ let us introduce the bilinear forms $(\mathcal E_\infty, \fcon_{b,\Theta}^1(X))$ and $(\mathcal E_\varepsilon,\fcon_{b,\Theta}^1(X))$ as
\begin{align}
\notag
\mathcal E_\infty(u,v) & :=\int_X[D_{H_\infty}u,D_{H_\infty}v]_{H_\infty}d\mu_\infty, \\
\mathcal E_{\varepsilon}(u,v) & :=\int_X(-[BD_Hu,D_Hv]_H+\varepsilon[D_{H_\infty}u,D_{H_\infty}v]_{H_\infty})d\mu_\infty
=\mathcal E(u,v)+\varepsilon\mathcal E_\infty(u,v),
\label{forma_epsilon}
\end{align}
for any $ u,v\in D(\mathcal E_\varepsilon):=\fcon_{b,\Theta}^1(X)$. The fact that $D_H$ and $D_{H_\infty}$ are closed operators, and the same arguments as in Section \ref{sec:OU_operator}, give that $(\mathcal E_\varepsilon, \fcon_{b,\Theta}^1(X))$ is a coercive closable bilinear form on $L^2(X,\mu_\infty)$ for any $\varepsilon>0$. Let us denote by $D(\mathcal E_\varepsilon)$ its domain. Hence, we can define a densely defined operator $L_\varepsilon$ as follows:
\begin{align*}
\left\{
\begin{array}{ll}
D(L_\varepsilon) :=\Big\{u\in D(\mathcal E_\varepsilon)\Big|\ {\textrm{there exists $g\in L^2(X,\mu_\infty)$ such that }}\\
\qquad \qquad \quad \qquad \qquad \qquad \qquad \displaystyle \mathcal E_\varepsilon(u,v)=-\int_X gvd\mu_\infty, \ \forall v\in\fcon_{b,\Theta}^{1}(X)\Big\}, \\
L_\varepsilon u :=g.
\end{array}
\right.
\end{align*}
Our choice of $\fcon_{b,\Theta}^{1}(X)$ as the space of test functions follows by its density in $D(\mathcal E_\varepsilon)$ (see Proposition \ref{prop:appr_DH_CbTheta}). It is not hard to show that for any $n\in\N$ and any $f\in\fcon_{b,\Theta}^{2,n}(X)$ of the form $f(x)=\varphi(\langle x,f^*_1\rangle, \ldots,\langle x,f^*_n\rangle)$ with $\varphi\in C_b^2(\R^n)$ we have for any $x\in X$
\begin{align*}
L_{\varepsilon}f(x)
&=  \sum_{j,k=1}^n\left(\frac12 \langle Qf^*_j,f^*_k\rangle +\varepsilon\delta_{jk} \right)\frac{\partial^2\varphi}{\partial{{\xi}_j}\partial{{\xi}_k}}(\langle x,f^*_1\rangle, \ldots,\langle x,f^*_n\rangle)\\ 
&\quad+\sum_{k=1}^n\langle x,(A^*-\varepsilon)f^*_k\rangle\frac{\partial \varphi}{\partial\xi_k}(\langle x,f^*_1\rangle, \ldots,\langle x,f^*_n\rangle).
\end{align*}
Further, if we denote by $L_\infty$ the operator associated with $\mathcal E_\infty$ we get for any $x\in X$
\begin{align*}
L_\infty f(x)
= & {\rm Tr}_{H_\infty}[D^2_{H_\infty}f(x)]-\langle x,Df\rangle \\
= & \sum_{j,k=1}^n\frac{\partial^2\varphi}{\partial \xi_j\partial\xi_k}(\langle x,f^*_1\rangle,\ldots,\langle x,f^*_n\rangle)-\sum_{j=1}^n\langle x,f^*_j\rangle \frac{\partial\varphi}{\partial\xi_j}(\langle x,f^*_1\rangle,\ldots,\langle x,f^*_n\rangle),
\end{align*}
for any $f\in\fcon_{b,\Theta}^\infty(X)$.
For any $n\in\N$ and any $\varepsilon>0$ let us consider the finite dimensional operator defined on smooth functions $\varphi:\R^n\rightarrow \R$ by
\begin{align*}
\mathcal L^\varepsilon_k \varphi(\xi):=\frac12{\rm Tr}_{\R^n}[(\mathscr Q_n+\varepsilon {\rm Id}_n)D^2\varphi(\xi)]+\langle \xi,(\mathscr B_n-\varepsilon {\rm Id}_n)\nabla\varphi(\xi)\rangle_{\R^n}, \qquad \xi\in\R^n.
\end{align*} 
Let $f\in\fcon_{b,\Theta}^{\infty,N_0}(X)$ of the form $f(x)=\varphi(\langle x,f^*_1\rangle, \ldots,\langle x,f^*_{N_0}\rangle)$ with $\varphi\in C_b^2(\R^{N_0})$, and let $n>N_0$. The uniform ellipticity of $L_\varepsilon$ implies that for any $\lambda>0$ there exists a unique solution $v_n^\varepsilon$ to the equation $\lambda v-\mathcal L^\varepsilon_k v=\varphi$ with $v_n^\varepsilon \in C_b^{2+\gamma}(\R^n)$ for any $\gamma\in(0,1)$ and arguing as in Proposition \ref{Soluzion e' c3b} it follows that $v_n^\varepsilon\in C^3_b(\R^n)$.

Let us set $V^\varepsilon_k(x):=v_n^\varepsilon(\langle x,f^*_1\rangle, \ldots,\langle x,f^*_n\rangle)$ for any $x\in X$. By repeating the computations as in Lemma \ref{lem:sol_inf_dim_appr} we infer that
\begin{align}
\label{eq_ellipt_epsilon}
\lambda V_n^\varepsilon(x)-\frac{1}{2}{\rm Tr}_H[D_H^2V_n^\varepsilon(x)]-\langle P_nx,A^*DV_n^\varepsilon(x)\rangle-\varepsilon{\rm Tr}_{H_\infty}[D^2_{H_\infty}V_n^\varepsilon(x)]+\varepsilon\langle x,DV_n^\varepsilon(x)\rangle=f,
\end{align}
for any $x\in X$. We multiply both the sides of \eqref{eq_ellipt_epsilon} by $V_n^\varepsilon(x)$ and we integrate on $X$ with respect to $\mu_\infty$. By taking advantage of \eqref{rappr_forma_dir_n} it follows that
\begin{align*}
\lambda\|V_n^\varepsilon\|_{L^2(X,\mu_\infty)}^2+\frac12\|D_HV_n^\varepsilon\|_{L^2(X,\mu_\infty;H)}^2
+\varepsilon\|D_{H_\infty}V_n^\varepsilon\|^2_{L^2(X,\mu_\infty;H_\infty)}
\leq \|f\|_{L^2(X,\mu_\infty)}\|V_n^\varepsilon\|_{L^2(X,\mu_\infty)}
\end{align*}
This implies that for every $\eps>0$ and $n>N_0$
\begin{align}
\|V_n^\varepsilon\|_{L^2(X,\mu_\infty)}^2& \leq \frac1{\lambda}\|f\|_{L^2(X,\mu_\infty)}^2;\notag\\
\|D_HV_n^\varepsilon\|_{L^2(X,\mu_\infty;H)}^2& \leq \frac2{\sqrt\lambda}\|f\|_{L^2(X,\mu_\infty)}^2;\notag\\
\|D_{H_\infty}(\sqrt \varepsilon V_n^\varepsilon)\|^2_{L^2(X,\mu_\infty;H_\infty)}
& \leq \frac{1}{\sqrt\lambda}\|f\|_{L^2(X,\mu_\infty)}^2.\label{ham,}
\end{align}

Hence, there exists $V\in W^{1,2}_H(X,\mu_\infty)$ and a sequence $(V_{n_k}^{\varepsilon_k})_{k\in\N}$, where $(\varepsilon_k)_{k\in\N}$ is a decreasing infinitesimal sequence, such that $V_{n_k}^{\varepsilon_k}$ converges weakly to $V$ in $W^{1,2}_H(X,\mu_\infty)$ and $D_{H_\infty}(\sqrt {\varepsilon_k} V^{\varepsilon_k}_{n_k})$ converges weakly to some $\Phi$ in $L^2(X,\mu_\infty;H_\infty)$ as $n$ tends to infinity. 
We claim that $\Phi=0$. Indeed, for any $G\in D({\rm div}_{H_\infty})$ (see \cite[Section 5.8]{Bog98}, for the definition of ${\rm div}_{H_\infty}$) we have
\begin{align*}
\int_X[\Phi,G]_{H_\infty} d\mu_\infty
= & \lim_{n\rightarrow+\infty}\int_X[D_{H_\infty}(\sqrt{\varepsilon_k}V^{\varepsilon_{k}}_n),G]_{H_\infty}d\mu_\infty 
=  \lim_{n\rightarrow+\infty}\int_X(\sqrt{\varepsilon_k}V^{\varepsilon_k}_{n_k}){\rm div}_{H_\infty}Gd\mu_\infty=0.
\end{align*}
The density of $D({\rm div}_{H_\infty})$ gives the claim.

We claim that $V\in D(L)$. Let $g\in\fcon_{b,\Theta}^1(X)$, and let $r\in\N$ be such that $g\in \fcon_{b,\Theta}^{1,r}(X)$. We multiply both the sides of \eqref{eq_ellipt_epsilon}, with $n$ replaced by $n_k$ and $\varepsilon$ replaced by $\varepsilon_k$, by $g$ and we integrate on $X$ with respect $\mu_\infty$, taking into account \eqref{rappr_forma_dir_n}, \eqref{forma_epsilon} and \eqref{ham,} we infer that
\begin{align*}
\lambda\int_XV^{\varepsilon_k}_{n_k}gd\mu_\infty+\int_X[D_HV_{n_k}^{\varepsilon_k},B^*D_Hg]_Hd\mu_\infty+\varepsilon_k\int_X[D_{H_\infty}V^{\varepsilon_k}_{n_k},D_{H_\infty}g]_{H_\infty}d\mu_\infty=\int_Xfgd\mu_\infty,
\end{align*}
for any $n\in\N$. Letting $n$ go to infinity we get
\begin{align*}
\lambda\int_XV  gd\mu_\infty+\int_X[D_{H}V,B^*D_Hg]_Hd\mu_\infty=\int_Xfgd\mu_\infty,
\end{align*}
which gives
\begin{align*}
\mathcal E(V,g)=\int_X(\lambda V-f)gd\mu_\infty, \quad g\in\fcon_{b,\Theta}^1(X).
\end{align*}
This implies that $V\in D(L)$ and $LV=\lambda V-f$.

Let us apply the operator $\partial_\ell^H$ to both the sides in \eqref{eq_ellipt_epsilon}, multiply by $[BD_HV^\varepsilon_k,h_\ell]_H$ and summing $\ell$ from $1$ to $\infty$. By repeating the computations as in the proof of Proposition \ref{prop:der_seconde} we get
\begin{align*}
-\lambda [B & D_HV_n^\varepsilon (x),D_HV_n^\varepsilon (x)]_H 
+\frac{1}{2}\big[BD_HV_n^\varepsilon (x),{\rm Tr}_H[D^2_H(D_HV_n^\varepsilon )(x)]\big]_H\\
&+\big[BD_HV_n^\varepsilon (x),\langle P_nx,A^*D(D_HV_n^\varepsilon )(x)\rangle\big]_H
+\big[P_mD_{A_{\infty}}V_n^\varepsilon (x),P_mD_{A_{\infty}}V_n^\varepsilon (x)\big]_{H_\infty}\\
&+\varepsilon\left[BD_HV_n^\varepsilon(x),L_\infty(D_HV_n^\varepsilon)(x)\right]_H 
-\varepsilon[D_HV_n^\varepsilon(x),B D_HV_n^\varepsilon(x)]_H=[D_Hf(x),BD_HV_n^\varepsilon (x)]_H.
\end{align*}
We claim that integrating on $X$ with respect to $\mu_\infty$ we get
\begin{align}
\frac12(\lambda &+\varepsilon)\|D_HV_n^\varepsilon\|_{L^2(X,\mu_\infty;H)}^2+\frac14(1-\nu)\|D^2_HV_n^\varepsilon\|_{L^2(X,\mu;\mathcal H_2(H))}^2\notag\\
&+\|P_nD_{A_\infty}V^\varepsilon_k\|_{L^2(X,\mu_\infty;H_\infty)}^2 +\frac12\varepsilon\|D_{H_\infty}(D_H V_n^\varepsilon)\|_{L^2(X,\mu_\infty;H\otimes H_\infty)}^2 
\leq 2\|f\|_{L^2(X,\mu_\infty)}^2.
\label{stima_der_seconde_epsilon}
\end{align}
To prove the claim, it is enough to compute
\begin{align*}
\int_X\left[BD_HV_n^\varepsilon,L_\infty(D_HV_n^\varepsilon)\right]_Hd\mu_\infty,
\end{align*}
since the other terms can be evaluated arguing as in the proof of Proposition \ref{prop:der_seconde}. We have
\begin{align*}
\int_X&\left[BD_HV_n^\varepsilon,L_\infty(D_HV_n^\varepsilon)\right]_Hd\mu_\infty \\
= & \sum_{r,s=1}^n[i^*f^*_s,Bi^*f^*_r]_H\int_X \left(L_\infty\frac{\partial v_n^\varepsilon}{\partial\xi_r}\right)\frac{\partial v_n^\varepsilon}{\partial\xi_s}d\mu_\infty \\
= & -\sum_{r,s=1}^n[i^*f^*_s,Bi^*f^*_r]_H\int_X\left[D_{H_\infty}\frac{\partial v_n^\varepsilon}{\partial\xi_r},D_{H_\infty}\frac{\partial v_n^\varepsilon}{\partial\xi_s}\right]_{H_\infty}d\mu_\infty \\
= &-\sum_{r,s=1}^n\frac12([i^*f^*_s,Bi^*f^*_r]_H+[i^*f^*_r,Bi^*f^*_s]_H)\int_X\left[D_{H_\infty}\frac{\partial v_n^\varepsilon}{\partial\xi_r},D_{H_\infty}\frac{\partial v_n^\varepsilon}{\partial\xi_s}\right]_{H_\infty}d\mu_\infty \\
= & \frac12\sum_{r,s=1}^n[i^*f^*_s,i^*f^*_r]_H\int_X\left[D_{H_\infty}\frac{\partial v_n^\varepsilon}{\partial\xi_r},D_{H_\infty}\frac{\partial v_n^\varepsilon}{\partial\xi_s}\right]_{H_\infty}d\mu_\infty \\
= & \frac12\int_X[D_{H_\infty}(D_HV_n^\varepsilon),D_{H_\infty}(D_HV_n^\varepsilon)]_{H\otimes H_\infty}d\mu_\infty.
\end{align*}
This gives the claim.

By \eqref{stima_der_seconde_epsilon}, there exist $\Psi\in L^2(X,\mu_\infty;\mathcal H_2(H))$ and $\Phi\in L^2(X,\mu_\infty;H_\infty)$ such that, up to a subsequence, the sequences $(D^2_HV_{n_k}^{\varepsilon_k})_{n\in\N}$ and $(P_{n_k}D_{A_\infty}V_{n_k}^{\varepsilon_k})_{n\in\N}$ weakly converge to $\Psi$ and $\Phi$ in $L^2(X,\mu_\infty;\mathcal H_2(H))$ and in $L^2(X,\mu_\infty;H_\infty)$, respectively, as $n$ goes to infinity. Arguing as in Theorem \ref{thm:max_reg} we infer that 
\[D(L)\subseteq \overline{C_b^2(X)\cap W^{2,2}_H(X,\mu_\infty)\cap W^{1,2}_{A_\infty}(X,\mu_\infty)}^{\|\cdot\|_{\mathcal{U}(X,\mu_\infty)}},\] 
and there exists a positive constant $K$ such that for any $f\in L^2(X,\mu_\infty)$ we have
\begin{align*}
\|R(\lambda,L)f\|_{W^{2,2}_H(X,\mu_\infty)}+\|D_{A_\infty} R(\lambda,L)f\|_{L^2(X,\mu_\infty;H_\infty)}\leq K\|f\|_{L^2(X,\mu_\infty)}.
\end{align*}

So by the above considerations we get the following result.

\begin{thm}\label{emb1}
Assume Hypotheses \ref{ipo_2} hold true. If $\lambda>0$, $f\in L^2(X,\mu_\infty)$ and $u\in D(L)$ is the unique solution of $\lambda u-Lu=f$, then $u\in W^{2,2}_H(X,\mu_\infty)\cap W^{1,2}_{A_\infty}(X,\mu_\infty)$ and there exists a positive constant $K$ such that
\begin{align*}
\|u\|_{W^{2,2}_H(X,\mu_\infty)}+\|{D_{A_\infty}u}\|_{L^2(X,\mu_\infty;H_\infty)}\leq K\|f\|_{L^2(X,\mu_\infty)}.
\end{align*}
Finally, $D(L)$ is continuously embedded in the $\mathcal{U}(X,\mu_\infty)$, the completion of the space $C_b^2(X)\cap W^{2,2}_H(X,\mu_\infty)\cap W^{1,2}_{A_\infty}(X,\mu_\infty)$, with respect to the norm 
\[\|f\|^2_{\mathcal{U}(X,\mu_\infty)}:=\|f\|^2_{W^{2,2}_H(X,\mu_\infty)}+\|D_{A_\infty}f\|^2_{L^2(X,\mu_\infty;H_\infty)},\] 
defined for $f\in C_b^2(X)\cap W^{2,2}_H(X,\mu_\infty)\cap W^{1,2}_{A_\infty}(X,\mu_\infty)$.
\end{thm}

\section{The $H$-divergence operator}
\label{sec:Hdiv}

In this section we study the behaviour of the $H$-divergence operator (the adjoint of the gradient operator $D_H$ in $L^2(X,\mu_\infty)$). We start with a technical lemma.

\begin{lemma}
Assume Hypotheses \ref{ipo_RKH} hold true. For any $f,g\in \fcon_b^{1}(X)$ and any $h,k\in D(V^*)$ it holds
\begin{align}
\int_X\Big([D_Hf,h]_H-&f\widehat{V^*h}\Big)\Big([D_Hg,k]_H-g\widehat{V^*k}\Big)d\mu_\infty\notag\\
=&\int_X\Big(fg[V^*h,V^*k]_{H_\infty} +[D_Hg,h]_H[D_Hf,k]_H\Big)d\mu_\infty.
\label{int_parti_divergenza}
\end{align}
\end{lemma}

\begin{proof}
A standard argument and \eqref{int_parti_peso} gives us
\begin{align*}
\int_X \Big([D_Hf,h]_H-f\widehat{V^*h}\Big)g\widehat{V^*k}d\mu_\infty
= & -\int_Xfg\widehat{ V^*h}\widehat {V^*k}d\mu_\infty+\int_Xfg\widehat {V^*k}\widehat{V^*h}d\mu_\infty\\
&-\int_X[D_{H}g, {h}]_{H}f\widehat{V^*k}d\mu_\infty-\int_Xfg[D_{H}\widehat {V^*k},h]_{H}d\mu_\infty \\
=& \int_X\Big(f[D_{H}g, {h}]_{H}\widehat{V^*k}+fg[D_{H}\widehat {V^*k},h]_{H}\Big)d\mu_\infty.
\end{align*}
Arguing as in the proof of \cite[Theorem 5.1.13]{Bog98} we obtain that $[D_{H}\widehat {V^*k},h]_{H}=[V^*k,V^*h]_{H_\infty}$. So we can write
\begin{align*}
\int_X \Big([D_Hf,h]_H-f\widehat{V^*h}\Big)g\widehat{V^*k}d\mu_\infty
= -\int_X\Big(f[D_{H}g, h]_{H}\widehat{V^*k}+fg[V^*k,V^*h]_{H_\infty}\Big)d\mu_\infty.
\end{align*}
This conclude the proof.
\end{proof}

We introduce the $H$-divergence operator on $L^2(X,\mu_\infty;H)$ as the adjoint operator of $D_H$, and we denote it by $({\rm div}_H,D({\rm div}_H))$. By \eqref{int_parti_peso}, for any $\Phi\in \fcon_b^{1}(X;D(V^*))$ such that $\Phi=\sum_{i=1}^n\phi_i h_i$ for some $n\in\N$, $\phi_i\in \fcon_b^1(X)$ and $h_i\in D(V^*)$ for $i=1,\ldots,n$, it holds
\begin{align*}
\int_X[D_Hf,\Phi]_Hd\mu_\infty
=-\int_Xf\sum_{i=1}^n\left([D_H\phi_i,h_i]_H-\phi_i \widehat {V^*h_i}\right)d\mu_\infty,
\end{align*}
for any $f\in \fcon_b^{1}(X)$. Hence $\fcon_b^{1}(X;D(V^*))\subseteq D({\rm div}_H)$ and
\begin{align}
\label{divergenza_smooth_fct}
{\rm div}_H\Phi=-\sum_{i=1}^n\left([D_H\phi_i,h_i]_H-\phi_i \widehat {V^*h_i}\right),
\end{align}
whenever $\Phi\in \fcon_b^{1}(X;D(V^*))$ such that $\Phi=\sum_{i=1}^n\phi_i h_i$ for some $n\in\N$, $\phi_i\in \fcon_b^1(X)$ and $h_i\in D(V^*)$ for any $i=1,\ldots,n$.

In the next lemma we introduce a subspace of $D({\rm div}_H)$ which will help in the characterization of the domain of $L$.

\begin{lemma}
\label{lemm:incl_div}
Assume Hypotheses \ref{ipo_RKH} hold true and let $\mathcal{U}(X,\mu_\infty;H)$ be the completion of the space $\fcon_b^{1}(X;D(V^*))$ with respect to the norm
\begin{align*}
\|\Phi\|_{\mathcal U(X,\mu_\infty;H)}^2:=\|\mathbf{D}_H\Phi\|_{L^2(X,\mu_\infty;H)}^2+\|\mathbb V^*\Phi\|_{L^2(X,\mu_\infty;H_\infty)}^2,
\end{align*} 
defined for $\Phi\in \fcon_b^{1}(X;D(V^*))$. Moreover $\mathcal U(X,\mu_\infty;H)$ is continuously embedded in $D({\rm div}_H)$, endowed with the graph norm, and for any $\Phi\in \mathcal U(X,\mu_\infty;H)$ it holds 
\[\|{\rm div}_H\Phi\|_{L^2(X,\mu_\infty)}\leq \|\Phi\|_{\mathcal U(X,\mu_\infty;H)}.\]
\end{lemma}

\begin{proof}
Consider $\Psi\in \fcon_b^{1}(X;D(V^*))$ such that $\Psi=\sum_{i=1}^k\psi_i h_i$ for some $k\in\N$, $\psi_i\in \fcon_b^1(X)$ and $h_i\in D(V^*)$ for $i=1,\ldots,k$. Without loss of generality, we can assume that $\{h_1,\ldots,h_k\}$ is an orthonormal family in $H$. By \eqref{Mari}, \eqref{int_parti_divergenza} and \eqref{divergenza_smooth_fct} it holds
\begin{align*}
\int_X\left({\rm div}_H\Psi\right)^2d\mu_\infty&=\sum_{i=1}^k\sum_{j=1}^k\int_X\Big([D_H\psi_i,h_i]_H-\psi_i\widehat{V^*h_i}\Big)\Big([D_H\psi_j,h_j]_H-\psi_j\widehat{V^*h_j}\Big)d\mu_\infty\\
&=\sum_{i=1}^k\sum_{j=1}^k \int_X\Big(\psi_i \psi_j[V^*h_i,V^*h_j]_{H_\infty} +[D_H\psi_i,h_j]_H[D_H\psi_j,h_i]_H\Big)d\mu_\infty\\
&=\int_X\Big([\mathbb V^*\Psi,\mathbb V^*\Psi]_{H_\infty}
+{\rm Tr}_H[\mathbf{D}_H\Psi\otimes \mathbf{D}_H\Psi]\Big)d\mu_\infty.
\end{align*}
A standard calculation gives ${\rm Tr}_H[\mathbf{D}_H\Psi\otimes \mathbf{D}_H\Psi]\leq \|\mathbf{D}_H\Psi\|_{\mathcal H_2(H)}^2$, so
\begin{align*}
\|{\rm div}_H\Psi\|_{L^2(X,\mu_\infty)}^2 \leq \|\mathbf{D}_H\Psi\|_{L^2(X,\mu_\infty;\mathcal H_2(H))}^2+\|\mathbb V^*\Psi\|_{L^2(X,\mu_\infty;H_\infty)}^2=\|\Psi\|_{\mathcal U(X,\mu_\infty;H)}^2.\qquad\quad\qedhere
\end{align*}
\end{proof}

Before stating the main result of the current section we need the following lemma.

\begin{lemma}
If Hypotheses \ref{ipo_RKH} hold true, then for every $v\in \fcon_{b,\Theta}^2(X)$ it holds
\begin{align}
\label{stima_D_H_vett_D2}
\|\mathbf{D}_HBD_Hv\|_{\mathcal H_2(H)}
= \frac12\|D^2_Hv\|_{\mathcal H_2(H)}.
\end{align}
\end{lemma}

\begin{proof}
Let $v\in \fcon_{b,\Theta}^2(X)$ be of the form $v(x):=\psi(\langle x,f^*_1\rangle,\ldots,\langle x,f^*_{r}\rangle)$ for some $r\in\N$ and $\psi\in C_b^\infty(\R^{r})$ and any $x\in X$. It holds
\begin{align*}
BD_Hv(x)=\sum_{j=1}^{r}\frac{\partial \psi}{\partial \xi_j}(\langle x,f^*_1\rangle,\ldots,\langle x,f^*_{r}\rangle)Bi^*f^*_j, 
\end{align*}
and
\begin{align*}
\mathbf{D}_HBD_Hv(x)
= \sum_{j,k=1}^{k}\frac{\partial^2 \psi}{\partial \xi_j\partial \xi_k}(\langle x,f^*_1\rangle,\ldots,\langle x,f^*_{r}\rangle)[(Bi^*f^*_j)\otimes f^*_k].
\end{align*}
So for every $x\in X$
\begin{align*}
\|D^2_Hv(x)\|_{\mathcal H_2(H)}^2
= & \sum_{j,k=1}^{r}\pa{\frac{\partial^2 \psi}{\partial \xi_j\partial \xi_k}(\langle x,f^*_1\rangle,\ldots,\langle x,f^*_{r}\rangle)}^2[i^*f^*_j,i^*f^*_k]_H^2,\\
\|\mathbf{D}_HBD_Hv(x)\|_{\mathcal H_2(H)}^2
= & \sum_{j,k=1}^{r}\pa{\frac{\partial^2 \psi}{\partial \xi_j\partial \xi_k}(\langle x,f^*_1\rangle,\ldots,\langle x,f^*_{r}\rangle)}^2[Bi^*f^*_j,i^*f^*_k]_H^2.
\end{align*}
Setting $d_{jk}(x):=\frac{\partial^2 \psi}{\partial \xi_j\partial \xi_k}(\langle x,f^*_1\rangle,\ldots,\langle x,f^*_{r}\rangle)$, taking into account the properties of $B$ (see Lemma \ref{propr_B}) and the fact that $d_{jk}=d_{kj}$ for $j,k=1,\ldots,r$, we get
\begin{align*}
d_{jk}[Bi^*f^*_j,i^*f^*_k]_H
= & \frac12d_{jk}[Bi^*f^*_j,i^*f^*_k]_H+\frac12d_{jk}[Bi^*f^*_j,i^*f^*_k]_H \\
= & \frac12d_{jk}[Bi^*f^*_j,i^*f^*_k]_H+\frac12d_{jk}[i^*f^*_j,B^*i^*f^*_k]_H \\
= & \frac12d_{jk}[Bi^*f^*_j,i^*f^*_k]_H+\frac12d_{jk}[B^*i^*f^*_j,i^*f^*_k]_H \\
= & -\frac12d_{jk}[i^*f^*_j,i^*f^*_k]_H.
\end{align*}
Hence, for every $x\in X$, it holds $4\|\mathbf{D}_HBD_Hv\|_{\mathcal H_2(H)}^2
= \|D^2_Hv\|_{\mathcal H_2(H)}^2$.
\end{proof}

We are now ready to show the characterization of the domain of $L$ we anticipated in the introduction.

\begin{thm}\label{main_thm}
Assume Hypotheses \ref{ipo_2} hold true. If $\mathcal{U}(X,\mu_\infty)$ is the space introduced in Theorem \ref{thm:max_reg}, then $D(L)$, endowed with the graph norm, coincide with $\mathcal{U}(X,\mu_\infty)$, up to an equivalent renorming.
\end{thm}

\begin{proof}
The first inclusion was already proved in Theorem \ref{emb1}, so we just need to show the second embedding.
Let $u\in \mathcal{U}(X,\mu_\infty)$. By \eqref{cerniera} and Lemma \ref{lemm:incl_div} it is enough to show that $BD_Hu$ belongs to $\mathcal U(X,\mu_\infty;H)$. Let $(u_m)_{m\in\N}$ be a sequence in $C_b^2(X)\cap W_H^{2,2}(X,\mu_\infty)\cap W^{1,2}_{A_\infty}(X,\mu_\infty)$ converging to $u$ with respect to $\norm{\cdot}_{\mathcal{U}(X,\mu_\infty)}$. We claim that for any $m\in\N$ the map $x\mapsto \mathbf{D}_HBD_Hu_m(x)$ belongs to $L^2(X,\mu_\infty;\mathcal H_2(H))$ and for every $\ell,m\in\N$ 
\begin{align*}
\| \mathbf{D}_HBD_Hu_m- \mathbf{D}_HBD_Hu_\ell\|_{L^2(X,\mu_\infty;\mathcal H_2(H))}^2
= \frac14\|D^2_Hu_m-D_H^2u_\ell\|_{L^2(X,\mu_\infty;\mathcal H_2(H))}^2.
\end{align*}

To prove the claim, fix $m\in\N$ and consider a sequence $(u_m^n)_{n\in\N}$ in $\fcon_{b,\Theta}^2(X)$ converging to $u_m$ in $W^{2,2}_H(X,\mu_\infty)$ as $n$ goes to infinity.
For any $m,n\in\N$ there exist $s_{m,n}\in\N$ and $\varphi_m^n\in C_b^1(\R^{s_{m,n}})$ such that $u^n_m(x)=\varphi^n_m(\langle x,f^*_1\rangle,\ldots,\langle x,f^*_{s_{m,n}}\rangle)$ for any $x\in X$. It holds
\begin{align*}
BD_Hu^n_m(x)=\sum_{j=1}^{s_{m,n}}\frac{\partial \varphi^n_m}{\partial \xi_j}(\langle x,f^*_1\rangle,\ldots,\langle x,f^*_{s_{m,n}}\rangle)Bi^*f^*_j, 
\end{align*}
and
\begin{align*}
\mathbf{D}_HBD_Hu^n_m(x)
= \sum_{j,k=1}^{s_{m,n}}\frac{\partial^2 \varphi^n_m}{\partial \xi_j\partial \xi_k}(\langle x,f^*_1\rangle,\ldots,\langle x,f^*_{s_{m,n}}\rangle)[(Bi^*f^*_j)\otimes f^*_k].
\end{align*}
Thanks to Lemma \ref{domV^*} and \eqref{Mari} it follows that $BD_Hu^n_m\in \mathcal U(X,\mu_\infty;H)$ for any $m,n\in\N$. 

Letting $v=u_m^n-u_m^\ell$ in \eqref{stima_D_H_vett_D2} we get for every $\ell,n,m\in\N$
\begin{align*}
\|\mathbf{D}_HBD_Hu^n_m-\mathbf{D}_HBD_Hu_m^{\ell}\|_{L^2(X,\mu_\infty;\mathcal H_2(H))}^2
= \frac14\|D^2_Hu^n_m-D^2_Hu_m^{\ell}\|_{L^2(X,\mu_\infty;\mathcal H_2(H))}^2.
\end{align*}
This means that, for every $m\in\N$, $(\mathbf{D}_HBD_Hu^n_m)_{n\in\N}$ is a Cauchy sequence in $L^2(X,\mu_\infty;\mathcal H_2(H))$. By the closure of $\mathbf {D}_H$ (see Proposition \ref{apo}) and the fact that $BD_Hu^n_m$ converges to $BD_Hu_m$ in $L^2(X,\mu_\infty;H)$ as $n$ goes to infinity, we can conclude that $BD_Hu_m^n$ converges to $BD_Hu_m$ in $W^{1,2}_H(X,\mu_\infty;H)$ as $n$ goes to infinity. 


We conclude the proof showing that $BD_Hu_m,BD_Hu\in D(\mathbb V^*)$ for any $m\in\N$ and $\mathbb V^*(BD_Hu_m)$ converges to $\mathbb V^*(BD_Hf)$ in $L^2(X,\mu_\infty;H_\infty)$ as $m$ goes to infinity. Since $u_m$ converges to $u$ in $W^{1,2}_{A_\infty}(X,\mu_\infty)$ as $m$ goes to infinity, by \eqref{iden_V*2}, it follows that 
\[D_{A_\infty}u_m=\mathbb V^*(BD_Hu_m),\] 
and $(BD_Hu_m)_{m\in\N}\subseteq D(\mathbb V^*)$. Moreover $\mathbb V^*(BD_Hu_m)$ converges to $D_{A_\infty}u$ in $L^2(X,\mu_\infty;H_\infty)$ as $m$ goes to infinty. For any $G\in D(\mathbb V)$ we have
\begin{align*}
\int_X[BD_Hu,\mathbb VG]_Hd\mu_\infty
= & \lim_{m\rightarrow+\infty}\int_X[BD_Hu_m,\mathbb VG]_Hd\mu_\infty
= \lim_{m\rightarrow+\infty}\int_X[\mathbb V^*(BD_Hu_m),G]_{H_\infty}d\mu_\infty \\
= & \int_X[D_{A_\infty}u,G]_{H_\infty}d\mu_\infty,
\end{align*}
which implies that $BD_Hu\in D(\mathbb V^*)$ and $\mathbb V^*(BD_Hu)=D_{A_\infty}u$.
\end{proof}

\section{An example}\label{sect_example}
Let $X=L^2([0,1],dx)$ and for every $n\in \N$  we let 
\begin{align}\label{base_laplace}
e_{n}(\xi):=\sqrt{\frac{2}{\pi}}\sin(\pi n\xi),\qquad \xi\in[0,1].
\end{align}
Let $A$ be the realization of the of the Laplace operator $\Delta_{\xi}$ with Dirichlet boundary conditions
\begin{align*}
Ax=\Delta_{\xi}x,\qquad  x\in H^2([0,1],dx)\cap H^1_0([0,1],dx)=:D(A).
\end{align*}
$A$ is self-adjoint and the family $\{e_{n}\,|\,n\in\N\}$, introduced in \eqref{base_laplace}, is an orthonormal basis of $X$ made of eigenvectors of $A$. More precisely we have
\[Ae_{n}=-(\pi n)^2e_{n},\qquad n\in \N.\]
It is a known fact that $A$ generates a strongly continuous semigroup of contractions on $X$ (see \cite[Chapter 4]{DaPra04}). Consider the linear operator $Q:X\ra X$ defined as
\begin{align*}
Qx=(q_1x_{1}+q_2x_{2})e_{1}+(q_2x_{1}+q_3x_{2})e_{2}+\sum_{i=3}^{+\infty}x_{i}e_{i},
\end{align*}
whenever $x=\sum_{i=1}^{+\infty}x_{i}e_{i}$, for some $q_1,q_2,q_3\in(0,+\infty)$ such that $q_1q_3-q_2^2>0$ and
\begin{align}
\label{esempio_cond}
\frac{9}{25}\frac{q_2^{2}(q_1q_3+q_2^{2})}{(q_1q_3-q_2^{2})^2}<1.
\end{align}
As we are going to see, condition \eqref{esempio_cond} implies that Hypotheses \ref{ipo_2} is fulfilled. Further, it is not hard to see that a sufficient condition for \eqref{esempio_cond} to hold is that $3q_2^{2}\leq q_1q_3$.
Straightforward computations give that $Q$ and $A$ satisfy Hypotheses \ref{ipo_1} and that $Q$ and $A$ do not commute. 

By the very definition of $Q$, it is easy to see that the RKHS $H$ associated to $Q$ in $X$ is simply $L^2([0,1],dx)$ up to an equivalent renorming. Furthermore
\begin{align*}
Q_\infty x &=\int_0^{+\infty}e^{sA}Qe^{sA}xds\\
&=\pa{q_1x_{1}\int_0^{+\infty}e^{-2\pi^2 s}ds+q_2x_{2}\int_0^{+\infty}e^{-5\pi^2 s}ds} e_{1}\\
&\quad +\pa{q_2x_{1} \int_0^{+\infty}e^{-5\pi^2 s}ds+q_3x_{2}\int_0^{+\infty}e^{-8\pi^2 s}ds} e_{2}+\sum_{i=3}^{+\infty}\pa{x_{i}\int_0^{+\infty}e^{-2i^2\pi^2 s}ds}e_{i}\\
&=\pa{\frac{q_1x_{1}}{2\pi^2}+\frac{q_2x_{2}}{5\pi^2}}e_{1}+\pa{\frac{q_2x_{1}}{5\pi^2}+\frac{q_3x_{2}}{8\pi^2}}e_{2}+\sum_{i=3}^{+\infty}\frac{x_{i}}{i^2\pi^2}e_{i}.
\end{align*}
This computations gives that Hypotheses \ref{portafoglio} hold true.

To show that Hypotheses \ref{ipo_RKH} hold true just observe that, by the equivalence of the norms of $H$ and $X$, it is enough to show that 
\begin{align}\label{polpette}
|Q_\infty Ax|_X\leq K|x|_X
\end{align}
for some $K>0$ and any $x\in D(A)$. It is easy to see that \eqref{polpette} holds true. Indeed since
\begin{align*}
Q_\infty Ax &=-\pa{\frac{q_1x_{1}}{2}+\frac{4q_2x_{2}}{5}}e_{1}-\pa{\frac{q_3x_{2}}{2}+\frac{q_2x_{1}}{5}}e_{2}-\sum_{i=3}^{+\infty}x_{i}e_{i},
\end{align*}
we have, repeatedly using the Young inequality,
\begin{align*}
|Q_\infty A x|^2_X &\leq \pa{\frac{q_1^2}{4}+\frac{4q_1^2+q_2^2}{10}+\frac{q_2^2}{100}}x_{1}^2+\pa{\frac{q_3^2}{4}+\frac{4q_2^2+q_3^2}{10}+\frac{4 q_2b^2}{25}}x_{2}^2+\sum_{i=3}^{+\infty}x_{i}^2.
\end{align*}
This last inequality gives \eqref{polpette}, and so Hypotheses \ref{ipo_RKH} hold true. To define the operator $B$, introduced in Lemma \ref{propr_B}, observe that the operator $Q_\infty A$ is easily estendable to an operator on the whole space $X$. Indeed it is enough to let 
\[Bx:=-\pa{\frac{q_1x_{1}}{2}+\frac{4q_2x_{2}}{5}}e_{1}-\pa{\frac{q_3x_{2}}{2}+\frac{q_2x_{1}}{5}}e_{2}-\sum_{i=3}^{+\infty}x_{i}e_{i}.\]

We now show that Hypotheses \ref{ipo_2} hold true. Observe that we have, for any $m\in\N$,
\begin{align*}
\mathscr Q_m= & \left(
\begin{matrix}
q_1 & q_2 & \vline & 0 \\
q_2 & q_3 & \vline & 0\\
\hline 
0 & 0 & \vline &  I_{m-2} 
\end{matrix}
\right),\\
\mathscr B_m= & \left(
\begin{matrix}
- q_1/2 & -4q_2/5 & \vline & 0 \\
-q_2/5 & -q_3/2  & \vline & 0 \\
\hline 
0 & 0 & \vline &  I_{m-2} 
\end{matrix}
\right),\\
\mathscr B_m-\mathscr{B}_m^*= & \left(
\begin{matrix}
 0 & -3q_2/5 & \vline & 0 \\
3  q_2/5 &  0 & \vline & 0 \\
\hline 
0 & 0 & \vline &  \mathbf{0}_{m-2} 
\end{matrix}
\right).
\end{align*}
Now let $C=(c_{ij})_{i,j=1}^m$ be any $m\times m$ symmetric matrix with real entries. We have
\begin{align*}
& {\rm Tr}_{\R^m}[(\mathscr{B}_m-\mathscr{B}_m^*)C(\mathscr{B}_m-\mathscr{B}_m^*) C]=  \frac{18q_2^2}{25}(c_{12}^2-c_{11}c_{22}),
\end{align*}
and
\begin{align*}
 {\rm Tr}_{\R^m}[\mathscr{Q}_m C \mathscr{Q}_m C]
 = &   (q_1c_{11}+q_2c_{12})^2+2(q_1c_{12}+q_2c_{22})(q_2c_{11}+q_3c_{12}) +(q_2c_{12}+q_3c_{22})^2+\sum_{i=3}^{m}c_{ii}^2 \\
= &q_1^2c_{11}^2+(2q_1q_3+2q_2^2)c_{12}^2+q_3^2c_{22}^2+4q_1q_2c_{11}c_{12}+4q_2q_3 c_{12}c_{22}+2q_2^2c_{11}c_{22} \\
&+\sum_{i=3}^\infty c_{ii}^2.
\end{align*}
Let us notice that if $\det C\leq 0$ then $ {\rm Tr}_{\R^m}[\mathscr{Q}_m C \mathscr{Q}_m C]\geq0$, and so Hypotheses \ref{ipo_2} is verified. Let us consider the case $\det C>0$. This implies that $c_{11}c_{22}>0$. Let us set
\begin{align*}
X:=\frac{2q_1q_2}{(2q_1q_3+2q_2^2)^{1/2}}c_{11},  \quad Y:=\frac{2q_2q_3}{(2q_1q_3+2q_2^2)^{1/2}}c_{22}, \quad Z:=(2q_1q_3+2q_2^2)^{1/2}c_{12}.
\end{align*}
Then,
\begin{align*}
 {\rm Tr}_{\R^m}[\mathscr{Q}_m C \mathscr{Q}_m C]
 \geq & (X+Y+Z)^2+q_1^2c_{11}^2\left(1-\frac{2q_2^2}{q_1q_3+q_2^2}\right)+q_3^2c_{22}^2\left(1-\frac{2q_2^2}{q_1q_3+q_2^2}\right) \\
&\quad -\frac{4q_1q_3q_2^2}{q_1q_3+q_2^2}c_{11}c_{22}+2q_2^2c_{11}c_{22} \\
\geq & q_1^2c_{11}^2\frac{q_1q_3-q_2^2}{q_1q_3+q_2^2}+q_3^2c_{22}^2\frac{q_1q_3-q_2^2}{q_1q_3+q_2^2}+2q_2^2c_{11}c_{22}\frac{-q_1q_3+q_2^2}{q_1q_3+q_2^2} \\
= & \frac{q_1q_3-q_2^2}{q_1q_3+q_2^2}\left(q_1^2c_{11}^2+q_3^2c_{22}^2-2q_2^2c_{11}c_{22}\right) \\
= & \frac{q_1q_3-q_2^2}{q_1q_3+q_2^2}\left((q_1c_{11}-q_3c_{22})^2+2q_1q_3c_{11}c_{22}-2q_2^2c_{11}c_{22}\right) \\
\geq & 2c_{11}c_{22}\frac{(q_1q_3-q_2^2)^2}{q_1q_3+q_2^2}=  2c_{11}c_{22}q_2^2\frac{(q_1q_3q_2^{-2}-1)^2}{q_1q_3q_2^{-2}+1}.
\end{align*}
The above computations show that \eqref{cond_traccia_m} is satisfied if there exists $\nu\in[0,1)$ such that
\begin{align}\label{fine?}
\frac{18q_2^2}{25}(c_{12}^2-c_{11}c_{22})\geq -2\nu c_{11}c_{22}q_2^2\frac{(q_1q_3q_2^{-2}-1)^2}{q_1q_3q_2^{-2}+1}.
\end{align}
A sufficient condition for \eqref{fine?} to hold is
\begin{align*}
-\frac{18q_2^2}{25}c_{11}c_{22}\geq -2\nu c_{11}c_{22}q_2^2\frac{(q_1q_3q_2^{-2}-1)^2}{q_1q_3q_2^{-2}+1}.
\end{align*}
We recall that $c_{11}c_{22}>0$. Hence,
\begin{align*}
\frac{18}{25}\leq 2\nu \frac{(q_1q_3q_2^{-2}-1)^2}{q_1q_3q_2^{-2}+1},
\end{align*}
i.e., 
\begin{align}
\label{esempio_stima_finale}
\nu\geq \frac{9}{25} \frac{q_1q_3q_2^{-2}+1}{(q_1q_3q_2^{-2}-1)^2}.
\end{align}
By \eqref{esempio_cond} it follows that the right-hand side of \eqref{esempio_stima_finale} is smaller than $1$, and so it is enough to take
\begin{align*}
\nu=\frac{9}{25} \frac{q_1q_3q_2^{-2}+1}{(q_1q_3q_2^{-2}-1)^2}.
\end{align*}
So all the results of the paper can be applied, in particular the characterization of the domain (Theorem \ref{main_thm}) of the operator $L$ associated to the quadratic form, defined for $u,v\in W^{1,2}_H(X,\mu_\infty)$,
\begin{align*}
(u,v)\mapsto-\int_X [BD_Hu,  D_Hv]_H d\mu_\infty,  
\end{align*}
holds true.

\appendix

\section{}\label{app_A}

This appendix will be dedicated to state and prove a couple of results which we have used throughout the paper, but for which we where unable to find an appropriate reference in the literature. 

Before stating the first lemma we need to recall some definitions. Whenever $(A,\preceq)$ is a directed set (see \cite[Definition 11.1]{Wil}), we say that a function $\varphi:A\ra\N$ is increasing and cofinal if $\varphi(a_1)\leq \varphi(a_2)$, whenever $a_1\preceq a_2$, and for each $n\in\N$, there exists $a\in A$ such that $n\leq \varphi(a)$.
Throughout the paper we will use the following lemma.

\begin{lemma}\label{lemma_top}
Let $Y$ be a separable Banach space and let $Y^*$ be its topological dual. We denote by $\norm{\cdot}_{Y}$ and $\norm{\cdot}_{Y^*}$ their respective norms. Let $C$ be a weak-star dense subset of $Y^*$. For every $y^*\in Y^*$ there exists a sequence $(y^*_k)_{k\in\N}\subseteq C$ which converges weakly-star to $y^*$ in $Y^*$, as $k$ goes to infinity.
\end{lemma}

\begin{proof}
Let $y^*\in Y^*$. It is known that there exists a directed set $(D,\preceq)$ and a net $(y_d^*)_{d\in D}\subseteq C$ such that 
\begin{align*}
\text{weak-star-}\lim_{d\in D}y^*_{d}=y^*.
\end{align*}
By the Banach--Steinhaus theorem (see \cite[Theorem 3.88]{FAB1}), it exists $K>0$ such that 
\begin{align*}
\sup_{d\in D}\{\|y^*_{d}\|_{Y^*},\|y^*\|_{Y^*}\}\leq K.
\end{align*}
By \cite[Proposition 3.103]{FAB1}, the ball $\mathscr{B}_K(0)$ of $Y^*$ with center the origin and radius $K$ endowed with the weak-star topolgy is metrizable. Now a standard argument allows us to extract a sequence $(y_k^*)_{k\in\N}$ from the net $(y_d^*)_{d\in D}$ still converging to $y^*$, as $k$ goes to infinity.
\end{proof}

The next result is a linear algebra lemma which has been useful in some computations throughout the paper. 
\begin{lemma}
Let $n\in\N$ and let $H,M$ be two $n\times n$-matrices with real entries. If $M$ is symmetric and $H+H^*=-M$, then, for any $n\times n$-symmetric matrix $C$ it holds
\begin{align}
\label{prop_nonsymm_matrix}
4{\rm Tr}_{\R^n}[HCHC]
={\rm Tr}_{\R^n}[MCMC]+{\rm Tr}_{\R^n}[(H-H^*)C(H-H^*)C].
\end{align}
\end{lemma}

\begin{proof}
Identity \eqref{prop_nonsymm_matrix} follows by some straightforward calculations. Indeed
\begin{align*}
&{\rm Tr}_{\R^n}[MCMC]+{\rm Tr}_{\R^n}[(H-H^*)C(H-H^*)C]\\
&\quad\quad={\rm Tr}_{\R^n}[(H+H^*)C(H+H^*)C]+{\rm Tr}_{\R^n}[(H-H^*)C(H-H^*)C]\\
&\quad\quad={\rm Tr}_{\R^n}[HCHC]+{\rm Tr}_{\R^n}[HCH^*C]+{\rm Tr}_{\R^n}[H^*CHC]+{\rm Tr}_{\R^n}[H^*CH^*C]\\
&\quad\quad\phantom{{\rm Tr}_{\R^n}[HCHC]}+{\rm Tr}_{\R^n}[HCHC]-{\rm Tr}_{\R^n}[HCH^*C]-{\rm Tr}_{\R^n}[H^*CHC]+{\rm Tr}_{\R^n}[H^*CH^*C]\\
&\quad\quad= 2{\rm Tr}_{\R^n}[HCHC]+2{\rm Tr}_{\R^n}[H^*CH^*C]=4{\rm Tr}_{\R^n}[HCHC].\qedhere
\end{align*}
\end{proof}

\end{document}